\documentclass[12pt]{amsart}
\usepackage{}

\usepackage{amsmath}
\usepackage{amsfonts}
\usepackage{amssymb}
\usepackage[all]{xy}           

\usepackage{bbding}
\usepackage{txfonts}
\usepackage{amscd}

\usepackage[shortlabels]{enumitem}
\usepackage{ifpdf}
\ifpdf
  \usepackage[colorlinks,final,
  hyperindex]{hyperref}
\else
  \usepackage[colorlinks,final,
  hyperindex]{hyperref}
\fi
\usepackage{tikz}
\usepackage[active]{srcltx}

\makeatletter

\numberwithin{equation}{section}

\newtheorem{thm}{Theorem}[section]
\newtheorem{lem}[thm]{Lemma}
\newtheorem{cor}[thm]{Corollary}
\newtheorem{pro}[thm]{Proposition}
\newtheorem{ex}[thm]{Example}
\newtheorem{rmk}[thm]{Remark}
\newtheorem{defi}[thm]{Definition}

\newcommand {\emptycomment}[1]{}

\newcommand{\la}{\mathcal{A}}

\newcommand{\gl}{\mathfrak {gl}}
\newcommand{\kl}{\mathfrak l}
\newcommand{\kr}{\mathfrak r}
\newcommand{\kt}{\mathfrak t}

\newcommand{\bz}{\mathbb Z}

\newcommand{\fl}{\mathbf l}
\newcommand{\fr}{\mathbf r}

\newcommand{\Img}{\mathrm{Im}}

\newcommand{\Hom}{\mathrm{Hom}}

\def\id{\mathop {\fam0 id}\nolimits}

\begin{document}

\title[Anti-Leibniz bialgebras and anti-Leibniz Yang-Baxter equation]
{Quasi-triangular, triangular, factorizable anti-Leibniz bialgebras
and anti-Leibniz Yang-Baxter equation}

\author{Bo Hou}
\address{School of Mathematics and Statistics, Henan University, Kaifeng 475004,
China}
\email{bohou1981@163.com}
%
\author{Zhanpeng Cui}
\address{School of Mathematics and Statistics, Henan University, Kaifeng 475004, China}
\email{czp15824833068@163.com}


\begin{abstract}
We introduce the notion of an anti-Leibniz bialgebra which is equivalent to a
Manin triple of anti-Leibniz algebras, is equivalent to a matched pair of
anti-Leibniz algebras. The study of some special anti-Leibniz bialgebras leads
to the introduction of the anti-Leibniz Yang-Baxter equation in an anti-Leibniz algebra.
A symmetric (or an invariant) solution of the anti-Leibniz Yang-Baxter equation
gives an anti-Leibniz bialgebra. The notion of a relative Rota-Baxter operator of an
anti-Leibniz algebra is introduced to construct symmetric solutions of the
anti-Leibniz Yang-Baxter equation. Moreover, we introduce the notions of factorizable
anti-Leibniz bialgebras and skew-symmetric Rota-Baxter anti-Leibniz algebras,
and show that a factorizable anti-Leibniz bialgebra leads to a factorization of the
underlying anti-Leibniz algebra. There is a one-to-one correspondence between
factorizable anti-Leibniz bialgebras and skew-quadratic Rota-Baxter anti-Leibniz algebras.
Finally, we constrict anti-Leibniz bialgebras form Leibniz bialgebras by the
tensor product and constrict infinite-dimensional anti-Leibniz bialgebras form
finite-dimensional anti-Leibniz bialgebras by the completed tensor product.
\end{abstract}

\keywords{anti-Leibniz algebra, anti-Leibniz bialgebra, Yang-Baxter equation,
relative Rota-Baxter operator, factorizable anti-Leibniz bialgebras, infinite-dimensional
anti-Leibniz bialgebras.}
\subjclass[2010]{16T10, 17A30, 16T25, 17B40.}

\maketitle

\tableofcontents 



\vspace{-4mm}

\section{Introduction}\label{sec:intr}

A bialgebra structure on a given algebraic structure is obtained as a coalgebra
structure together which gives the same algebraic structure on the dual space with
a set of compatibility conditions between the multiplications and comultiplications.
One of the most famous examples of bialgebras is the Lie bialgebra \cite{Dri} and
more importantly there have been a lot of bialgebra theories for other algebra
structures that essentially follow the approach of Lie bialgebras such as
antisymmetric infinitesimal bialgebras \cite{Agu,Bai1,BGN}, left-symmetric bialgebras
\cite{Bai}, Jordan bialgebras \cite{Zhe}, Poisson bialgebras \cite{NB},
transposed Poisson bialgebras \cite{LB}, Leibniz bialgebra \cite{TS}
and mock-Lie bialgebras \cite{BCHM}.
In this paper, we give a bialgebra theory for anti-Leibniz algebras.

Leibniz algebra is a non-commutative analogue of Lie algebra, which was first considered
in \cite{Lod} by the study of the periodicity in algebraic K-theory.
A mock-Lie algebra (also called Jacobi-Jordan algebra, Lie-Jordan algebra) is a
commutative algebra satisfying the Jacobi identity, which is a special Jordan algebra.
In recent years, mock-Lie algebra has attracted the attention of many scholars
\cite{BB,BBMM,Ben,BCHM,BF,HC1,Lar,Lar1,Zus}. Recently, the notion of anti-Leibniz algebras
as a noncommutative version of mock-Lie algebras was introduced in \cite{BCM}.
In this paper, we still take a similar approach as the study on Lie bialgebras.
The notion of anti-Leibniz bialgebra is thus introduced as an equivalent structure
of a Manin triple of anti-Leibniz algebras, which is interpreted in terms of
matched pairs of anti-Leibniz algebras. One reason for the usefulness of the Lie
bialgebra is that it has a coboundary theory, which leads to the construction of Lie
bialgebras from solutions of the classical Yang-Baxter equations. An important class
of quasi-triangular Lie bialgebras, factorizable Lie bialgebras are used to connect
classical $r$-matrices with certain factorization problems, and have various
applications in integrable systems \cite{RS,Sem}. Recently, the factorizable Lie
bialgebras, factorizable antisymmetric infinitesimal bialgebras and factorizable
Leibniz bialgebras have been studied in \cite{LS}, \cite{SW} and \cite{BLST}.
Here we also introduce the notion of quasi-triangular anti-Leibniz
bialgebra and the anti-Leibniz Yang-Baxter equation in an anti-Leibniz algebra.
We give the relative Rota-Baxter operators of anti-Leibniz
algebras and construct a solution of the Yang-Baxter equation in an anti-Leibniz algebra.
Moreover, the notion of factorizable anti-Leibniz bialgebra is introduced,
and a one-to-one correspondence between factorizable anti-Leibniz bialgebras
and skew-quadratic Rota-Baxter anti-Leibniz algebras of nonzero weight is given.
We know very little about specific examples of anti-Leibniz algebras and anti-Leibniz
bilgebras. Note that the tensor product of a Leibniz algebra and an anti-commutative
anti-associative algebra is an anti-Leibniz algebra, we take this result to the level
of the bialgebras, and provide some specific examples of anti-Leibniz bialgebras.
The bialgebra theory of anti-Leibniz algebras is based on finite-dimensional
anti-Leibniz algebras. For the infinite-dimensional anti-Leibniz bialgebras,
we provide a construction method using the completed tensor products.

The paper is organized as follows. In Section \ref{sec:bialg}, we recall some basic
definitions of anti-Leibniz algebras and bimodules over anti-Leibniz algebras.
We introduce the notions of matched pairs of anti-Leibniz algebras, Manin triples of
anti-Leibniz algebras and anti-Leibniz bialgebras. The equivalence between them
is interpreted. In Section \ref{sec:spec}, we consider some special classes of
anti-Leibniz bialgebras, such as triangular anti-Leibniz bialgebras and factorizable
anti-Leibniz bialgebras, and introduce the notion of the anti-Leibniz Yang-Baxter equation
in an anti-Leibniz algebra. The notion of relative Rota-Baxter operators of anti-Leibniz
algebras is introduced to provide symmetric solutions of the anti-Leibniz Yang-Baxter
equation in semi-direct product anti-Leibniz algebras and hence gives rise to
anti-Leibniz bialgebras. Moreover, we establish the factorizable theories for
anti-Leibniz bialgebras, and give a one-to-one correspondence between
factorizable anti-Leibniz bialgebras and skew-quadratic Rota-Baxter anti-Leibniz algebras.
At end of this section, we give some examples of anti-Leibniz bialgebras by
Leibniz bialgebras. In Section \ref{sec:infin}, we introduce the notions of completed
tensor product of $\bz$-graded vector spaces, completed anti-Leibniz coalgebras,
and propose a method for constructing infinite-dimensional anti-Leibniz bialgebras
using the completed tensor product of finite-dimensional anti-Leibniz bialgebras and
quadratic $\bz$-graded commutative associative algebra.

Throughout this paper, we fix $\Bbbk$ a field and characteristic zero. All the vector spaces
and algebras are of finite dimension over $\Bbbk$, and all tensor products are also
taking over $\Bbbk$.

\section{Anti-Leibniz algebras and anti-Leibniz bialgebras} \label{sec:bialg}

In this section, we recall the background on anti-Leibniz algebras and bimodules over
anti-Leibniz algebras. We introduce the notions of matched pairs of anti-Leibniz
algebras, Manin triples of anti-Leibniz algebras and anti-Leibniz bialgebras.
The equivalence between them is interpreted.

\begin{defi}[\cite{BCM}]\label{def:anti}
An {\rm anti-Leibniz algebra} over $\Bbbk$ is a $\Bbbk$-vector space $A$ with a
multiplication $\cdot: A\otimes A\rightarrow A$, such that
for any $a_{1}, a_{2}, a_{3}\in A$,
$$
a_{1}(a_{2}a_{3})+(a_{1}a_{2})a_{3}+a_{2}(a_{1}a_{3})=0,
$$
where $a_{1}a_{2}:=a_{1}\cdot a_{2}$ for simply.
\end{defi}

By the definition, it is easy to see that $(a_{1}a_{2})a_{3}=(a_{2}a_{1})a_{3}$
for any elements $a_{1}, a_{2}, a_{3}$ in an anti-Leibniz algebra $(A, \cdot)$.
Anti-Leibniz algebra is a noncommutative version of the mock-Lie algebras.
Recall that a mock-Lie algebra is a vector space $V$ together with a bilinear map
$[-,-]: V\times V\rightarrow V$, such that for any $v_{1}, v_{2}, v_{3}\in V$,
$$
[v_{1}, v_{2}]=[v_{2}, v_{1}],\qquad\qquad
[v_{1}, [v_{2}, v_{3}]]+[v_{2}, [v_{3}, v_{1}]]+[v_{3}, [v_{1}, v_{2}]]=0.
$$
Since $[v_{1}, v_{2}]=[v_{2}, v_{1}]$, the second equation can be rewrite as
$[v_{1}, [v_{2}, v_{3}]]+[[v_{1}, v_{2}], v_{3}]+[v_{2}, [v_{1}, v_{3}]]=0$
or $[v_{1}, [v_{2}, v_{3}]]+[[v_{1}, v_{2}], v_{3}]+[[v_{1}, v_{3}], v_{2}]=0$.
The anti-Leibniz algebra given in Definition \ref{def:anti} is also called a
{\it left anti-Leibniz algebra}. A {\it right anti-Leibniz algebra} $(A, \cdot)$
is given by $a_{1}(a_{2}a_{3})+(a_{1}a_{2})a_{3}+(a_{1}a_{3})a_{2}=0$, for any
$a_{1}, a_{2}, a_{3}\in A$. It is easy to see that $(A, \cdot)$ is a left anti-Leibniz
algebra if and only if its opposite algebra is right anti-Leibniz.
For the classification of anti-Leibniz algebras, we still know very little.
Here we give some simple examples.

\begin{ex}\label{ex:ant-alg}
Let vector space $A=\Bbbk\{e_{1}, e_{2}\}$. Then the non trivial anti-Leibniz
algebraic structure on $A$ is only isomorphic to the algebra given by
\begin{itemize}
\item[$\Lambda^{2}_{1}$]:\quad $e_{1}e_{1}=e_{2}$;
\item[$\Lambda^{2}_{2}$]:\quad $e_{1}e_{1}=-ae_{1}-\frac{a^{2}}{b}e_{2}$,
    $e_{1}e_{2}=e_{2}e_{1}=be_{1}+ae_{2}$, $e_{2}e_{2}=-\frac{b^{2}}{a}e_{1}-be_{2}$,
    with $0\neq a, b\in\Bbbk$.
\end{itemize}
That is to say, all $2$-dimensional anti-Leibniz algebras are commutative.

Let vector space $A=\Bbbk\{e_{1}, e_{2}, e_{3}\}$. Define $e_{1}e_{2}=e_{3}=e_{2}e_{2}$.
Then $A$ is a $3$-dimensional (noncommutative) anti-Leibniz algebra under this product.
\end{ex}

Let $(A, \cdot)$ and $(B, \cdot')$ be two anti-Leibniz algebras.
We call $(B, \cdot')$ is a subalgebra of $(A, \cdot)$ if $B$ is a subspace of $A$
and $\cdot'$ is the restriction of $\cdot$ on $B$. A linear map $f: A\rightarrow B$
is called a {\it homomorphism of anti-Leibniz algebras} if for any $a_{1}, a_{2}\in A$,
$f(a_{1}\cdot a_{2})=f(a_{1})\cdot'f(a_{2})$. A homomorphism $f$ is said to be an
isomorphism if $f$ is a bijection.
Now we introduce the definition of module over an anti-Leibniz algebra.

\begin{defi}\label{def:anti-mod}
Let $(A, \cdot)$ be an anti-Leibniz algebra, $M$ be a $\Bbbk$-vector space.
If there exists a pair of bilinear maps $\kl, \kr: A\rightarrow\gl(M)$, such that
for any $a_{1}, a_{2}\in A$ and $m\in M$,
\begin{align*}
&\kl(a_{1}a_{2})(m)+\kl(a_{1})(\kl(a_{2})(m))+\kl(a_{2})(\kl(a_{1})(m))=0,\\
&\kl(a_{1})(\kr(a_{2})(m))+\kr(a_{2})(\kl(a_{1})(m))+\kr(a_{1}a_{2})(m)=0,\\
&\kr(a_{1}a_{2})(m)+\kr(a_{2})(\kr(a_{1})(m))+\kl(a_{1})(\kr(a_{2})(m))=0,
\end{align*}
we call $(M, \kl, \kr)$ (or simply $M$) is a {\rm bimodule} over $(A, \cdot)$.
\end{defi}

Let $(A, \cdot)$ be an anti-Leibniz algebra, $(M, \kl, \kr)$
be a bimodule over $(A, \cdot)$. Then we get
$$
\kr(a_{1})(\kr(a_{2})(m))=\kr(a_{1})(\kl(a_{2})(m)),
$$
for any $a_{1}, a_{2}\in A$ and $m\in M$.
Now, let $(M, \kl, \kr)$ and $(N, \kl', \kr')$ be two bimodules over $(A, \cdot)$
and $f: M\rightarrow N$ be a linear map. Then $f$ is called a {\it morphism of bimodule}
if $f(\kl(a)(m))=\kl'(a)(f(m))$ and $f(\kr(a)(m))=\kr'(a)(f(m))$ for any $a\in A$ and
$m\in M$. $(M, \kl, \kr)$ and $(N, \kl', \kr')$ are called {\it isomorphic as bimodule}
if the morphism $f$ is an isomorphism. Moreover, if we denote $\fl_{A},\, \fr_{A}:
A\rightarrow\gl(A)$ by $\fl_{A}(a_{1})(a_{2})=a_{1}a_{2}$ and $\fr_{A}(a_{1})(a_{2})
=a_{2}a_{1}$ for any $a_{1}, a_{2}\in A$, then it is easy to see that $(A, \fl_{A},
\fr_{A})$ is a bimodule over $(A, \cdot)$, which is called the {\it regular bimodule}
over $(A, \cdot)$. By direct calculations, we can give an
equivalent condition as follows.

\begin{pro}\label{pro:rep}
Let $(A, \cdot)$ be an anti-Leibniz algebra, $M$ be a $\Bbbk$-vector space,
$\kl,\, \kr: A\rightarrow\gl(M)$ be two bilinear maps. Then $(M, \kl, \kr)$ is a
bimodule over $(A, \cdot)$ if and only if $A\oplus M$ is an anti-Leibniz algebra
under the following operation:
$$
(a_{1}, m_{1})(a_{2}, m_{2})=\big(a_{1}a_{2},\ \ \kl(a_{1})(m_{2})+\kr(a_{2})m_{1}\big),
$$
for all $a_{1}, a_{2}\in A$ and $m_{1}, m_{2}\in M$. This anti-Leibniz algebra
is called the {\rm semi-direct product} of $(A, \cdot)$ by bimodule $(M, \kl, \kr)$,
denoted by $A\ltimes M$.
\end{pro}

Let $V$ be a vector space. Denote the standard pairing between the dual space
$V^{\ast}$ and $V$ by
\begin{align*}
\langle-,-\rangle:\quad V^{\ast}\otimes V\rightarrow \Bbbk, \qquad\quad
\langle \xi,\; v \rangle:=\xi(v),
\end{align*}
for any $\xi\in V^{\ast}$ and $v\in V$. Let $V$, $W$ be two vector spaces. For a linear
map $\varphi: V\rightarrow W$, the transpose map $\varphi^{\ast}: W^{\ast}\rightarrow
V^{\ast}$ is defined by
\begin{align*}
\langle \varphi^{\ast}(\xi),\; v \rangle:=\langle\xi,\; \varphi(v)\rangle,
\end{align*}
for any $v\in V$ and $\xi\in W^{\ast}$. Let $(A, \cdot)$ be an anti-Leibniz algebra
and $V$ be a vector space. For a linear map $\psi: A\rightarrow\gl(V)$, the linear map
$\psi^{\ast}: A\rightarrow\gl(V^{\ast})$ is defined by
\begin{align*}
\langle\psi^{\ast}(a)(\xi),\; v\rangle:=\langle\xi,\; \psi(a)(v)\rangle,
\end{align*}
for any $a\in A$, $v\in V$, $\xi\in V^{\ast}$. That is, $\psi^{\ast}(a)=\psi(a)^{\ast}$
for all $a\in A$. Then, for any bimodule $(M, \kl, \kr)$ over anti-Leibniz algebra
$(A, \cdot)$, it is easy to see that,  $(M^{\ast}, \kl^{\ast}, \kl^{\ast}-\kr^{\ast})$
is again a bimodule over $(A, \cdot)$. In particular, we get $(A^{\ast}, \fl^{\ast}_{A},
\fl^{\ast}_{A}-\fr^{\ast}_{A})$ is a bimodule over $(A, \cdot)$, which is called the
{\it coregular bimodule}.

We give the definition of matched pair of anti-Leibniz algebras.

\begin{defi}\label{def:mat-pair}
Let $(A, \cdot)$ and $(B, \cdot')$ be two anti-Leibniz algebras, $(B, \kl_{A}, \kr_{A})$
and $(A, \kl_{B}, \kr_{B})$ be bimodules over $(A, \cdot)$ and $(B, \cdot')$ respectively.
If for any $a, a_{1}, a_{2}\in A$ and $b, b_{1}, b_{2}\in B$,
\begin{align}
&\kr_{A}(a)(b_{1}b_{2})+b_{1}\kr_{A}(a)(b_{2})+b_{2}\kr_{A}(a)(b_{1})
+\kr_{A}(\kl_{B}(b_{2})(a))(b_{1})+\kr_{A}(\kl_{B}(b_{1})(a))(b_{2})=0,   \label{mat1}\\
&\kl_{A}(a)(b_{1}b_{2})+\kl_{A}(a)(b_{1})b_{2}+b_{1}\kl_{A}(a)(b_{2})
+\kl_{A}(\kr_{B}(b_{1})(a))(b_{2})+\kr_{A}(\kr_{B}(b_{2})(a))(b_{1})=0,   \label{mat2}\\
&\qquad \kl_{A}(a)(b_{1})b_{2}+\kl_{A}(\kr_{B}(b_{1})(a))(b_{2})
-\kr_{A}(a)(b_{1})b_{2}-\kl_{A}(\kl_{B}(b_{1})(a))(b_{2})=0,              \label{mat3}\\
&\kr_{B}(b)(a_{1}a_{2})+a_{1}\kr_{B}(b)(a_{2})+a_{2}\kr_{B}(b)(a_{1})
+\kr_{B}(\kl_{A}(a_{2})(b))(a_{1})+\kr_{B}(\kl_{A}(a_{1})(b))(a_{2})=0,   \label{mat4}\\
&\kl_{B}(b)(a_{1}a_{2})+\kl_{B}(b)(a_{1})a_{2}+a_{1}\kl_{B}(b)(a_{2})
+\kl_{B}(\kr_{A}(a_{1})(b))(a_{2})+\kr_{B}(\kr_{A}(a_{2})(b))(a_{1})=0,   \label{mat5}\\
&\qquad\kl_{B}(b)(a_{1})a_{2}+\kl_{B}(\kr_{A}(a_{1})(b))(a_{2})
-\kr_{B}(b)(a_{1})a_{2}-\kl_{B}(\kl_{A}(a_{1})(b))(a_{2})=0,              \label{mat6}
\end{align}
where $a_{1}a_{2}:=a_{1}\cdot a_{2}$ and $b_{1}b_{2}:=b_{1}\cdot'b_{2}$,
we called $(A, B, \kl_{A}, \kr_{A}, \kl_{B}, \kr_{B})$ is a {\rm matched pair of
anti-Leibniz algebras} $(A, \cdot)$ and $(B, \cdot')$.
\end{defi}

For the matched pair of anti-Leibniz algebras, we have the following
equivalent description.

\begin{pro}\label{pro:rep-dir}
Let $(A, \cdot)$ and $(B, \cdot')$ be two anti-Leibniz algebras, $\kl_{A},\, \kr_{A}:
A\rightarrow\gl(B)$ and $\kl_{B},\, \kr_{B}: B\rightarrow\gl(A)$ be four linear maps.
Then $(A, B, \kl_{A}, \kr_{A}, \kl_{B}, \kr_{B})$ is a matched pair of
anti-Leibniz algebras $(A, \cdot)$ and $(B, \cdot')$ if and only if $A\oplus B$
under the product
\begin{align}
(a_{1}, b_{1})\ast(a_{2}, b_{2})
=\big(a_{1}a_{2}+\kl_{B}(b_{1})(a_{2})+\kr_{B}(b_{2})(a_{1}),\ \
b_{1}b_{2}+\kl_{A}(a_{1})(b_{2})+\kr_{A}(a_{2})(b_{1})\big),     \label{product}
\end{align}
for all $a_{1}, a_{2}\in A$ and $b_{1}, b_{2}\in B$, is also an anti-Leibniz algebra.
This anti-Leibniz algebra is called a {\rm crossed product} of $(A, \cdot)$ and
$(B, \cdot')$, denoted by $A\bowtie B$.
\end{pro}

\begin{proof}
It is straightforward.
\end{proof}

In particular, consider the matched pair of anti-Leibniz algebra structures on $A$
and $A^{\ast}$, we have the following corollary.

\begin{cor}\label{cor:rep-dir}
Let $(A, \cdot)$ and $(A^{\ast}, \cdot')$ be two anti-Leibniz algebras.
Then $(A, A^{\ast}, \fl_{A}^{\ast}, \fl_{A}^{\ast}-\fr_{A}^{\ast}, \fl_{A^{\ast}}^{\ast},
\fl_{A^{\ast}}^{\ast}-\fr_{A^{\ast}}^{\ast})$ is a matched pair of anti-Leibniz algebras
$(A, \cdot)$ and $(A^{\ast}, \cdot')$ if and only if for any $a_{1}, a_{2}\in A$
and $\xi\in A^{\ast}$,
\begin{align}
&(\fl_{A^{\ast}}^{\ast}-\fr_{A^{\ast}}^{\ast})(\xi)(a_{1}a_{2})
+a_{1}(\fl_{A^{\ast}}^{\ast}-\fr_{A^{\ast}}^{\ast})(\xi)(a_{2})
+a_{2}(\fl_{A^{\ast}}^{\ast}-\fr_{A^{\ast}}^{\ast})(\xi)(a_{1})     \label{matt1}\\[-1mm]
&\qquad\qquad+(\fl_{A^{\ast}}^{\ast}-\fr_{A^{\ast}}^{\ast})(\fl_{A}^{\ast}
(a_{2})(\xi))(a_{1})+(\fl_{A^{\ast}}^{\ast}-\fr_{A^{\ast}}^{\ast})
(\fl_{A}^{\ast}(a_{1})(\xi))(a_{2})=0,                             \nonumber\\
&\fl_{A^{\ast}}^{\ast}(\xi)(a_{1}a_{2})+\fl_{A^{\ast}}^{\ast}(\xi)(a_{1})a_{2}
+a_{1}\fl_{A^{\ast}}^{\ast}(\xi)(a_{2}) \label{matt2}\\[-1mm]
&\qquad\qquad+\fl_{A^{\ast}}^{\ast}((\fl_{A}^{\ast}-\fr_{A}^{\ast})(a_{1})(\xi))(a_{2})
+(\fl_{A^{\ast}}^{\ast}-\fr_{A^{\ast}}^{\ast})((\fl_{A}^{\ast}-\fr_{A}^{\ast})
(a_{2})(\xi))(a_{1})=0,                                           \nonumber\\
&\qquad\qquad\qquad\fr_{A^{\ast}}^{\ast}(\xi)(a_{1})a_{2}
-\fl_{A^{\ast}}^{\ast}(\fr_{A}^{\ast}(a_{1})(\xi))(a_{2})=0.      \label{matt3}
\end{align}
\end{cor}

\begin{proof}
Taking $\kl_{A}=\fl_{A}^{\ast}$, $\kr_{A}=\fl_{A}^{\ast}-\fr_{A}^{\ast}$,
$\kl_{B}=\fl_{A^{\ast}}^{\ast}$ and $\kr_{B}=\fl_{A^{\ast}}^{\ast}-\fr_{A^{\ast}}^{\ast}$
in Eqs. \eqref{mat4}-\eqref{mat6}, we get Eqs. \eqref{matt1}-\eqref{matt3} respectively.
Note that for any $a_{1}, a_{2}\in A$ and $\xi_{1}, \xi_{2}\in A^{\ast}$,
\begin{align*}
\langle(\fr_{A^{\ast}}^{\ast}(\xi_{1})(a_{1})a_{2},\; \xi_{2}\rangle
-\langle\fl_{A^{\ast}}^{\ast}(\fr_{A}^{\ast}(a_{1})(\xi_{1}))(a_{2}),\; \xi_{2}\rangle
=&\langle\fr_{A^{\ast}}^{\ast}(\xi_{1})(a_{1})a_{2},\; \xi_{2}\rangle
-\langle a_{2},\; \fr_{A}^{\ast}(a_{1})(\xi_{1})\xi_{2}\rangle\\
=&\langle a_{2},\; \fl_{A^{\ast}}^{\ast}(\fr_{A^{\ast}}^{\ast}(\xi_{1})(a_{1}))(\xi_{2})\rangle
-\langle a_{2},\; \fr_{A}^{\ast}(a_{1})(\xi_{1})\xi_{2}\rangle,
\end{align*}
we get Eq. \eqref{matt1} $\Leftrightarrow$ Eq. \eqref{mat3}. Similarly, if Eq. \eqref{mat3}
holds, we can deduce Eq. \eqref{matt2} $\Leftrightarrow$ Eq. \eqref{mat2} and
Eq. \eqref{matt1} $\Leftrightarrow$ Eq. \eqref{mat1}. Thus we obtain this corollary.
\end{proof}

Let $\mathfrak{B}(-,-)$ be a bilinear form on an anti-Leibniz algebra $(A, \cdot)$.
\begin{enumerate}\itemsep=0pt
\item[-] $\mathfrak{B}(-,-)$ is called {\it nondegenerate} if $\mathfrak{B}(a_{1},\;
        a_{2})=0$ for any $a_{2}\in A$, then $a_{1}=0$;
\item[-] $\mathfrak{B}(-,-)$ is called {\it skew-symmetric} if $\mathfrak{B}(a_{1},\; a_{2})
        =-\mathfrak{B}(a_{2},\; a_{1})$, for any $a_{1}, a_{2}\in A$;
\item[-] $\mathfrak{B}(-,-)$ is called {\it invariant} if $\mathfrak{B}(a_{1}a_{2},\;a_{3})
        =\mathfrak{B}(a_{1},\; a_{2}a_{3}-a_{3}a_{2})$, for any $a_{1}, a_{2}, a_{3}\in A$.
\end{enumerate}
Let $\mathfrak{B}(-,-)$ be a skew-symmetric invariant bilinear form on an anti-Leibniz
algebra $(A, \cdot)$. Then, we also have $\mathfrak{B}(a_{1}a_{2},\; a_{3})=
-\mathfrak{B}(a_{1}a_{3},\; a_{2})=\mathfrak{B}(a_{2},\; a_{1}a_{3})$,
for any $a_{1}, a_{2}, a_{3}\in A$.

\begin{lem}\label{lem:dual}
Let $(A, \cdot)$ be an anti-Leibniz algebra and $(A, \fl_{A}, \fr_{A})$ be the regular
bimodule over $(A, \cdot)$. Then $(A, \fl_{A}, \fr_{A})$ and $(A^{\ast}, \fl_{A}^{\ast},
\fl_{A}^{\ast}-\fr_{A}^{\ast})$ are isomorphic as bimodules over $(A, \cdot)$ if and only
if there exists a nondegenerate skew-symmetric invariant bilinear form $\mathfrak{B}(-,-)$
on $(A, \cdot)$.
\end{lem}

\begin{proof}
Suppose that there exists a nondegenerate skew-symmetric invariant bilinear form
$\mathfrak{B}(-,-)$ on $(A, \cdot)$. Then there exists a linear isomorphism
$\varphi: A \rightarrow A^{\ast}$ defined by $\langle\varphi(a_{1}), a_{2}\rangle=
\mathfrak{B}(a_{1}, a_{2})$. Moreover, note that, for any $a_{1}, a_{2}, a_{3}\in A$,
\begin{align*}
&\langle\varphi(\fl_{A}(a_{1})(a_{2})),\; a_{3}\rangle
=\mathfrak{B}(a_{1}a_{2},\; a_{3})=\mathfrak{B}(a_{2},\; a_{1}a_{3})
=\langle\fl^{\ast}_{A}(a_{1})(\varphi(a_{2})),\; a_{3}\rangle,\\
&\langle\varphi(\fr_{A}(a_{1})(a_{2})),\; a_{3}\rangle
=\mathfrak{B}(a_{2}a_{1},\; a_{3})=\mathfrak{B}(a_{2},\; a_{1}a_{3}-a_{3}a_{1})
=\langle(\fl^{\ast}_{A}-\fr^{\ast}_{A})(a_{1})(\varphi(a_{2})),\; a_{3}\rangle,
\end{align*}
we get $\varphi$ is an isomorphism. Conversely, in a similar way, we can get the
conclusion.
\end{proof}

Now we give the definition of Manin triple of anti-Leibniz algebras.

\begin{defi}\label{def:manin}
Let $(A, \cdot)$ be an anti-Leibniz algebra. Suppose that there is an anti-Leibniz
algebra structure $(A^{\ast}, \cdot')$ on the dual space $A^{\ast}$ of $A$ and there is
an anti-Leibniz algebra structure $\ast$ on the direct sum $A\oplus A^{\ast}$ such that
$(A, \cdot)$ and $(A^{\ast}, \cdot')$ are subalgebras of $(A\oplus A^{\ast}, \ast)$. If the
natural nondegenerate skew-symmetric bilinear form on $A\oplus A^{\ast}$ given by
\begin{align}
\mathfrak{B}_{d}((a_{1}, \xi_{1}),\; (a_{2}, \xi_{2}))
:=\langle\xi_{1}, a_{2}\rangle-\langle\xi_{2}, a_{1}\rangle, \label{biform}
\end{align}
for any $a_{1}, a_{2}\in A$ and $\xi_{1}, \xi_{2}\in A^{\ast}$, is invariant, then
$(A\oplus A^{\ast}, A, A^{\ast})$ is called a {\rm (standard) Manin triple of
anti-Leibniz algebras} associated to the standard bilinear form $\mathfrak{B}_{d}$,
and it is denoted by $((A\oplus A^{\ast}, \mathfrak{B}_{d}), A, A^{\ast})$ for simply.
\end{defi}

For the matched pair of anti-Leibniz algebras and (standard) Manin triple of
anti-Leibniz algebras, we have the following proposition.

\begin{pro}\label{pro:manin-mat}
Let $(A, \cdot)$ be an anti-Leibniz algebra. Suppose that there is an
anti-Leibniz algebra structure $(A^{\ast}, \cdot')$ on $A^{\ast}$. Then there
exists an anti-Leibniz algebra structure on the vector space $A\oplus A^{\ast}$
such that $(A\oplus A^{\ast}, A, A^{\ast})$ is a (standard) Manin triple of
anti-Leibniz algebras with respect to $\mathfrak{B}_{d}(-,-)$ defined by \eqref{biform}
if and only if $(A, A^{\ast}, \fl_{A}^{\ast}, \fl_{A}^{\ast}-\fr_{A}^{\ast},
\fl_{A^{\ast}}^{\ast}, \fl_{A^{\ast}}^{\ast}-\fr_{A^{\ast}}^{\ast})$ is a matched
pair of anti-Leibniz algebras.
\end{pro}

\begin{proof}
If $(A, A^{\ast}, \fl_{A}^{\ast}, \fl_{A}^{\ast}-\fr_{A}^{\ast}, \fl_{A^{\ast}}^{\ast},
\fl_{A^{\ast}}^{\ast}-\fr_{A^{\ast}}^{\ast})$ is a matched pair of anti-Leibniz algebras,
we have an anti-Leibniz algebra $(A\oplus A^{\ast}, \ast)$, where the product $\ast$
is given by \eqref{product}. It is easy to see that $(A, \cdot)$ and $(A^{\ast}, \cdot')$
are subalgebras of $(A\oplus A^{\ast}, \ast)$. We only need to show that
$\mathfrak{B}_{d}(-,-)$ is invariant. Indeed, for any $a_{1}, a_{2}, a_{3}\in A$ and
$\xi_{1}, \xi_{2}, \xi_{3}\in A^{\ast}$,
\begin{align*}
&\;\mathfrak{B}_{d}((a_{1}, \xi_{1})\ast(a_{2}, \xi_{2}),\; (a_{3}, \xi_{3}))\\
=&\;\langle\xi_{1}\xi_{2},\; a_{3}\rangle
+\langle\fl_{A}^{\ast}(a_{1})(\xi_{2}),\; a_{3}\rangle
+\langle\fl_{A}^{\ast}(a_{2})(\xi_{1})
-\fr_{A}^{\ast}(a_{2})(\xi_{1}),\; a_{3}\rangle\\[-1mm]
&\qquad-\langle\xi_{3},\; a_{1}a_{2}\rangle
-\langle\xi_{3},\; \fl_{A^{\ast}}^{\ast}(\xi_{1})(a_{2})\rangle
-\langle\xi_{3},\; \fl_{A^{\ast}}^{\ast}(\xi_{2})(a_{1})
+\fr_{A^{\ast}}^{\ast}(\xi_{2})(a_{1})\rangle\\
=&\;\langle\xi_{1},\; a_{2}a_{3}\rangle-\langle\xi_{1}\xi_{3},\; a_{2}\rangle
-\langle\xi_{2}\xi_{3},\; a_{1}\rangle+\langle\xi_{2},\; a_{1}a_{3}\rangle\\[-1mm]
&\qquad-\langle\xi_{1},\; a_{3}a_{2}\rangle+\langle\xi_{1}\xi_{2},\; a_{3}\rangle
+\langle\xi_{3}\xi_{2},\; a_{1}\rangle-\langle\xi_{3},\; a_{1}a_{2}\rangle\\
=&\;\mathfrak{B}_{d}((a_{1}, \xi_{1}),\; (a_{2}, \xi_{2})\ast(a_{3}, \xi_{3})-
(a_{3}, \xi_{3})\ast(a_{2}, \xi_{2})).
\end{align*}
Thus, $((A\oplus A^{\ast}, \mathfrak{B}_{d}), A, A^{\ast})$ is a (standard)
Manin triple of anti-Leibniz algebras.

Conversely, if $((A\oplus A^{\ast}, \mathfrak{B}_{d}), A, A^{\ast})$ is a (standard)
Manin triple of anti-Leibniz algebras, then $\mathfrak{B}_{d}(-,-)$ induces an
isomorphism from $A\oplus A^{\ast}$ to $A\oplus A^{\ast}$. That is, the anti-Leibniz
algebra structure on $A\oplus A^{\ast}$ just the crossed product $A\bowtie A^{\ast}$.
Thus $(A, A^{\ast}, \fl_{A}^{\ast}, \fl_{A}^{\ast}-\fr_{A}^{\ast}, \fl_{A^{\ast}}^{\ast},
\fl_{A^{\ast}}^{\ast}-\fr_{A^{\ast}}^{\ast})$ is a matched pair of anti-Leibniz algebras.
\end{proof}

Finally, we consider anti-Leibniz bialgebras. Let $A$ be a vector space and
$\Delta: A\rightarrow A\otimes A$ be a linear map. Then $(A, \Delta)$ is called
an {\it anti-Leibniz coalgebra} if
$$
(\Delta\otimes\id)\Delta+(\id\otimes\Delta)\Delta
+(\tau\otimes\id)(\id\otimes\Delta)\Delta=0,
$$
where the twisting map $\tau: A\otimes A\rightarrow A\otimes A$ is
given by $\tau(a_{1}\otimes a_{2})=a_{2}\otimes a_{1}$ for all $a_{1}, a_{2}\in A$.
Clearly, the notion of anti-Leibniz coalgebras is the dualization of the notion of
anti-Leibniz algebras, that is, $(A, \Delta)$ is an anti-Leibniz coalgebra
if and only if $(A^{\ast}, \Delta^{\ast})$ is an anti-Leibniz algebra.

\begin{defi}\label{def:bialg}
An {\rm anti-Leibniz bialgebra} is a triple $(A, \cdot, \Delta)$ containing
an anti-Leibniz algebra $(A, \cdot)$ and an anti-Leibniz coalgebra $(A, \Delta)$
such that for any $a_{1}, a_{2}\in A$,
\begin{align}
&\Delta(a_{1}a_{2})+\Big(\big(\id\otimes\,\fr_{A}(a_{2})-\fr_{A}(a_{2})\otimes\id
+\fl_{A}(a_{2})\otimes\id\big)(\id\otimes\id-\tau)\Big)\Delta(a_{1})   \label{bialg1}\\[-2mm]
&\qquad\qquad\qquad\qquad\qquad\qquad
+\big(\id\otimes\,\fl_{A}(a_{1})+\fl_{A}(a_{1})\otimes\id\big)\Delta(a_{2})=0,  \nonumber\\
&\qquad\quad\big(\fr_{A}(a_{1})\otimes\id\big)\Delta(a_{2})
-\tau\Big(\big(\fr_{A}(a_{2})\otimes\id\big)\Delta(a_{1})\Big)=0.           \label{bialg2}
\end{align}
\end{defi}

\begin{ex}\label{ex:bialg}
Consider the anti-Leibniz algebra $\Lambda^{2}_{1}=(\Bbbk\{e_{1}, e_{2}\}, \cdot)$ given
in Example \ref{ex:ant-alg}, define linear map $\Delta: \Lambda^{2}_{1}\rightarrow
\Lambda^{2}_{1}\otimes\Lambda^{2}_{1}$ by $\Delta(e_{1})=ke_{2}\otimes e_{2}$,
$k\in\Bbbk$, and $\Delta(e_{2})=0$, then we get a $2$-dimensional anti-Leibniz bialgebra.
\end{ex}

Let $(A, \cdot, \Delta)$ and $(B, \cdot', \Delta')$ be two anti-Leibniz bialgebras.
A linear map $\psi: A\rightarrow B$ is a {\it homomorphism of anti-Leibniz
bialgebras} if $\psi$ satisfies, for any $a_{1}, a_{2}\in A$, $\psi(a_{1}\cdot a_{2})
=\psi(a_{1})\cdot'\psi(a_{2})$ and $(\psi\otimes\psi)\Delta=\Delta'\psi$.

\begin{pro}\label{pro:bi-mat}
Let $(A, \cdot)$ be an anti-Leibniz algebra. Suppose that there is a linear map
$\Delta: A\rightarrow A\otimes A$ such that $(A, \Delta)$ is an anti-Leibniz coalgebra.
Then the triple $(A, \cdot, \Delta)$ is an anti-Leibniz bialgebra if and only if
$(A, A^{\ast}, \fl_{A}^{\ast}, \fl_{A}^{\ast}-\fr_{A}^{\ast}, \fl_{A^{\ast}}^{\ast},
\fl_{A^{\ast}}^{\ast}-\fr_{A^{\ast}}^{\ast})$ is a matched pair of anti-Leibniz algebras, where
$(A^{\ast}, \Delta^{\ast})$ is the dual algebra of anti-Leibniz coalgebra $(A, \Delta)$.
\end{pro}

\begin{proof}
For any $a_{1}, a_{2}\in A$ and $\xi_{1}, \xi_{2}\in A^{\ast}$, since
\begin{align*}
&\;\langle\fr_{A^{\ast}}^{\ast}(\xi_{2})(a_{1})a_{2},\; \xi_{1}\rangle
-\langle\fl_{A^{\ast}}^{\ast}(\fr_{A}^{\ast}(a_{1})(\xi_{2}))(a_{2}),\; \xi_{1}\rangle\\
=&\;\langle a_{1},\; \fr_{A}^{\ast}(a_{2})(\xi_{1})\xi_{2}\rangle
-\langle a_{2},\; \fr_{A}^{\ast}(a_{1})(\xi_{2})\xi_{1}\rangle\\
=&\;\langle(\fr_{A}(a_{2})\otimes\id)\Delta(a_{1}),\; \xi_{1}\otimes\xi_{2}\rangle
-\langle((\fr_{A}(a_{1})\otimes\id)\Delta(a_{2}),\; \xi_{2}\otimes\xi_{1}\rangle,
\end{align*}
we get Eq. \eqref{matt3} $\Leftrightarrow$ Eq. \eqref{bialg2}. Moreover, note that

\begin{align*}
&\;\langle\fl_{A^{\ast}}^{\ast}(\xi_{2})(a_{1}a_{2}),\; \xi_{1}\rangle
+\langle\fl_{A^{\ast}}^{\ast}(\xi_{2})(a_{1})a_{2},\; \xi_{1}\rangle
+\langle a_{1}\fl_{A^{\ast}}^{\ast}(\xi_{2})(a_{2}),\; \xi_{1}\rangle\\[-1mm]
&\quad+\langle\fl_{A^{\ast}}^{\ast}(\fl_{A}^{\ast}(a_{1})(\xi_{2}))(a_{2}),\; \xi_{1}\rangle
-\langle\fl_{A^{\ast}}^{\ast}(\fr_{A}^{\ast}(a_{1})(\xi_{2}))(a_{2}),\; \xi_{1}\rangle
+\langle\fl_{A^{\ast}}^{\ast}(\fl_{A}^{\ast}(a_{2})(\xi_{2}))(a_{1}),\; \xi_{1}\rangle\\[-1mm]
&\quad-\langle\fr_{A^{\ast}}^{\ast}(\fl_{A}^{\ast}(a_{2})(\xi_{2}))(a_{1}),\; \xi_{1}\rangle
-\langle\fl_{A^{\ast}}^{\ast}(\fr_{A}^{\ast}(a_{2})(\xi_{2}))(a_{1}),\; \xi_{1}\rangle
+\langle\fr_{A^{\ast}}^{\ast}(\fr_{A}^{\ast}(a_{2})(\xi_{2}))(a_{1}),\; \xi_{1}\rangle
\qquad\qquad\quad
\end{align*}
\begin{align*}
=&\;\langle\Delta(a_{1}a_{2}),\; \xi_{2}\otimes\xi_{1}\rangle
+\langle(\id\otimes\,\fr_{A}(a_{2}))\Delta(a_{1}),\; \xi_{2}\otimes\xi_{1}\rangle
+\langle(\id\otimes\,\fl_{A}(a_{1}))\Delta(a_{2}),\; \xi_{2}\otimes\xi_{1}\rangle\\[-1mm]
&\quad+\langle(\fl_{A}(a_{1})\otimes\id)\Delta(a_{2}),\; \xi_{2}\otimes\xi_{1}\rangle
-\langle(\fr_{A}(a_{1})\otimes\id)\Delta(a_{2}),\; \xi_{2}\otimes\xi_{1}\rangle
+\langle(\fl_{A}(a_{2})\otimes\id)\Delta(a_{1}),\; \xi_{2}\otimes\xi_{1}\rangle\\[-1mm]
&\quad-\langle(\id\otimes\,\fl_{A}(a_{2}))\Delta(a_{1}),\; \xi_{1}\otimes\xi_{2}\rangle
-\langle(\fr_{A}(a_{2})\otimes\id)\Delta(a_{1}),\; \xi_{2}\otimes\xi_{1}\rangle
+\langle(\id\otimes\,\fr_{A}(a_{2}))\Delta(a_{1}),\;\xi_{1}\otimes\xi_{2}\rangle,
\end{align*}
we get Eq. \eqref{matt2} holds if and only if
\begin{align*}
0=&\;\Delta(a_{1}a_{2})+(\id\otimes\,\fr_{A}(a_{2}))\Delta(a_{1})
+(\id\otimes\,\fl_{A}(a_{1}))\Delta(a_{2})+(\fl_{A}(a_{1})\otimes\id)\Delta(a_{2})
-(\fr_{A}(a_{1})\otimes\id)\Delta(a_{2})\\[-1mm]
&\quad+(\fl_{A}(a_{2})\otimes\id)\Delta(a_{1})-\tau(\id\otimes\,\fl_{A}(a_{2}))\Delta(a_{1})
-(\fr_{A}(a_{2})\otimes\id)\Delta(a_{1})+\tau(\id\otimes\,\fr_{A}(a_{2}))\Delta(a_{1})\\
=&\;\Delta(a_{1}a_{2})+(\id\otimes\,\fr_{A}(a_{2}))\Delta(a_{1})
+(\id\otimes\,\fl_{A}(a_{1}))\Delta(a_{2})+(\fl_{A}(a_{1})\otimes\id)\Delta(a_{2})
-\tau(\fr_{A}(a_{2})\otimes\id)\Delta(a_{1})\\[-1mm]
&\quad+(\fl_{A}(a_{2})\otimes\id)\Delta(a_{1})-\tau(\id\otimes\,\fl_{A}(a_{2}))\Delta(a_{1})
-(\fr_{A}(a_{2})\otimes\id)\Delta(a_{1})+\tau(\id\otimes\,\fr_{A}(a_{2}))\Delta(a_{1})\\
=&\;\Delta(a_{1}a_{2})+\Big(\big(\id\otimes\,\fr_{A}(a_{2})-\fr_{A}(a_{2})\otimes\id
+\fl_{A}(a_{2})\otimes\id\big)(\id\otimes\id-\tau)\Big)\Delta(a_{1}) \\[-1mm]
&\quad+\big(\id\otimes\,\fl_{A}(a_{1})+\fl_{A}(a_{1})\otimes\id\big)\Delta(a_{2}).
\end{align*}
That is, Eq. \eqref{matt2} $\Leftrightarrow$ Eq. \eqref{bialg1} if Eq. \eqref{matt2} holds.
Similarly, we also have Eq. \eqref{matt1} $\Leftrightarrow$ Eq. \eqref{bialg1} if Eq.
\eqref{matt2} holds. Thus, $(A, \cdot, \Delta)$ is an anti-Leibniz
bialgebra if and only if $(A, A^{\ast}, \fl_{A}^{\ast}, \fl_{A}^{\ast}-\fr_{A}^{\ast},
\fl_{A^{\ast}}^{\ast}, \fl_{A^{\ast}}^{\ast}-\fr_{A^{\ast}}^{\ast})$ is a matched
pair of anti-Leibniz algebras.
\end{proof}

Combining Propositions \ref{pro:manin-mat} and \ref{pro:bi-mat},
we have the following conclusion.

\begin{thm}\label{thm:equ}
Let $(A, \cdot)$ be an anti-Leibniz algebra. Suppose that there is a linear map
$\Delta: A\rightarrow A\otimes A$ such that $(A, \Delta)$ is an anti-Leibniz coalgebra.
Then the following conditions are equivalent:
\begin{enumerate}\itemsep=0pt
\item[$(i)$] $(A, A^{\ast}, \fl_{A}^{\ast}, \fl_{A}^{\ast}-\fr_{A}^{\ast},
     \fl_{A^{\ast}}^{\ast}, \fl_{A^{\ast}}^{\ast}-\fr_{A^{\ast}}^{\ast})$ is a matched
      pair of anti-Leibniz algebras;
\item[$(ii)$] $((A\oplus A^{\ast}, \mathfrak{B}_{d}), A, A^{\ast})$ is a (standard)
     Manin triple of anti-Leibniz algebras $(A, \cdot)$ and $(A^{\ast}, \Delta^{\ast})$;
\item[$(iii)$] $(A, \cdot, \Delta)$ is an anti-Leibniz bialgebra.
\end{enumerate}
\end{thm}

As a direct conclusion, we have

\begin{cor}\label{cor:dualbia}
Let $(A, \cdot, \Delta)$ be an anti-Leibniz bialgebra. Then $(A^{\ast}, \Delta^{\ast},
\bar{\Delta})$ is also an anti-Leibniz bialgebra, where $\bar{\Delta}: A^{\ast}\rightarrow
A^{\ast}\otimes A^{\ast}$ is given by $\langle\bar{\Delta}(\xi), a_{1}\otimes a_{2}
\rangle=\langle\xi, a_{1}a_{2}\rangle$, for any $\xi\in A^{\ast}$ and $a_{1}, a_{2}\in A$.
\end{cor}

\section{Some special anti-Leibniz bialgebras and anti-Leibniz Yang-Baxter equation}
\label{sec:spec}
In this section, we consider some special classes of anti-Leibniz bialgebras called
quasi-triangular anti-Leibniz bialgebras, triangular anti-Leibniz bialgebras,
factorizable anti-Leibniz bialgebras, and introduce the notion of the Yang-Baxter equation
in an anti-Leibniz algebra. The notion of relative Rota-Baxter operators of anti-Leibniz algebras
is introduced to provide antisymmetric solutions of the Yang-Baxter equation in semi-direct
product anti-Leibniz algebras and hence give rise to anti-Leibniz bialgebras.

\subsection{Triangular anti-Leibniz bialgebras}\label{subsec:tria}
Here, we consider a special class of anti-Leibniz bialgebras $(A, \cdot, \Delta_{r})$,
where the comultiplication $\Delta_{r}$ is given by
\begin{align}
\Delta_{r}(a):=((\fr_{A}-\fl_{A})(a)\otimes\id+\id\otimes\,\fr_{A}(a))(r), \label{cobo}
\end{align}
for any $a\in A$ and some $r\in A\otimes A$. In this case, we also call
$(A, \cdot, \Delta_{r})$ is {\rm an anti-Leibniz bialgebra induced by $r$}.
Such anti-Leibniz bialgebras are quite similar as the coboundary Lie bialgebras
for Lie algebras and the coboundary infinitesimal bialgebras for associative algebras.
Now, we consider the anti-Leibniz coalgebra $(A, \Delta_{r})$.

\begin{pro}\label{pro:coalg}
Let $(A, \cdot)$ be an anti-Leibniz algebra and $r=\sum_{i}x_{i}\otimes y_{i}\in A\otimes A$.
Define a linear map $\Delta_{r}: A\rightarrow A\otimes A$ by Eq. \eqref{cobo}.
Then $(A, \Delta_{r})$ is an anti-Leibniz coalgebra if and only if
\begin{align}
&\big((\fl_{A}-\fr_{A})(a)\otimes\id\otimes\id\big)\big([\![r, r]\!]\big)
+\Big(\id\otimes\id\otimes\,\fr_{A}(a)+\id\otimes\,(\fl_{A}-\fr_{A})(a)\otimes\id
\Big)\Big((\tau\otimes\id)\big([\![r, r]\!]\big)\Big) \label{coalg}\\[-1mm]
&-\Big(\id\otimes\id\otimes\,\fr_{A}(a)-\id\otimes\,(\fl_{A}-\fr_{A})(a)\otimes\id\Big)
\Big(\sum_{i}((\fr_{A}-\fl_{A})(x_{i})\otimes\id+\id\otimes\,\fr_{A}(x_{i}))(r-\tau(r))
\otimes y_{i}\Big)     \nonumber
\end{align}
\begin{align}
& +\sum_{j}((\fl_{A}-\fr_{A})(x_{j})\otimes\id\otimes\id)\big(\tau(((\fr_{A}-\fl_{A})(a)
\otimes\id+\id\otimes\,\fr_{A}(a))(r-\tau(r)))\otimes y_{j}\big),\qquad\qquad \nonumber
\end{align}
for any $a\in A$, where we denote by
\begin{align*}
[\![r, r]\!]:&=r_{12}r_{13}+r_{12}r_{23}-r_{23}r_{12}-r_{23}r_{13}\\
&=\sum_{i,j}x_{i}x_{j}\otimes y_{i}\otimes y_{j}+x_{i}\otimes y_{i}x_{j}\otimes y_{j}
-x_{j}\otimes x_{i}y_{j}\otimes y_{i}-x_{j}\otimes x_{i}\otimes y_{i}y_{j}.
\end{align*}
\end{pro}

\begin{proof}
First, for any $a\in A$, we have
\begin{align*}
&\;(\Delta_{r}\otimes\id)\Delta_{r}(a)+(\id\otimes\Delta_{r})\Delta_{r}(a)
+(\tau\otimes\id)(\id\otimes\Delta_{r})\Delta_{r}(a)\\
=&\; \sum_{i,j}x_{j}\otimes y_{j}x_{i}\otimes y_{i}a-x_{i}x_{j}\otimes y_{j}\otimes y_{i}a
+x_{j}x_{i}\otimes y_{j}\otimes y_{i}a-x_{j}\otimes y_{j}(ax_{i})\otimes y_{i}\\[-4mm]
&\qquad+(ax_{i})x_{j}\otimes y_{j}\otimes y_{i}-x_{j}(ax_{i})\otimes y_{j}\otimes y_{i}
+x_{j}\otimes y_{j}(x_{i}a)\otimes y_{i}-(x_{i}a)x_{j}\otimes y_{j}\otimes y_{i}\\
&\qquad+x_{j}(x_{i}a)\otimes y_{j}\otimes y_{i}+x_{i}\otimes x_{j}\otimes y_{j}(y_{i}a)
-x_{i}\otimes(y_{i}a)x_{j}\otimes y_{j}+x_{i}\otimes x_{j}(y_{i}a)\otimes y_{j}\\
&\qquad-ax_{i}\otimes x_{j}\otimes y_{j}y_{i}+ax_{i}\otimes y_{i}x_{j}\otimes y_{j}
-ax_{i}\otimes x_{j}y_{i}\otimes y_{j}+x_{i}a\otimes x_{j}\otimes y_{j}y_{i}\\
&\qquad-x_{i}a\otimes y_{i}x_{j}\otimes y_{j}+x_{i}a\otimes x_{j}y_{i}\otimes y_{j}
+x_{j}\otimes x_{i}\otimes y_{j}(y_{i}a)-(y_{i}a)x_{j}\otimes x_{i}\otimes y_{j}\\
&\qquad+x_{j}(y_{i}a)\otimes x_{i}\otimes y_{j}-x_{j}\otimes ax_{i}\otimes y_{j}y_{i}
+y_{i}x_{j}\otimes ax_{i}\otimes y_{j}-x_{j}y_{i}\otimes ax_{i}\otimes y_{j}\\
&\qquad+x_{j}\otimes x_{i}a\otimes y_{j}y_{i}-y_{i}x_{j}\otimes x_{i}a\otimes y_{j}
+x_{j}y_{i}\otimes x_{i}a\otimes y_{j}.
\end{align*}
By the anti-Leibniz identity, we get
\begin{align*}
\mathcal{A}:=&\sum_{i,j}x_{j}\otimes y_{j}x_{i}\otimes y_{i}a
-x_{i}x_{j}\otimes y_{j}\otimes y_{i}a+x_{j}x_{i}\otimes y_{j}\otimes y_{i}a
+x_{i}\otimes x_{j}\otimes y_{j}(y_{i}a)+x_{j}\otimes x_{i}\otimes y_{j}(y_{i}a)\\[-2mm]
=&(\id\otimes\id\otimes\,\fr_{A}(a))\big((\tau\otimes\id)\big([\![r, r]\!]\big)\big)
+\sum_{j}((\fr_{A}-\fl_{A})(x_{j})\otimes\id+\id\otimes\,\fr_{A}(x_{j}))(r-\tau(r))
\otimes y_{j}a,\\
\mathcal{B}:=&\sum_{i,j}-x_{j}(ax_{i})\otimes y_{j}\otimes y_{i}
-ax_{i}\otimes x_{j}\otimes y_{j}y_{i}+ax_{i}\otimes y_{i}x_{j}\otimes y_{j}
-ax_{i}\otimes x_{j}y_{i}\otimes y_{j})\\[-2mm]
=&(\fl_{A}(a)\otimes\id\otimes\id)\big([\![r, r]\!]\big)
+\sum_{i,j}(ax_{i})x_{j}\otimes y_{i}\otimes y_{j},\\
\mathcal{C}:=&\sum_{i,j}-x_{j}\otimes ax_{i}\otimes y_{j}y_{i}
+y_{i}x_{j}\otimes ax_{i}\otimes y_{j}-x_{j}y_{i}\otimes ax_{i}\otimes y_{j})\\[-2mm]
=&(\id\otimes\,\fl_{A}(a)\otimes\id)(\tau\otimes\id)\big([\![r, r]\!]\big)
-\sum_{i,j}y_{i}\otimes a(x_{i}x_{j})\otimes y_{j},\\
\mathcal{D}:=&\sum_{i,j}x_{j}(x_{i}a)\otimes y_{j}\otimes y_{i}
+x_{i}a\otimes x_{j}\otimes y_{j}y_{i}-x_{i}a\otimes y_{i}x_{j}\otimes y_{j}
+x_{i}a\otimes x_{j}y_{i}\otimes y_{j})\\[-2mm]
=&-(\fr_{A}(a)\otimes\id\otimes\id)\big([\![r, r]\!]\big)
-\sum_{i,j}x_{j}(x_{i}a)\otimes y_{i}\otimes y_{j},\\
\mathcal{E}:=&\sum_{i,j}x_{j}\otimes x_{i}a\otimes y_{j}y_{i}
-y_{i}x_{j}\otimes x_{i}a\otimes y_{j}+x_{j}y_{i}\otimes x_{i}a\otimes y_{j})\\[-2mm]
=&-(\id\otimes\fr_{A}(a)\otimes\id)\big((\tau\otimes\id)\big([\![r, r]\!]\big)\big)
+\sum_{i,j} y_{i}\otimes(x_{i}x_{j})a\otimes y_{j},
\end{align*}
Second, since $(ax_{i})x_{j}\otimes y_{j}\otimes y_{i}-(x_{i}a)x_{j}\otimes
y_{j}\otimes y_{i}=0$, we obtain
\begin{align*}
&\;(\Delta_{r}\otimes\id)\Delta_{r}(a)+(\id\otimes\Delta_{r})\Delta_{r}(a)
+(\tau\otimes\id)(\id\otimes\Delta_{r})\Delta_{r}(a)\\
=&\; \mathcal{A}+\mathcal{B}+\mathcal{C}+\mathcal{D}+\mathcal{E}
+\Big(\sum_{i,j}x_{j}(y_{i}a)\otimes x_{i}\otimes y_{j}
-(y_{i}a)x_{j}\otimes x_{i}\otimes y_{j}\\[-4mm]
&\qquad-x_{j}\otimes y_{j}(ax_{i})\otimes y_{i}+x_{j}\otimes y_{j}(x_{i}a)\otimes y_{i}
-x_{i}\otimes(y_{i}a)x_{j}\otimes y_{j}+x_{i}\otimes x_{j}(y_{i}a)\otimes y_{j}\Big)\\
=&\;\big((\fl_{A}-\fr_{A})(a)\otimes\id\otimes\id\big)\big([\![r, r]\!]\big)
+\Big(\id\otimes\id\otimes\,\fr_{A}(a)+\id\otimes\,(\fl_{A}-\fr_{A})(a)\otimes\id
\Big)\Big((\tau\otimes\id)\big([\![r, r]\!]\big)\Big)
\end{align*}
\begin{align*}
&\quad+\Big(\sum_{i,j}x_{j}(y_{i}a)\otimes x_{i}\otimes y_{j}
-(y_{i}a)x_{j}\otimes x_{i}\otimes y_{j}-x_{j}\otimes y_{j}(ax_{i})\otimes y_{i}
+x_{j}\otimes y_{j}(x_{i}a)\otimes y_{i}\\[-4mm]
&\qquad\qquad-x_{i}\otimes(y_{i}a)x_{j}\otimes y_{j}+x_{i}\otimes x_{j}(y_{i}a)\otimes y_{j}
+(ax_{i})x_{j}\otimes y_{i}\otimes y_{j}-y_{i}\otimes a(x_{i}x_{j})\otimes y_{j}\\[-1mm]
&\qquad\qquad-x_{j}(x_{i}a)\otimes y_{i}\otimes y_{j}
+y_{i}\otimes(x_{i}x_{j})a\otimes y_{j}\Big)\qquad\qquad\\[-2mm]
&\quad +\sum_{j}((\fr_{A}-\fl_{A})(x_{j})\otimes\id+\id\otimes\,\fr_{A}(x_{j}))(r-\tau(r))
\otimes y_{j}a.
\end{align*}
Finally, note that
\begin{align*}
&\;-\sum_{j}((\fl_{A}-\fr_{A})(x_{j})\otimes\id\otimes\id)\tau(((\fr_{A}-\fl_{A})(a)\otimes\id
+\id\otimes\,\fr_{A}(a))(r-\tau(r)))\otimes y_{j}\\[-2mm]
=&\;\sum_{i,j}x_{j}(x_{i}a)\otimes y_{i}\otimes y_{j}-x_{j}(y_{i}a)\otimes x_{i}\otimes y_{j}
+x_{j}y_{i}\otimes ax_{i}\otimes y_{j}-x_{j}y_{i}\otimes x_{i}a\otimes y_{j}\\[-5mm]
&\qquad-x_{j}x_{i}\otimes ay_{i}\otimes y_{j}+x_{j}x_{i}\otimes y_{i}a\otimes y_{j}
-(x_{i}a)x_{j}\otimes y_{i}\otimes y_{j}+(y_{i}a)x_{j}\otimes x_{i}\otimes y_{j}\\[-1mm]
&\qquad-y_{i}x_{j}\otimes ax_{i}\otimes y_{j}+y_{i}x_{j}\otimes x_{i}a\otimes y_{j}
+x_{i}x_{j}\otimes ay_{i}\otimes y_{j}-x_{i}x_{j}\otimes y_{i}a\otimes y_{j},
\end{align*}
and
\begin{align*}
&\;-(\id\otimes\,(\fl_{A}+\fr_{A})(a)\otimes\id)\Big(\sum_{j}((\fr_{A}-\fl_{A})(x_{j})
\otimes\id+\id\otimes\,\fr_{A}(x_{j}))(r-\tau(r))\otimes y_{j}\Big)\\[-2mm]
=&\;\sum_{i,j}y_{i}x_{j}\otimes ax_{i}\otimes y_{j}-y_{i}x_{j}\otimes x_{i}a\otimes y_{j}
-x_{i}x_{j}\otimes ay_{i}\otimes y_{j}+x_{i}x_{j}\otimes y_{i}a\otimes y_{j}\\[-5mm]
&\qquad-x_{j}y_{i}\otimes ax_{i}\otimes y_{j}+x_{j}y_{i}\otimes x_{i}a\otimes y_{j}
+x_{j}x_{i}\otimes ay_{i}\otimes y_{j}-x_{j}x_{i}\otimes y_{i}a\otimes y_{j}\\[-1mm]
&\qquad-x_{i}\otimes a(y_{i}x_{j})\otimes y_{j}+x_{i}\otimes(y_{i}x_{j})a\otimes y_{j}
+y_{i}\otimes a(x_{i}x_{j})\otimes y_{j}-y_{i}\otimes(x_{i}x_{j})a\otimes y_{j},
\end{align*}
we get
\begin{align*}
&\;(\Delta_{r}\otimes\id)\Delta_{r}(a)+(\id\otimes\Delta_{r})\Delta_{r}(a)
+(\tau\otimes\id)(\id\otimes\Delta_{r})\Delta_{r}(a)\\
=&\;\big((\fl_{A}-\fr_{A})(a)\otimes\id\otimes\id\big)\big([\![r, r]\!]\big)
+\Big(\id\otimes\id\otimes\,\fr_{A}(a)+\id\otimes\,(\fl_{A}-\fr_{A})(a)\otimes\id
\Big)\Big((\tau\otimes\id)\big([\![r, r]\!]\big)\Big)\\[-1mm]
&\; -\Big(\id\otimes\id\otimes\fr_{A}(a)-\id\otimes(\fl_{A}-\fr_{A})(a)\otimes\id\Big)
\Big(\sum_{i}((\fr_{A}-\fl_{A})(x_{i})\otimes\id+\id\otimes\,\fr_{A}(x_{i}))(r-\tau(r))
\otimes y_{i}\Big)\\[-3mm]
&\; +\sum_{j}((\fl_{A}-\fr_{A})(x_{j})\otimes\id\otimes\id)\big(\tau(((\fr_{A}-\fl_{A})(a)
\otimes\id+\id\otimes\,\fr_{A}(a))(r-\tau(r)))\otimes y_{j}\big).
\end{align*}
Thus, $(A, \Delta_{r})$ is an anti-Leibniz coalgebra if and only if Eq. \eqref{coalg} holds.
\end{proof}

Next, we consider the anti-Leibniz bialgebra $(A, \cdot, \Delta_{r})$.

\begin{pro}\label{pro:bialg}
Let $(A, \cdot)$ be an anti-Leibniz algebra and $r=\sum_{i}x_{i}\otimes y_{i}\in A\otimes A$.
Define a linear map $\Delta_{r}: A\rightarrow A\otimes A$ by Eq. \eqref{cobo}.
Then Eq. \eqref{bialg1} if and only if
\begin{align}
&\Big(((\fr_{A}-\fl_{A})(a_{2})\otimes\id)\tau+(\id\otimes\id-\tau)
(\fr_{A}(a_{2})\otimes\id)\Big)\big(((\fr_{A}-\fl_{A})(a_{1})\otimes\id
+\id\otimes\,\fr_{A}(a_{1}))(r-\tau(r))\big)          \label{bial1}\\[-1mm]
&\;+\tau(\fr_{A}(a_{1})\otimes\id)\big(((\fr_{A}-\fl_{A})(a_{2})\otimes\id
+\id\otimes\,\fr_{A}(a_{2}))(r-\tau(r))\big)=0, \nonumber
\end{align}
and Eq. \eqref{bialg2} if and only if
\begin{align}
(\fr_{A}(a_{1})\otimes\id)\big(((\fr_{A}-\fl_{A})(a_{2})\otimes\id
+\id\otimes\,\fr_{A}(a_{2}))(r-\tau(r))\big)=0,                         \label{bial2}
\end{align}
for any $a_{1}, a_{2}\in A$.
\end{pro}

\begin{proof}
First, by the definition of $\Delta_{r}$, for any $a_{1}, a_{2}\in A$, we have
\begin{align*}
&\;\Delta_{r}(a_{1}a_{2})+\Big(\big(\id\otimes\,\fr_{A}(a_{2})-\fr_{A}(a_{2})\otimes\id
+\fl_{A}(a_{2})\otimes\id\big)(\id-\tau)\Big)\Delta_{r}(a_{1})\\[-1mm]
&\qquad+\big(\id\otimes\,\fl_{A}(a_{1})+\fl_{A}(a_{1})\otimes\id\big)\Delta_{r}(a_{2})\\
=&\;\sum_{i}x_{i}(a_{1}a_{2})\otimes y_{i}-(a_{1}a_{2})x_{i}\otimes y_{i}
+x_{i}\otimes y_{i}(a_{1}a_{2})-a_{1}x_{i}\otimes y_{i}a_{2}
+x_{i}a_{1}\otimes y_{i}a_{2}\\[-4mm]
&\qquad+x_{i}\otimes(y_{i}a_{1})a_{2}+(a_{1}x_{i})a_{2}\otimes y_{i}
-(x_{i}a_{1})a_{2}\otimes y_{i}-x_{i}a_{2}\otimes y_{i}a_{1}
-a_{2}(a_{1}x_{i})\otimes y_{i}
\end{align*}
\begin{align*}
&\qquad+a_{2}(x_{i}a_{1})\otimes y_{i}+a_{2}x_{i}\otimes y_{i}a_{1}
+y_{i}\otimes(a_{1}x_{i})a_{2}-y_{i}\otimes(x_{i}a_{1})a_{2}
-y_{i}a_{1}\otimes x_{i}a_{2}\\[-1mm]
&\qquad-y_{i}a_{2}\otimes a_{1}x_{i}+y_{i}a_{2}\otimes x_{i}a_{1}
+(y_{i}a_{1})a_{2}\otimes x_{i}+a_{2}y_{i}\otimes a_{1}x_{i}
-a_{2}y_{i}\otimes x_{i}a_{1}\\[-1mm]
&\qquad-a_{2}(y_{i}a_{1})\otimes x_{i}-a_{2}x_{i}\otimes a_{1}y_{i}
+x_{i}a_{2}\otimes a_{1}y_{i}+x_{i}\otimes a_{1}(y_{i}a_{2})
-a_{1}(a_{2}x_{i})\otimes y_{i}\\[-1mm]
&\qquad+a_{1}(x_{i}a_{2})\otimes y_{i}+a_{1}x_{i}\otimes y_{i}a_{2}.
\end{align*}
Note that
\begin{align*}
&\quad\sum_{i}x_{i}\otimes y_{i}(a_{1}a_{2})+x_{i}\otimes a_{1}(y_{i}a_{2})
+x_{i}\otimes(y_{i}a_{1})a_{2}=0
=\sum_{i}y_{i}\otimes(a_{1}x_{i})a_{2}-y_{i}\otimes(x_{i}a_{1})a_{2},\\
&\quad\sum_{i}-(a_{1}a_{2})x_{i}\otimes y_{i}-a_{1}(a_{2}x_{i})\otimes y_{i}
-a_{2}(a_{1}x_{i})\otimes y_{i}=0
=\sum_{i}(a_{1}x_{i})a_{2}\otimes y_{i}-(x_{i}a_{1})a_{2}\otimes y_{i},\\
&\sum_{i}x_{i}a_{1}\otimes y_{i}a_{2}-y_{i}a_{1}\otimes x_{i}a_{2}
=-\tau\Big((\fr_{A}(a_{2})\otimes\id)\big(((\fr_{A}-\fl_{A})(a_{1})\otimes\id
+\id\otimes\,\fr_{A}(a_{1}))(r-\tau(r))\big)\Big),\\
&\sum_{i}y_{i}a_{2}\otimes x_{i}a_{1}-x_{i}a_{2}\otimes y_{i}a_{1}
=\tau\Big((\fr_{A}(a_{1})\otimes\id)\big(((\fr_{A}-\fl_{A})(a_{2})\otimes\id
+\id\otimes\,\fr_{A}(a_{2}))(r-\tau(r))\big)\Big),\\
&\qquad\qquad\qquad \sum_{i}a_{2}y_{i}\otimes a_{1}x_{i}-a_{2}x_{i}\otimes a_{1}y_{i}
+a_{2}x_{i}\otimes y_{i}a_{1}\\[-5mm]
&\qquad\qquad\qquad\qquad -a_{2}y_{i}\otimes x_{i}a_{1}+a_{2}(x_{i}a_{1})\otimes y_{i}
-a_{2}(y_{i}a_{1})\otimes x_{i}\\[-1mm]
&\qquad\qquad\quad=-(\fl_{A}(a_{2})\otimes\id)\Big(\tau\big(((\fr_{A}-\fl_{A})
(a_{1})\otimes\id+\id\otimes\,\fr_{A}(a_{1}))(r-\tau(r))\big)\Big),\\
&\qquad\qquad\sum_{i}x_{i}a_{2}\otimes a_{1}y_{i}-y_{i}a_{2}\otimes a_{1}x_{i}
+x_{i}(a_{1}a_{2})\otimes y_{i}+a_{1}(x_{i}a_{2})\otimes y_{i}
+(y_{i}a_{1})a_{2}\otimes x_{i}\\[-2mm]
&\qquad\quad=(\fr_{A}(a_{2})\otimes\id)\Big((\tau+\id\otimes\id)\big(((\fr_{A}-\fl_{A})
(a_{1})\otimes\id+\id\otimes\,\fr_{A}(a_{1}))(r-\tau(r))\big)\Big),
\end{align*}
we get Eq. \eqref{bialg1} holds if and only if Eq. \eqref{bial1} holds.
Finally, since for any $a_{1}, a_{2}\in A$,
\begin{align*}
&\;\big(\fr_{A}(a_{1})\otimes\id\big)\Delta_{r}(a_{2})
-\tau\Big(\big(\fr_{A}(a_{2})\otimes\id\big)\Delta_{r}(a_{1})\Big)\\
=&\; \sum_{i}x_{i}a_{1}\otimes y_{i}a_{2}-y_{i}a_{1}\otimes x_{i}a_{2}\\[-2mm]
=&\; (\fr_{A}(a_{1})\otimes\id)\big(((\fr_{A}-\fl_{A})(a_{2})\otimes\id
+\id\otimes\,\fr_{A}(a_{2}))(r-\tau(r))\big)
\end{align*}
we obtain this proposition.
\end{proof}

\begin{defi}\label{def:YBE}
Let $(A, \cdot)$ be an anti-Leibniz algebra and $r\in A\otimes A$. Then the equation
\begin{align}
[\![r, r]\!]:=r_{12}r_{13}+r_{12}r_{23}-r_{23}r_{12}-r_{23}r_{13}=0   \label{YBE}
\end{align}
is called the {\rm anti-Leibniz Yang-Baxter equation} (or {\rm aLYBE}) in the anti-Leibniz
algebra $(A, \cdot)$.

An element $r\in A\otimes A$ is called an {\rm anti-Leibniz $r$-matrix}
if $r$ is a solution of the {\rm aLYBE}, is called {\rm invariant}, if for any $a\in A$,
\begin{align}
((\fr_{A}-\fl_{A})(a)\otimes\id+\id\otimes\,\fr_{A}(a))(r)=0.   \label{invar}
\end{align}
\end{defi}

By Propositions \ref{pro:coalg} and \ref{pro:bialg}, we can provide sufficient
and necessary conditions for $(A, \cdot, \Delta_{r})$ to be an anti-Leibniz bialgebra.

\begin{cor}\label{cor:bia}
Let $(A, \cdot)$ be an anti-Leibniz algebra, $r=\sum_{i}x_{i}\otimes y_{i}\in A\otimes A$,
and $\Delta_{r}: A\rightarrow A\otimes A$ be linear map given by Eq. \eqref{cobo}.
Then $(A, \cdot, \Delta_{r})$ is an anti-Leibniz bialgebra if and only if Eqs.
\eqref{coalg}, \eqref{bial1} and \eqref{bial2} hold. In particular,
\begin{enumerate}\itemsep=0pt
\item[$(i)$] if $r-\tau(r)$ is invariant and $r$ is a solution of the {\rm aLYBE} in
     $(A, \cdot)$, or
\item[$(ii)$]  if $r$ is a symmetric solution of the {\rm aLYBE} in $(A, \cdot)$,
\end{enumerate}
then $(A, \cdot, \Delta_{r})$ is an anti-Leibniz bialgebra, which is called the
anti-Leibniz bialgebra induced by the solution $r$.
\end{cor}

\begin{defi}\label{def:tria}
Let $(A, \cdot)$ be an anti-Leibniz algebra and $r\in A\otimes A$. If $r$ is a solution
of the {\rm aLYBE} in $(A, \cdot)$ and $r-\tau(r)$ is invariant, the induced
anti-Leibniz bialgebra $(A, \cdot, \Delta_{r})$ is called {\rm quasi-triangular};
if $r$ is a symmetric solution of the {\rm aLYBE} in $(A, \cdot)$, the induced
anti-Leibniz bialgebra $(A, \cdot, \Delta_{r})$ is called {\rm triangular}.
\end{defi}

\begin{ex}\label{ex:YBE}
Let $(A=\Bbbk\{e_{1}, e_{2}\}, \cdot)$ be the $2$-dimensional anti-Leibniz algebra
$\Lambda^{2}_{1}$ given in Example \ref{ex:ant-alg}, i.e., $e_{1}e_{1}=e_{2}$. Then
$r=e_{1}\otimes e_{2}+e_{2}\otimes e_{1}$ is a symmetric solution of the {\rm aLYBE}
in $(A, \cdot)$. Thus we get a triangular anti-Leibniz bialgebra $(A, \cdot, \Delta_{r})$,
where the linear map $\Delta_{r}: A\rightarrow A\otimes A$ is given by
$\Delta_{r}(e_{1})=e_{2}\otimes e_{2}$, $\Delta(e_{2})=0$. This triangular anti-Leibniz
bialgebra just the $2$-dimensional anti-Leibniz bialgebra in Example \ref{ex:bialg}.
\end{ex}

For any vector space $V$, we define $\tau_{13}\in\gl(V\otimes V\otimes V)$ by
$\tau_{13}(v_{1}\otimes v_{2}\otimes v_{3})=v_{3}\otimes v_{2}\otimes v_{1}$ for any
$v_{1}, v_{2}, v_{3}\in V$. Then by direct calculations, we have

\begin{pro}\label{pro:dualYBE}
Let $(A, \cdot)$ be an anti-Leibniz algebra and $r=\sum_{i}x_{i}\otimes y_{i}\in A\otimes A$.
Then,
$$
[\![\tau(r),\; \tau(r)]\!]=-\tau_{13}[\![r,\; r]\!].
$$
That is to say, $r$ is a solution of the {\rm aLYBE} in $(A, \cdot)$ if and only if
$\tau(r)$ is a solution of the {\rm aLYBE} in $(A, \cdot)$.
\end{pro}

Thus, by Corollary \ref{cor:bia} and Proposition \ref{pro:dualYBE}, we have

\begin{cor}\label{cor:extbia}
If $(A, \cdot, \Delta_{r})$ is a quasi-triangular (resp. triangular) anti-Leibniz bialgebra,
then $(A, \cdot, \Delta_{\tau(r)})$ is also a quasi-triangular (resp. triangular)
anti-Leibniz bialgebra.
\end{cor}

For the construction of anti-Leibniz bialgebras, we have

\begin{thm}\label{thm:double-coba}
Let $(A, \cdot, \Delta)$ be an anti-Leibniz bialgebra. Then there is a canonical
anti-Leibniz bialgebra structure on $A\oplus A^{\ast}$ such that both
$\iota_{A}: A\rightarrow A\oplus A^{\ast}$, $a\mapsto (a, 0)$, and
$\iota_{A^{\ast}}: A^{\ast}\rightarrow A \oplus A^{\ast}$, $\xi\mapsto(0, \xi)$,
are homomorphisms of anti-Leibniz bialgebras, where the anti-Leibniz bialgebra
structure on $A\oplus A^{\ast}$ is $(A\oplus A^{\ast}, \ast, \Delta_{\tilde{r}})$,
$\Delta_{\tilde{r}}$ is given by Eq. \eqref{cobo} for some $\tilde{r}\in
(A\oplus A^{\ast})\otimes(A\oplus A^{\ast})$.
\end{thm}

\begin{proof}
Let $r\in A\otimes A^{\ast}\subset(A\oplus A^{\ast})\otimes(A\oplus A^{\ast})$ correspond
to the identity map $\id: A \rightarrow A$. That is, $\tilde{r}=\sum_{i}e_{i}\otimes f_{i}$,
where $\{e_{1}, e_{2}, \cdots, e_{n}\}$ is a basis of $A$ and $\{f_{1}, f_{2}, \cdots,
f_{n}\}$ is its dual basis. Since $(A, \cdot, \Delta)$ is an anti-Leibniz bialgebra,
there is an anti-Leibniz algebra structure on $A^{\ast}$, i.e.,
$\xi_{1}\xi_{2}=\Delta^{\ast}(\xi_{1}\otimes\xi_{2})$, and an anti-Leibniz
algebra structure $\ast$ on $A\oplus A^{\ast}$ which is given by
$$
(a_{1}, \xi_{1})\ast(a_{2}, \xi_{2})=\big(a_{1}a_{2}+\fl_{A^{\ast}}^{\ast}(\xi_{1})(a_{2})
+(\fl_{A^{\ast}}^{\ast}-\fr_{A^{\ast}}^{\ast})(\xi_{2})(a_{1}),\ \
\xi_{1}\xi_{2}+\fl_{A}^{\ast}(a_{1})(\xi_{2})+(\fl_{A}^{\ast}-\fr_{A}^{\ast})
(a_{2})(\xi_{1})\big),
$$
for any $a_{1}, a_{2}\in A$ and $\xi_{1}, \xi_{2}\in A^{\ast}$.
We denote $A\bowtie A^{\ast}:=(A\oplus A^{\ast}, \ast)$ and show that there is a
quasi-triangular anti-Leibniz bialgebra structure on $A\bowtie A^{\ast}$, i.e., Eqs.
\eqref{YBE} and \eqref{invar} hold. Note that for any $(e_{s}, f_{t})$,
$(e_{k}, f_{l})$, $(e_{p}, f_{q})\in A\oplus A^{\ast}$,
\begin{align*}
&\;\langle\tilde{r}_{12}\tilde{r}_{13}+\tilde{r}_{12}\tilde{r}_{23}
-\tilde{r}_{23}\tilde{r}_{12}-\tilde{r}_{23}\tilde{r}_{13},\ \
(e_{s}, f_{t})\otimes(e_{k}, f_{l})\otimes(e_{p}, f_{q})\rangle\\
=&\;\sum_{i,j}\langle e_{i}\ast e_{j}\otimes f_{i}\otimes f_{j}
+e_{i}\otimes f_{i}\ast e_{j}\otimes f_{j}-e_{j}\otimes e_{i}\ast f_{j}\otimes f_{i}
-e_{j}\otimes e_{i}\otimes f_{i}\ast f_{j},\\[-5mm]
&\qquad\qquad\qquad\qquad\qquad\qquad\qquad\qquad\qquad\qquad
(e_{s}, f_{t})\otimes(e_{k}, f_{l})\otimes(e_{p}, f_{q})\rangle\\
=&\;\sum_{i,j}\langle f_{t}, e_{i}e_{j}\rangle\langle f_{i}, e_{k}\rangle
\langle f_{j}, e_{p}\rangle+\langle f_{t}, e_{i}\rangle\langle f_{i}f_{l}, e_{j}\rangle
\langle f_{i}, e_{p}\rangle+\langle f_{t}, e_{i}\rangle\langle f_{l}, e_{j}e_{k}\rangle
\langle f_{j}, e_{p}\rangle\\[-5mm]
&\qquad-\langle f_{t}, e_{i}\rangle\langle f_{l}, e_{k}e_{j}\rangle
\langle f_{j}, e_{p}\rangle-\langle f_{t},  e_{j}\rangle\langle f_{j}f_{l}, e_{i}\rangle
\langle f_{i}, e_{p}\rangle+\langle f_{t},  e_{j}\rangle\langle f_{l}f_{j}, e_{i}\rangle
\langle f_{i}, e_{p}\rangle\\[-1mm]
&\qquad-\langle f_{t},  e_{j}\rangle\langle f_{j}, e_{i}e_{k}\rangle
\langle f_{i}, e_{p}\rangle-\langle f_{t},  e_{j}\rangle\langle f_{l}, e_{i}\rangle
\langle f_{i}f_{j}, e_{p}\rangle\\
=&\;0,
\end{align*}
we get $[\![\tilde{r}, \tilde{r}]\!]=0$. Similarly, one can check that
$((\fr_{A\oplus A^{\ast}}-\fl_{A\oplus A^{\ast}})(x)\otimes\id+\id\otimes\,\fr_{A
\oplus A^{\ast}}(x))(\tilde{r}-\tau(\tilde{r}))=0$ for any $x\in A\bowtie A^{\ast}$.
Hence, there is a quasi-triangular anti-Leibniz bialgebra structure on
$A\bowtie A^{\ast}$ by Corollary \ref{cor:bia}.

For any $e_{j}\in A$, note that
\begin{align*}
\Delta_{\tilde{r}}(e_{j})&=((\fr_{A\oplus A^{\ast}}-\fl_{A\oplus A^{\ast}})(e_{j})
\otimes\id+\id\otimes\,\fr_{A\oplus A^{\ast}}(e_{j}))\Big(\sum_{i}e_{i}
\otimes f_{i}\Big)\\[-2mm]
&=\sum_{i, k}e_{i}e_{j}\otimes f_{i}-e_{j}e_{i}\otimes f_{i}
+e_{i}\otimes\langle f_{i}f_{k}, e_{j}\rangle e_{k}
+e_{i}\otimes\langle f_{i}, e_{j}e_{k}\rangle f_{k}
-e_{i}\otimes\langle f_{i},e_{k}e_{j}\rangle f_{k}\\[-2mm]
&=\sum_{i, k}\langle f_{i}f_{k}, e_{j}\rangle e_{i}\otimes e_{k}\\[-2mm]
&=\Delta(e_{j}).
\end{align*}
Therefore, $\iota_{A}: A\rightarrow A\oplus A^{\ast}$ is a homomorphism of
anti-Leibniz bialgebras. Similarly, $\iota_{A^{\ast}}: A^{\ast}\rightarrow A
\oplus A^{\ast}$ is also a homomorphism of
anti-Leibniz bialgebras. The proof is finished.
\end{proof}

With the anti-Leibniz bialgebra structure given in the theorem above,
$(A\oplus A^{\ast}, \ast, \Delta_{\tilde{r}})$ is called the {\it double of anti-Leibniz
bialgebra $(A, \cdot, \Delta)$}.

\begin{ex}\label{ex:double}
Let $(A=\Bbbk\{e_{1}, e_{2}\}, \cdot, \Delta)$ be the $2$-dimensional anti-Leibniz
bialgebra in Example \ref{ex:bialg}, i.e., $e_{1}e_{1}=e_{2}$, $\Delta(e_{1})
=e_{2}\otimes e_{2}$. Denote $\{f_{1}, f_{2}\}$ the dual basis
of $\{e_{1}, e_{2}\}$ and $\tilde{r}=e_{1}\otimes f_{1}+e_{2}\otimes f_{2}$.
Then we get a $4$-dimensional anti-Leibniz bialgebra $(D(A)=\Bbbk\{e_{1}, e_{2},
f_{1}, f_{2}\}, \ast, \Delta_{\tilde{r}})$, where
\begin{align*}
&e_{1}\ast e_{1}=e_{2}, \qquad f_{2}\ast f_{2}=f_{1},\qquad
e_{1}\ast f_{2}=f_{1}, \qquad f_{2}\ast e_{1}=e_{2},\\
&\qquad\qquad\Delta_{\tilde{r}}(e_{1})=e_{2}\otimes e_{2}, \qquad\qquad
\Delta_{\tilde{r}}(f_{2})=f_{1}\otimes f_{1}.
\end{align*}
\end{ex}

\subsection{Relative Rota-Baxter operator and anti-Leibniz Yang-Baxter equation}
\label{subsec:o-oper}
In this subsection, we are discussing the relationship between (relative) Rota-Baxter
operators on an anti-Leibniz algebra and some special resolutions of the {\rm aLYBE},
and show that each relative Rota-Baxter operator induces an anti-Leibniz bialgebra.
First, we recall the definition of relative Rota-Baxter operator.

\begin{defi}\label{def:ooper}
Let $(A, \cdot)$ be an anti-Leibniz algebra and $(M, \kl, \kr)$ be a bimodule over
$(A, \cdot)$. A linear map $R: M\rightarrow A$ is called a {\rm relative Rota-Baxter
operator of $(A, \cdot)$ associated to $(M, \kl, \kr)$} if $R$ satisfies
$$
R(m_{1})R(m_{2})=R\Big(\kl(R(m_{1}))(m_{2})+\kr(R(m_{2}))(m_{1})\Big),
$$
for any $m_{1}, m_{2}\in M$. In particular, if the bimodule $(M, \kl, \kr)$
is the regular bimodule $(A, \fl_{A}, \fr_{A})$, the operator $R$ is called
a {\rm Rota-Baxter operator} on $(A, \cdot)$.
\end{defi}

Let $V$ be a vector space. For all $r\in V\otimes V$, there is a linear map $r^{\sharp}:
V^{\ast}\rightarrow V$ in the following way:
$$
\langle\xi_{1}\otimes\xi_{2},\; r\rangle=\langle r^{\sharp}(\xi_{1}),\; \xi_{2}\rangle ,
$$
for any $\xi_{1}, \xi_{2}\in V^{\ast}$. If $r=\sum_{i}x_{i}\otimes y_{i}\in V\otimes V$,
it is easy to see that $r^{\sharp}(\xi)=\sum_{i}\langle\xi, x_{i}\rangle y_{i}$, for any
$\xi\in V^{\ast}$. We say that $r\in V\otimes V$ is {\it nondegenerate}
if the map $r^{\sharp}$ is a linear isomorphism.

\begin{pro}\label{pro:r-map}
Let $(A, \cdot)$ be an anti-Leibniz algebra and $r=\sum_{i}x_{i}\otimes y_{i}\in A\otimes A$.
Then $r$ is a solution of the {\rm aLYBE} in $(A, \cdot)$ if and only if
for any $\xi_{1}, \xi_{2}\in A^{\ast}$,
$$
\tau(r)^{\sharp}(\xi_{1})\tau(r)^{\sharp}(\xi_{2})
=\tau(r)^{\sharp}\Big(\fl_{A}^{\ast}(r^{\sharp}(\xi_{1}))(\xi_{2})
+(\fl_{A}^{\ast}-\fr_{A}^{\ast})(\tau(r)^{\sharp}(\xi_{2}))(\xi_{1})\Big).
$$
\end{pro}

\begin{proof}
For any $\xi_{1}, \xi_{2}, \xi_{3}\in A^{\ast}$, we have
\begin{align*}
\langle\tau(r)^{\sharp}(\xi_{1})\tau(r)^{\sharp}(\xi_{2}),\; \xi_{3}\rangle
&=\langle r,\; \fr_{A}^{\ast}(\tau(r)^{\sharp}(\xi_{2}))(\xi_{3})\otimes\xi_{1}\rangle
=\sum_{i}\langle x_{i},\; \fr_{A}^{\ast}(\tau(r)^{\sharp}(\xi_{2}))(\xi_{3})\rangle
\langle y_{i}, \xi_{1}\rangle\\[-2mm]
&=\sum_{i}\langle r,\; \fl_{A}^{\ast}(x_{i})(\xi_{3})\otimes\xi_{2}\rangle
\langle y_{i}, \xi_{1}\rangle
=\sum_{i,j}\langle x_{j},\; \fl_{A}^{\ast}(x_{i})(\xi_{3})\rangle
\langle y_{j}, \xi_{2}\rangle\langle y_{i}, \xi_{1}\rangle\\[-2mm]
&=\sum_{i,j}\langle y_{i}\otimes y_{j}\otimes x_{i}x_{j},\;
\xi_{1}\otimes\xi_{2}\otimes\xi_{3}\rangle.
\end{align*}
Similarly, we have
\begin{align*}
\langle\tau(r)^{\sharp}\big(\fl_{A}^{\ast}(r^{\sharp}(\xi_{1}))
(\xi_{2})\big),\; \xi_{3}\rangle
&=\sum_{i,j}\langle x_{j}\otimes y_{j}y_{i}\otimes x_{i},\;
\xi_{1}\otimes\xi_{2}\otimes\xi_{3}\rangle,\\[-2mm]
\langle\tau(r)^{\sharp}\big(\fl_{A}^{\ast}(\tau(r)^{\sharp}
(\xi_{2}))(\xi_{1})\big),\; \xi_{3}\rangle
&=\sum_{i,j}\langle x_{j}y_{i}\otimes y_{j}\otimes x_{i},\;
\xi_{1}\otimes\xi_{2}\otimes\xi_{3}\rangle,\\[-2mm]
\langle\tau(r)^{\sharp}\big(\fr_{A}^{\ast}(\tau(r)^{\sharp}
(\xi_{2}))(\xi_{1})\big),\; \xi_{3}\rangle
&=\sum_{i,j}\langle y_{i}x_{j}\otimes y_{j}\otimes x_{i},\;
\xi_{1}\otimes\xi_{2}\otimes\xi_{3}\rangle.
\end{align*}
Hence, we have
\begin{align*}
&\;\langle\tau(r)^{\sharp}(\xi_{1})\tau(r)^{\sharp}(\xi_{2})
-\tau(r)^{\sharp}\big(\fl_{A}^{\ast}(r^{\sharp}(\xi_{1}))(\xi_{2})
+(\fl_{A}^{\ast}-\fr_{A}^{\ast})(\tau(r)^{\sharp}(\xi_{2}))(\xi_{1})\big),\; \xi_{3}\rangle\\
&=\sum_{i,j}\langle y_{i}\otimes y_{j}\otimes x_{i}x_{j}-x_{j}\otimes y_{j}y_{i}\otimes x_{i}
-x_{j}y_{i}\otimes y_{j}\otimes x_{i}+y_{i}x_{j}\otimes y_{j}\otimes x_{i},\;
\xi_{1}\otimes\xi_{2}\otimes\xi_{3}\rangle,\\[-2mm]
&=\sum_{i,j}\langle(\id\otimes\tau)\big((\tau\otimes\id)([\![r, r]\!])\big),\;
\xi_{1}\otimes\xi_{2}\otimes\xi_{3}\rangle.
\end{align*}
That is to say, this proposition hold.
\end{proof}

If $r\in A\otimes A$ is symmetric, i.e., $r=\tau(r)$, then we have

\begin{cor}\label{cor:r-map1}
Let $(A, \cdot)$ be an anti-Leibniz algebra and $r\in A\otimes A$ be symmetric.
Then $r$ is a solution of the {\rm aLYBE} in $(A, \cdot)$ if and only if $r^{\sharp}$
is a relative Rota-Baxter operator of $(A, \cdot)$ associated to the coregular
bimodule $(A^{\ast}, \fl_{A}^{\ast}, \fl_{A}^{\ast}-\fr_{A}^{\ast})$.
\end{cor}

Let $(A, \cdot, \mathfrak{B})$ be an anti-Leibniz algebra with a nondegenerate
skew-symmetric invariant bilinear form. Then under the natural bijection $\Hom(A\otimes A,
\Bbbk)\cong\Hom(A, A^{\ast})$, the bilinear form $\mathfrak{B}(-,-)$ corresponds to a
linear map $\varphi: A\rightarrow A^{\ast}$, which is given by $\langle\varphi(a_{1}),\;
a_{2}\rangle=\mathfrak{B}(a_{1}, a_{2})$, for any $a_{1}, a_{2}\in A$.
Define a $2$-tensor $r_{\mathfrak{B}}$ to be the tensor form of $\varphi^{-1}$, i.e.,
$\langle\xi_{1}\otimes\xi_{2},\; r_{\mathfrak{B}}\rangle=\langle\varphi^{-1}
(\xi_{1}),\; \xi_{2}\rangle$. Then, by direct calculations, we have

\begin{lem}\label{lem:inv-inv}
Let $(A, \cdot, \mathfrak{B})$ be an anti-Leibniz algebra with a nondegenerate
bilinear form. Then $\mathfrak{B}(-,-)$ is skew-symmetric invariant if and only if
$r_{\mathfrak{B}}$ is skew-symmetric invariant.
\end{lem}

For any $r\in A\otimes A$, define a linear map $R_{r}: A\rightarrow A$, $a\mapsto
R_{r}(a):=r^{\sharp}(\varphi(a))$, then we have the following proposition and corollary.

\begin{pro}\label{pro:yberb}
Let $(A, \cdot, \mathfrak{B})$ be an anti-Leibniz algebra with a nondegenerate
skew-symmetric invariant bilinear form and $r\in A\otimes A$. Then $r^{\sharp}$ is
a relative Rota-Baxter operator of $(A, \cdot)$ associated to $(A^{\ast},
\fl_{A}^{\ast}, \fl_{A}^{\ast}-\fr_{A}^{\ast})$ if and only if $R_{r}: A\rightarrow A$
is a Rota-Baxter operator on $(A, \cdot)$.
\end{pro}

\begin{proof}
For any $a_{1}, a_{2}\in A$, by Lemma \ref{lem:dual}, we have
$R_{r}(a_{1})R_{r}(a_{2})=r^{\sharp}(\varphi(a_{1}))r^{\sharp}(\varphi(a_{2}))$ and
\begin{align*}
R_{r}(a_{1}R_{r}(a_{2})+R_{r}(a_{1})a_{2})
=&\; r^{\sharp}\big(\varphi\big(a_{1}r^{\sharp}(\varphi(a_{2}))
+r^{\sharp}(\varphi(a_{1}))a_{2}\big)\big)\\
=&\; r^{\sharp}\Big(\fr_{A}^{\ast}(r^{\sharp}(\varphi(a_{1})))(\varphi(a_{2}))
+(\fl_{A}^{\ast}-\fr_{A}^{\ast})(r^{\sharp}(\varphi(a_{2})))(\varphi(a_{1}))\Big).
\end{align*}
Thus, $r^{\sharp}$ is a relative Rota-Baxter operator if and only if $R_{r}$
is a Rota-Baxter operator.
\end{proof}

As a direct conclusion of Corollary \ref{cor:r-map1} and Proposition \ref{pro:yberb}, we have

\begin{cor}\label{cor:yberb1}
Let $(A, \cdot, \mathfrak{B})$ be an anti-Leibniz algebra with a nondegenerate
skew-symmetric invariant bilinear form and $r\in A\otimes A$ be symmetric.
Then, $r$ is a solution of the {\rm aLYBE} in $(A, \cdot)$ if and only if
$R_{r}$ is a Rota-Baxter operator on $(A, \cdot)$.
\end{cor}

More generally, if $r\in A\otimes A$ such that $r-\tau(r)$ is invariant, we consider when
$r$ is a solution of the {\rm aLYBE} in $(A, \cdot)$.

\begin{pro}\label{pro:inv-sol1}
Let $(A, \cdot)$ be an anti-Leibniz algebra. Suppose $r\in A\otimes A$ such that
$r-\tau(r)$ is invariant. Then, $r$ is a solution of the {\rm aLYBE} in $(A, \cdot)$
if and only if for any $\xi_{1}, \xi_{2}\in A^{\ast}$,
$$
r^{\sharp}(\xi_{1})r^{\sharp}(\xi_{2})=r^{\sharp}\Big(\fl_{A}^{\ast}(r^{\sharp}(\xi_{1}))
(\xi_{2})+(\fl_{A}^{\ast}-\fr_{A}^{\ast})(r^{\sharp}(\xi_{2}))(\xi_{1})
-\fl_{A}^{\ast}((r^{\sharp}-\tau(r)^{\sharp})(\xi_{1}))(\xi_{2})\Big).
$$
\end{pro}

\begin{proof}
First, since $r-\tau(r)$ is skew-symmetric, it is easy to see that $r-\tau(r)$ is invariant
if and only if $\fl_{A}^{\ast}((r^{\sharp}-\tau(r)^{\sharp})(\xi_{1}))(\xi_{2})=
(\fl_{A}^{\ast}-\fr_{A}^{\ast})((r^{\sharp}-\tau(r)^{\sharp})(\xi_{2}))(\xi_{1})$,
for any $\xi_{1}, \xi_{2}\in A^{\ast}$. Thus, by direct calculations, we have
\begin{align*}
&\;r^{\sharp}\Big(\fl_{A}^{\ast}(r^{\sharp}(\xi_{1}))
(\xi_{2})+(\fl_{A}^{\ast}-\fr_{A}^{\ast})(r^{\sharp}(\xi_{2}))(\xi_{1})
-\fl_{A}^{\ast}((r^{\sharp}-\tau(r)^{\sharp})(\xi_{1}))(\xi_{2})\Big)\\
=&\; r^{\sharp}\Big(\fl_{A}^{\ast}(r^{\sharp}(\xi_{1}))
(\xi_{2})+(\fl_{A}^{\ast}-\fr_{A}^{\ast})(r^{\sharp}(\xi_{2}))(\xi_{1})
-(\fl_{A}^{\ast}-\fr_{A}^{\ast})((r^{\sharp}-\tau(r)^{\sharp})(\xi_{2}))(\xi_{1})\Big)\\
=&\; r^{\sharp}\Big(\fl_{A}^{\ast}(r^{\sharp}(\xi_{1}))
(\xi_{2})+(\fl_{A}^{\ast}-\fr_{A}^{\ast})(\tau(r)^{\sharp}(\xi_{2}))(\xi_{1})\Big).
\end{align*}
Second, since $r-\tau(r)$ is invariant, by Proposition \ref{pro:r-map}, $r$ is a solution
of the {\rm aLYBE} in $(A, \cdot)$ if and only if $r^{\sharp}(\xi_{1})r^{\sharp}(\xi_{2})=
r^{\sharp}\big(\fl_{A}^{\ast}(r^{\sharp}(\xi_{1}))(\xi_{2})+(\fl_{A}^{\ast}-\fr_{A}^{\ast})
(\tau(r)^{\sharp}(\xi_{2}))(\xi_{1})\big)$. Thus, we get this proposition.
\end{proof}

Let $(A, \cdot)$ be an anti-Leibniz algebra. By the Definition \ref{def:YBE}, it is
easy to see that $r\in A\otimes A$ is invariant if and only if $r^{\sharp}(\xi)a=
r^{\sharp}\big((\fl_{A}^{\ast}-\fr_{A}^{\ast})(a)(\xi)\big)$ for any $a\in A$ and
$\xi\in A^{\ast}$.

\begin{cor}\label{cor:inv-solu}
Let $(A, \cdot, \mathfrak{B})$ be an anti-Leibniz algebra with a nondegenerate
skew-symmetric invariant bilinear form and $r\in A\otimes A$ be invariant.
Then, $r$ is a solution of the {\rm aLYBE} in $(A, \cdot)$ if and only if
for any $a_{1}, a_{2}\in A$,
$$
R_{r}(a_{1})R_{r}(a_{1})=R_{r}\Big(R_{r}(a_{1})a_{2}+a_{1}R_{r}(a_{2})
-a_{1}R_{r}(a_{2})+a_{1}\tau(r)^{\sharp}(\varphi(a_{2}))\Big),
$$
where $R_{r}: A\rightarrow A$ is defined by $R_{r}(a)=r^{\sharp}(\varphi(a))$.
\end{cor}

\begin{proof}
Since $r_{\mathfrak{B}}$ is skew-symmetric invariant in this case, we get
$\varphi\big(r^{\sharp}(\xi_{1})\varphi^{-1}(\xi_{2})\big)=\fl_{A}^{\ast}(r^{\sharp}
(\xi_{1}))(\xi_{2})$ and $\varphi\big(\varphi^{-1}(\xi_{1})r^{\sharp}(\xi_{2})\big)
=(\fl_{A}^{\ast}-\fr_{A}^{\ast})(r^{\sharp}(\xi_{2}))(\xi_{1})$. For any $a_{1},
a_{2}\in A$, suppose $a_{1}=\varphi^{-1}(\xi_{1})$ and $a_{2}=\varphi^{-1}(\xi_{2})$,
$\xi_{1}, \xi_{2}\in A^{\ast}$. Then, we have
\begin{align*}
R_{r}(a_{1})R_{r}(a_{1})&=r^{\sharp}(\xi_{1})r^{\sharp}(\xi_{2}),\\
R_{r}(a_{1}R_{r}(a_{2}))&=r^{\sharp}\big(\fl_{A}^{\ast}(r^{\sharp}(\xi_{1}))(\xi_{2})\big),\\
R_{r}(R_{r}(a_{1})a_{2})&=r^{\sharp}\big((\fl_{A}^{\ast}-\fr_{A}^{\ast})(r^{\sharp}
(\xi_{2}))(\xi_{1})\big),\\
R_{r}\big(a_{1}\tau(r)^{\sharp}(\varphi(a_{2}))-a_{1}R_{r}(a_{2})\big)&=r^{\sharp}\big(
\fl_{A}^{\ast}((\tau(r)^{\sharp}-r^{\sharp})(\xi_{1}))(\xi_{2})\big).
\end{align*}
By Proposition \ref{pro:inv-sol1}, we get the conclusion.
\end{proof}

In the Lie algebra context, it is well-known that the dual description of a classical
$r$-matrix is a symplectic structure on a Lie algebra. For the anti-Leibniz algebras, we have

\begin{pro}\label{pro:ybe}
Let $(A, \cdot)$ be an anti-Leibniz algebra and $r\in A\otimes A$.
Suppose that $r$ is symmetric and nondegenerate. The linear isomorphism
$r^{\sharp}: A^{\ast}\rightarrow A$ define a bilinear form $\omega(-,-)$ on
$(A, \cdot)$ by $\omega(a_{1}, a_{2})=\langle(r^{\sharp})^{-1}(a_{1}),\, a_{2}\rangle$
for any $a_{1}, a_{2}\in A$. Then $r$ is a solution of the {\rm aLYBE} in $(A, \cdot)$
if and only if the bilinear form $\omega(-,-)$ satisfies
\begin{align}
\omega(a_{2}a_{3}, a_{1})+\omega(a_{1}a_{3}, a_{2})
-\omega(a_{3}a_{1}, a_{2})-\omega(a_{2}a_{1}, a_{3})=0, \label{symp}
\end{align}
for any $a_{1}, a_{2}, a_{3}\in A$.
\end{pro}

\begin{proof}
Let $r=\sum_{i}x_{i}\otimes y_{i}\in A\otimes A$ be symmetric. Then $r^{\sharp}(\xi)
=\sum_{i}\langle\xi, x_{i}\rangle y_{i}=\sum_{i}\langle\xi, y_{i}\rangle x_{i}$ for any
$\xi\in A^{\ast}$ and the bilinear form $\omega(-,-)$ is symmetric, i.e., $\omega(a_{1},
a_{2})=\omega(a_{2}, a_{1})$ for any $a_{1}, a_{2}\in A$. Since $r$ is nondegenerate,
for any $a_{1}, a_{2}, a_{3}\in A$, there are $\xi_{1}, \xi_{2}, \xi_{3}\in A^{\ast}$
such that $r^{\sharp}(\xi_{i})=a_{i}$ for $i=1,2,3$. Thus, we have

\begin{align*}
&\;\omega(a_{2}a_{3}, a_{1})+\omega(a_{1}a_{3}, a_{2})
-\omega(a_{3}a_{1}, a_{2})-\omega(a_{2}a_{1}, a_{3})\\
=&\;\langle r^{\sharp}(\xi_{2})r^{\sharp}(\xi_{3}),\; \xi_{1}\rangle
+\langle r^{\sharp}(\xi_{1})r^{\sharp}(\xi_{3}),\; \xi_{2}\rangle
-\langle r^{\sharp}(\xi_{3})r^{\sharp}(\xi_{1}),\; \xi_{2}\rangle
-\langle r^{\sharp}(\xi_{2})r^{\sharp}(\xi_{1}),\; \xi_{3}\rangle\\[-1mm]
=&\;\sum_{i,j}\Big(\langle\xi_{1}, x_{i}x_{j}\rangle\langle\xi_{2}, y_{i}\rangle
\langle\xi_{3}, y_{j}\rangle+\langle\xi_{1}, x_{i}\rangle
\langle\xi_{2}, y_{i}x_{j}\rangle\langle\xi_{3}, y_{j}\rangle\\[-5mm]
&\qquad\quad-\langle\xi_{1}, x_{j}\rangle
\langle\xi_{2}, x_{i}y_{j}\rangle\langle\xi_{3}, y_{i}\rangle\Big)-\langle\xi_{1},
x_{j}\rangle\langle\xi_{2}, x_{i}\rangle\langle\xi_{3}, y_{i}y_{j}\rangle\Big)\\[-1mm]
=&\;\langle\xi_{1}\otimes\xi_{2}\otimes\xi_{3},\;
r_{12}r_{13}+r_{12}r_{23}-r_{23}r_{12}-r_{23}r_{13}\rangle.
\end{align*}
That is, $\omega(a_{2}a_{3}, a_{1})+\omega(a_{1}a_{3}, a_{2})
-\omega(a_{3}a_{1}, a_{2})-\omega(a_{2}a_{1}, a_{3})=0$ for any $a_{1}, a_{2},
a_{3}\in A$ if and only if $\langle\xi_{1}\otimes\xi_{2}\otimes
\xi_{3},\; [\![r,\; r]\!]\rangle=0$ for any $\xi_{1}, \xi_{2}, \xi_{3}\in A^{\ast}$,
if and only if $[\![r,\; r]\!]=0$. The proof is finished.
\end{proof}

Next, we show that each relative Rota-Baxter operator induces an anti-Leibniz bialgebra.

\begin{pro}\label{pro:O-cons}
Let $(A, \cdot)$ be an anti-Leibniz algebra and $(M, \kl, \kr)$ be a bimodule over
$(A, \cdot)$. Let $P: M\rightarrow A$ be a linear map which is identified as an element in
$(A\ltimes M^{\ast})\otimes(A \ltimes M^{\ast})$ through $\Hom(M, A)\cong A\otimes M^{\ast}
\subset(A\ltimes M^{\ast})\otimes(A \ltimes M^{\ast})$. Then $r:=P+\tau(P)$ is a
symmetric solution of the {\rm aLYBE} in semi-direct product anti-Leibniz
algebra $A \ltimes M^{\ast}$ if and only if $P$ is a relative Rota-Baxter operator of
$(A, \cdot)$ associated to $(M, \kl, \kr)$.
\end{pro}

\begin{proof}
Let $\{e_{1}, e_{2},\cdots, e_{n}\}$ be a basis of $M$ and $\{f_{1}, f_{2},\cdots, f_{n}\}$
be its dual basis in $M^{\ast}$. Then $P=\sum_{i=1}^{n}P(e_{i})\otimes f_{i}
\in(A\ltimes M^{\ast})\otimes(A\ltimes M^{\ast})$, $r=P+\tau(P)=\sum_{i=1}^{n}
\big(P(e_{i})\otimes f_{i}+f_{i}\otimes P(e_{i}))$, and
\begin{align*}
r_{12}r_{13}&=\sum_{i,j=1}^{n}P(e_{i})P(e_{j})\otimes f_{i}\otimes f_{j}
-\kl^{\ast}(P(e_{i}))(f_{j})\otimes f_{i}\otimes P(e_{j})
-(\kl^{\ast}-\kr^{\ast})(P(e_{j}))(f_{i})\otimes P(e_{i})\otimes f_{j},\\[-2mm]
r_{12}r_{23}&=\sum_{i,j=1}^{n}P(e_{i})\otimes(\kl^{\ast}-\kr^{\ast})(P(e_{j}))(f_{i})
\otimes f_{j}-f_{i}\otimes P(e_{i})P(e_{j})\otimes f_{j}
+f_{i}\otimes\kl^{\ast}(P(e_{i}))(f_{j})\otimes P(e_{j}),\\[-2mm]
r_{23}r_{12}&=\sum_{i,j=1}^{n}P(e_{j})\otimes\kl^{\ast}(P(e_{i}))(f_{j})
\otimes f_{i}-f_{j}\otimes P(e_{i})P(e_{j})\otimes f_{i}
+f_{j}\otimes(\kl^{\ast}-\kr^{\ast})(P(e_{j}))(f_{i})\otimes P(e_{i}),\\[-2mm]
r_{23}r_{13}&=\sum_{i,j=1}^{n}-f_{j}\otimes P(e_{i})\otimes(\kl^{\ast}-\kr^{\ast})
(P(e_{j}))(f_{i})-P(e_{j})\otimes f_{i}\otimes\kl^{\ast}(P(e_{i}))(f_{j})
+f_{i}\otimes f_{j}\otimes P(e_{i})P(e_{j}).
\end{align*}
Note that $\kl^{\ast}(P(e_{i}))(f_{j})=\sum_{k=1}f_{j}(\kl(P(e_{i}))
(e_{k}))f_{k}$, we get
\begin{align*}
\sum_{i,j=1}^{n}P(e_{j})\otimes\kl^{\ast}(P(e_{i}))(f_{j})\otimes f_{i}
=&\;\sum_{i,j, k=1}^{n}f_{j}(\kl(P(e_{i}))(e_{k}))P(e_{j})\otimes f_{k}\otimes f_{i}\\[-2mm]
=&\;\sum_{i,j=1}^{n}P\Big(\sum_{k=1}^{n}f_{k}(\kl(P(e_{i}))(e_{j}))e_{k}\Big)
\otimes f_{j}\otimes f_{i}
=\sum_{i,j=1}^{n}P(\kl(P(e_{i}))(e_{j}))\otimes f_{j}\otimes f_{i}.
\end{align*}
Similarly, we also have $\sum_{i,j=1}^{n}f_{j}\otimes\kl^{\ast}(P(e_{j}))(f_{i})
\otimes P(e_{i})=\sum_{i,j=1}^{n}f_{j}\otimes f_{i}\otimes P(\kl(P(e_{i}))(e_{j}))$,
$\sum_{i,j=1}^{n}P(e_{i})\otimes\kr^{\ast}(P(e_{j}))(f_{i})\otimes f_{j}=
\sum_{i,j=1}^{n}P(\kr(P(e_{j}))(e_{i}))\otimes f_{i}\otimes f_{j}$ and
$\sum_{i,j=1}^{n}f_{i}\otimes\kr^{\ast}(P(e_{i}))(f_{j})
\otimes P(e_{i})=f_{i}\otimes f_{j}\otimes P(\kr(P(e_{j}))(e_{i}))$. Thus
\begin{align*}
r_{12}r_{13}+r_{12}r_{23}-r_{23}r_{12}-r_{23}r_{13}
=&\;\sum_{i,j=1}^{n}\Big(\big(P(e_{i})P(e_{j})-P(\kr(P(e_{j}))(e_{i})
-P(\kl(P(e_{i}))(e_{j}))\big)\otimes f_{i}\otimes f_{j}\\[-4mm]
&\qquad\quad+f_{i}\otimes f_{j}\otimes\big(P(\kl(P(e_{i}))(e_{j}))+P(\kr(P(e_{j}))(e_{i}))
-P(e_{i})P(e_{j})\big)\Big).
\end{align*}
Thus, $r$ is a solution of the {\rm aLYBE} in $A\ltimes M^{\ast}$ if and only if
$P$ is a relative Rota-Baxter operator of $(A, \cdot)$ associated to $(M, \kl, \kr)$.
\end{proof}

As a direct conclusion, we have

\begin{cor}\label{cor:bia-semi-dir}
Let $(A, \cdot)$ be an anti-Leibniz algebra and $(M, \kl, \kr)$ be a bimodule over
$(A, \cdot)$. If $P: M\rightarrow A$ is a relative Rota-Baxter operator of
$(A, \cdot)$ associated to $(M, \kl, \kr)$, then there is an anti-Leibniz algebra
$(A \ltimes M^{\ast}, \Delta_{r})$, where $\Delta_{r}$ is given by Eq. \eqref{cobo} for
$r=P+\tau(P)$.
\end{cor}

\subsection{Factorizable anti-Leibniz bialgebras and skew-quadratic Rota-Baxter anti-Leibniz
algebras}\label{subsec:fact}
In this subsection, we establish the factorizable theories for anti-Leibniz bialgebras.
Let $(A, \cdot)$ be an anti-Leibniz algebra and $r\in A\otimes A$. We have defined a
linear map $r^{\sharp}: A^{\ast}\rightarrow A$ by
$$
\langle r^{\sharp}(\xi_{1}),\; \xi_{2}\rangle=\langle\xi_{1}\otimes\xi_{2},\; r\rangle.
$$
for any $\xi_{1}, \xi_{2}\in A^{\ast}$. If $(A, \cdot, \Delta_{r})$ is an anti-Leibniz
bialgebra, then the anti-Leibniz algebra structure $\cdot_{r}$ on $A^{\ast}$ dual to
the comultiplication $\Delta_{r}$ defined by Eq. \eqref{cobo} is given by
$$
\xi_{1}\cdot_{r}\xi_{2}=\fl_{A}^{\ast}(r^{\sharp}(\xi_{1}))(\xi_{2})
+\fl_{A}^{\ast}(\tau(r)^{\sharp}(\xi_{2}))(\xi_{1})-\fr_{A}^{\ast}(\tau(r)^{\sharp}
(\xi_{2}))(\xi_{1}),
$$
for any $\xi_{1}, \xi_{2}\in A^{\ast}$. Let $(A, \cdot)$ be an anti-Leibniz algebra.
In Subsection \ref{subsec:o-oper}, we have get that an element $r\in A\otimes A$ is
invariant if and only if $r^{\sharp}(\xi)a=r^{\sharp}\big((\fl_{A}^{\ast}
-\fr_{A}^{\ast})(a)(\xi)\big)$ for any $a\in A$ and $\xi\in A^{\ast}$.

\begin{lem}\label{lem:inva}
Let $(A,\cdot)$ be an anti-Leibniz algebra and $r\in A\otimes A$.
Denote by $\mathcal{I}=r^{\sharp}-\tau(r)^{\sharp}: A^{\ast}\rightarrow A$. Then
$r-\tau(r)$ is invariant if and only if $\mathcal{I}(\fl_{A}^{\ast}-\fr_{A}^{\ast})(a)
=\fr_{A}(a)\mathcal{I}$ or $\mathcal{I}\fr_{A}^{\ast}(a)=(\fl_{A}^{\ast}
-\fr_{A}^{\ast})(a)\mathcal{I}$, for any $a\in A$.
\end{lem}

\begin{proof}
It is easy to see that $\mathcal{I}^{\ast}=-\mathcal{I}$ if we identify $(A^{\ast})^{\ast}$
with $A$, and $\langle\mathcal{I}(\xi),\; \xi_{2}\rangle=(r-\tau(r))(\xi_{1}\otimes\xi_{2})$
for any $\xi_{1}, \xi_{2}\in A^{\ast}$. Since $r-\tau(r)$ is invariant if and only if
$\fr_{A}(a)((r^{\sharp}-\tau(r)^{\sharp})(\xi))=(r^{\sharp}-\tau(r)^{\sharp})
\big((\fl_{A}^{\ast}-\fr_{A}^{\ast})(a)(\xi)\big)$, by the dual, we get
$\mathcal{I}(\fl_{A}^{\ast}-\fr_{A}^{\ast})(a)=\fr_{A}(a)\mathcal{I}$ is equivalent to
$\mathcal{I}\fr_{A}^{\ast}(a)=(\fl_{A}-\fr_{A})(a)\mathcal{I}$ for any $a\in A$.
The proof is finished.
\end{proof}

\begin{pro}\label{pro:r-homo}
Let $(A, \cdot)$ be an anti-Leibniz algebra and $r\in A\otimes A$ such that
$r-\tau(r)$ is invariant. The following are equivalent.
\begin{enumerate}\itemsep=0pt
\item[$(i)$] $r$ is a solution of the {\rm aLYBE} in $(A, \cdot)$;
\item[$(ii)$] $(A^{\ast}, \cdot_{r})$ is an anti-Leibniz algebra and $r^{\sharp}:
     (A^{\ast}, \cdot_{r})\rightarrow(A, \cdot)$ is a homomorphism of anti-Leibniz algebras;
\item[$(iii)$] $(A^{\ast}, \cdot_{r})$ is an anti-Leibniz algebra and $\tau(r)^{\sharp}:
     (A^{\ast}, \cdot_{r})\rightarrow(A, \cdot)$ is a homomorphism of anti-Leibniz algebras.
\end{enumerate}
\end{pro}

\begin{proof}
$(i)\Leftrightarrow(ii)$. If $r$ is a solution of the {\rm aLYBE} in $(A, \cdot)$,
then $(A, \cdot, \Delta_{r})$ is a quasi-triangular anti-Leibniz bialgebra, and so that
$(A^{\ast}, \cdot_{r})$ is an anti-Leibniz algebra. Since the product $\cdot_{r}$ is
given by $\xi_{1}\cdot_{r}\xi_{2}=\fl_{A}^{\ast}(r^{\sharp}(\xi_{1}))(\xi_{2})
+\fl_{A}^{\ast}(\tau(r)^{\sharp}(\xi_{2}))(\xi_{1})-\fr_{A}^{\ast}(\tau(r)^{\sharp}
(\xi_{2}))(\xi_{1})$, this equivalence is following Proposition \ref{pro:r-map}.

$(i)\Leftrightarrow(iii)$. By a direct calculation,
for any $\xi_{1}, \xi_{2}, \xi_{3}\in A^{\ast}$, we have
\begin{align*}
\langle r^{\sharp}(\xi_{1})r^{\sharp}(\xi_{2}),\; \xi_{3}\rangle
&=\sum_{i, j}\langle x_{j}\otimes x_{i}\otimes y_{j}y_{i},\;
\xi_{1}\otimes\xi_{2}\otimes\xi_{3}\rangle,\\[-2mm]
\langle r^{\sharp}(\fr_{A}^{\ast}(\tau(r)^{\sharp}(\xi_{2}))(\xi_{1})),\; \xi_{3}\rangle
&=\sum_{i, j}\langle x_{i}x_{j}\otimes y_{j}\otimes y_{i},\;
\xi_{1}\otimes\xi_{2}\otimes\xi_{3}\rangle,\\[-2mm]
\langle r^{\sharp}(\fl_{A}^{\ast}(\tau(r)^{\sharp}(\xi_{2}))(\xi_{1})),\; \xi_{3}\rangle
&=\sum_{i, j}\langle x_{j}x_{i}\otimes y_{j}\otimes y_{i},\;
\xi_{1}\otimes\xi_{2}\otimes\xi_{3}\rangle,\\[-2mm]
\langle r^{\sharp}(\fl_{A}^{\ast}(r^{\sharp}(\xi_{1}))(\xi_{2})),\; \xi_{3}\rangle
&=\sum_{i, j}\langle x_{j}\otimes y_{j}x_{i}\otimes y_{i},\;
\xi_{1}\otimes\xi_{2}\otimes\xi_{3}\rangle.
\end{align*}
Hence,
\begin{align*}
&\;\langle r^{\sharp}(\xi_{1}\cdot_{r}\xi_{2})-r^{\sharp}(\xi_{1})r^{\sharp}(\xi_{2}),\; \xi_{3}\rangle\\
=&\;\langle r^{\sharp}(\fl_{A}^\ast(\tau(r)^{\sharp}(\xi_{2}))(\xi_{1}))
+r^{\sharp}(\fl_{A}^{\ast}(r^{\sharp}(\xi_{1}))(\xi_{2}))
-r^{\sharp}(\fr_{A}^{\ast}(\tau(r)^{\sharp}(\xi_{2}))(\xi_{1}))
-r^{\sharp}(\xi_{1})r^{\sharp}(\xi_{2}),\; \xi_{3}\rangle\\
=&\;\sum_{i, j}\langle x_{j}x_{i}\otimes y_{j}\otimes y_{i}
+x_{j}\otimes y_{j}x_{i}\otimes y_{i}-x_{i}x_{j}\otimes y_{j}\otimes y_{i}
-x_{j}\otimes x_{i}\otimes y_{j}y_{i},\; \xi_{1}\otimes\xi_{2}\otimes\xi_{3}\rangle\\
=&\;\langle((\fr_{A}-\fl_{A})(x_{j})\otimes\id+\id\otimes\fr_{A}(x_{j}))(r-\tau(r))
\otimes y_{j}+(\tau\otimes\id)([\![r, r]\!]),\; \xi_{1}\otimes\xi_{2}\otimes\xi_{3}\rangle.
\end{align*}
That is to say, if $r-\tau(r)$ is invariant and $r$ is a solution of the {\rm aLYBE}
in $(A, \cdot)$, then $r^{\sharp}$ is a homomorphism of anti-Leibniz algebras.
Conversely, it is directly available from the above calculation.
\end{proof}

Triangular anti-Leibniz bialgebras is an important class of anti-Leibniz bialgebras.
Another important class of anti-Leibniz bialgebras is the factorizable
anti-Leibniz bialgebras.

\begin{defi}\label{defi:fact}
A quasi-triangular anti-Leibniz bialgebra $(A, \cdot, \Delta)$ is called
{\rm factorizable} if the linear map $\mathcal{I}=r^{\sharp}-\tau(r)^{\sharp}: A^{\ast}
\rightarrow A$ is a linear isomorphism of vector spaces.
\end{defi}

Clearly, quasi-triangular anti-Leibniz bialgebras contain triangular anti-Leibniz
bialgebras and factorizable anti-Leibniz bialgebras as two subclasses.
Moreover, it is easy to see that a quasi-triangular anti-Leibniz bialgebra
$(A, \cdot, \Delta)$ is a triangular anti-Leibniz bialgebra if and only if
$\mathcal{I}=r^{\sharp}-\tau(r)^{\sharp}: A^{\ast}\rightarrow A$ is the zero map.
Factorizable anti-Leibniz bialgebras are however concerned with the opposite case.
Following, for convenience, we consider the linear map $\mathcal{I}=r^{\sharp}
-\tau(r)^{\sharp}: A^{\ast}\rightarrow A$ as a composition of maps as follows:
$$
A^{\ast}\xrightarrow{\quad r^{\sharp}\oplus \tau(r)^{\sharp}\quad} A\oplus A
\xrightarrow{\quad(a_{1}, a_{2})\mapsto a_{1}-a_{2}\quad} A.
$$
The following result justifies the terminology of a factorizable anti-Leibniz bialgebra.

\begin{pro}\label{pro:deco}
Let $(A, \cdot)$ be an anti-Leibniz algebra and $r\in A\otimes A$. If there is an
anti-Leibniz bialgebra $(A, \cdot, \Delta_{r})$ induced by $r$ and it is factorizable,
then $\Img(r^{\sharp}\oplus \tau(r)^{\sharp})$ is an anti-Leibniz subalgebra of the direct
sum anti-Leibniz algebra $A\oplus A$, which is isomorphic to the anti-Leibniz
algebra $(A^{\ast}, \cdot_{r})$. Moreover, any $a\in A$ has a unique decomposition
$a=a_{+}+a_{-}$, where $a_{+}\in\Img(r^{\sharp})$ and $a_{-}\in\Img(\tau(r)^{\sharp})$.
\end{pro}

\begin{proof}
Since $(A, \cdot, \Delta_{r})$ is quasi-triangular anti-Leibniz bialgebra,
both $r^{\sharp}$ and $\tau(r)^{\sharp}$ are homomorphisms of anti-Leibniz algebras.
Therefore, $\Img(r^{\sharp}\oplus \tau(r)^{\sharp})$ is an anti-Leibniz subalgebra
of the direct sum anti-Leibniz algebra $A\oplus A$. Since $\mathcal{I}:
A^{\ast}\rightarrow A$ is a linear isomorphism, it follows that $r^{\sharp}\oplus
\tau(r)^{\sharp}$ is injective, then the anti-Leibniz algebra $\Img(r^{\sharp}\oplus
\tau(r)^{\sharp})$ is isomorphic to the anti-Leibniz algebra $(A^{\ast}, \cdot_{r})$.
Moreover, since $\mathcal{I}$ is an isomorphism again, for any $a\in A$, we have
\begin{align*}
a=(r^{\sharp}-\tau(r)^{\sharp})(\mathcal{I}^{-1}(a))
=r^{\sharp}(\mathcal{I}^{-1}(a))-\tau(r)^{\sharp}(\mathcal{I}^{-1}(a)),
\end{align*}
which implies that $a=a_{+}+a_{-}$, where $a_{+}=r^{\sharp}(\mathcal{I}^{-1}(a))$ and
$a_{-}=-\tau(r)^{\sharp}(\mathcal{I}^{-1}(a))$. The uniqueness also follows from the fact
that $\mathcal{I}$ is an isomorphism.
\end{proof}

In Theorem \ref{thm:double-coba}, for any anti-Leibniz bialgebra $(A, \cdot, \Delta)$,
we construct a new anti-Leibniz bialgebra $(A\oplus A^{\ast}, \ast, \Delta_{\tilde{r}})$,
which is called the double of anti-Leibniz bialgebra $(A, \cdot, \Delta)$.
The comultiplication $\Delta_{\tilde{r}}$ is given by Eq. \eqref{cobo} for $\tilde{r}=\sum_{i}
e_{i}\otimes f_{i}$, where $\{e_{1}, e_{2}, \cdots, e_{n}\}$ is a basis of $A$ and
$\{f_{1}, f_{2}, \cdots, f_{n}\}$ be its dual basis. The importance of factorizable
anti-Leibniz bialgebras in the study of anti-Leibniz bialgebras can also be
observed from the following proposition that the double space of an arbitrary
anti-Leibniz bialgebra admits a factorizable anti-Leibniz bialgebra structure.

\begin{pro}\label{pro:doufact}
With the above notations, the double anti-Leibniz bialgebra $(A\bowtie A^\ast, \ast,
\Delta_{\tilde{r}})$ of any anti-Leibniz bialgebra $(A, \cdot, \Delta)$ is factorizable.
\end{pro}

\begin{proof}
Let $\tilde{r}=\sum_{i}e_{i}\otimes f_{i}$, where $\{e_{1}, e_{2}, \cdots, e_{n}\}$ is
a basis of $A$ and $\{f_{1}, f_{2}, \cdots, f_{n}\}$ be its dual basis in $A^{\ast}$.
Since $\tilde{r}-\tau(\tilde{r})=\sum_{i=1}^{n}(e_{i}\otimes f_{i}-f_{i}\otimes e_{i})$,
for any $(\xi, a)\in A^{\ast}\oplus A$, we get $\tilde{\mathcal{I}}(\xi, a)=(-a, \xi)
\in A\oplus A^{\ast}$, where $\tilde{\mathcal{I}}=\tilde{r}^{\sharp}
-\tau(\tilde{r})^{\sharp}$. Thus, for any $(a_{1}, \xi_{1})$, $(a_{2}, \xi_{2})
\in A\oplus A^{\ast}$,
\begin{align*}
&\;\fr^{\ast}_{A\bowtie A^{\ast}}(a_{1}, \xi_{1})
\big(\tilde{\mathcal{I}}(\xi_{2}, a_{2})\big)\\
=&\;(-a_{2}a_{1}+\fl_{A^{\ast}}^{\ast}(\xi_{2})(a_{1})
+(\fl_{A^{\ast}}^{\ast}-\fr_{A^{\ast}}^{\ast})(\xi_{1})(-a_{2}),\ \
\xi_{2}\cdot_{A^{\ast}}\xi_{1}+\fl_{A}^{\ast}(-a_{2})(\xi_{1})
+(\fl_{A}^{\ast}-\fr_{A}^{\ast})(a_{1})(\xi_{2}))\\
=&\;\tilde{\mathcal{I}}(\xi_{2}\cdot_{A^{\ast}}\xi_{1}+\fl_{A}^{\ast}(-a_{2})(\xi_{1})
+(\fl_{A}^{\ast}-\fr_{A}^{\ast})(a_{1})(\xi_{2}),\ \
a_{2}a_{1}-\fl_{A^{\ast}}^{\ast}(\xi_{2})(a_{1})+(\fr_{A^{\ast}}^{\ast}
-\fl_{A^{\ast}}^{\ast})(\xi_{1})(-a_{2}))\\
=&\;\tilde{\mathcal{I}}\big(((\fl_{A\bowtie A^{\ast}}^{\ast}-\fr_{A\bowtie A^{\ast}}^{\ast})
(a_{1}, \xi_{1}))(\xi_{2}, a_{2})\big).
\end{align*}
By Lemma \ref{lem:inva}, we get $\tilde{r}-\tau(\tilde{r})$ is invariant.
Thus, $(A\bowtie A^{\ast}, \ast, \Delta_{\tilde{r}})$ is a quasi-triangular
anti-Leibniz bialgebra. Finally, since $\mathcal{I}(\xi, a)=(-a, \xi)$ is a linear
isomorphism, we get that $(A\bowtie A^{\ast}, \ast, \Delta_{\tilde{r}})$ is a
factorizable anti-Leibniz bialgebra.
\end{proof}

Proposition \ref{pro:doufact} provides a method for constructing
factorizable anti-Leibniz bialgebras.

\begin{ex}\label{ex:double-fact}
Consider the $2$-dimensional anti-Leibniz bialgebra $(A, \cdot, \Delta)$ given by
Example \ref{ex:YBE}, where $A=\Bbbk\{e_{1}, e_{2}\}$, the nonzero multiplication and
comultiplication: $e_{1}e_{1}=e_{2}$ and $\Delta(e_{1})=e_{2}\otimes e_{2}$.
Denote by $\{f_{1}, f_{2}\}$ the dual basis of $\{e_{1}, e_{2}\}$ and
$\tilde{r}=e_{1}\otimes f_{1}+e_{2}\otimes f_{2}\in A\otimes A^{\ast}\subset
(A\oplus A^{\ast})\otimes(A\oplus A^{\ast})$. Then we get a $4$-dimensional
anti-Leibniz bialgebra $(A\oplus A^{\ast}, \ast, \Delta_{\tilde{r}})$, where
\begin{align*}
&e_{1}\ast e_{1}=e_{2}, \qquad f_{2}\ast f_{2}=f_{1},\qquad
e_{1}\ast f_{2}=f_{1}, \qquad f_{2}\ast e_{1}=e_{2},\\
&\qquad\qquad\Delta_{\tilde{r}}(e_{1})=e_{2}\otimes e_{2}, \qquad\qquad
\Delta_{\tilde{r}}(f_{2})=f_{1}\otimes f_{1}.
\end{align*}
One can check $(A\oplus A^{\ast}, \ast, \Delta_{\tilde{r}})$ be a quasi-triangular
anti-Leibniz bialgebra. Note that $r-\tau(r)=e_{1}\otimes f_{1}+e_2\otimes f_{2}
-f_{1}\otimes e_{1}-f_{2}\otimes e_{2}$, we get that the linear map
$\tilde{\mathcal{I}}: (A\oplus A^{\ast})^{\ast}\rightarrow A\oplus A^{\ast}$ is given by
$$
\tilde{\mathcal{I}}(e_{1}^{\star})=f_{1},\qquad\quad
\tilde{\mathcal{I}}(e_{2}^{\star})=f_{2},\qquad\quad
\tilde{\mathcal{I}}(f_{1}^{\star})=-e_{1},\qquad\quad
\tilde{\mathcal{I}}(f_{2}^{\star})=-e_{2},
$$
where $\{e_{1}^{\star}, e_{2}^{\star}, f_{1}^{\star}, f_{2}^{\star}\}$ in
$(A\oplus A^{\ast})^{\ast}$ is the dual basis of $\{e_{1}, e_{2}, f_{1}, f_{2}\}$ in
$A\oplus A^{\ast}$. Hence, we get $\tilde{\mathcal{I}}$ is a linear isomorphism,
and so that $(A\oplus A^{\ast}, \ast, \Delta_{\tilde{r}})$ is factorizable
anti-Leibniz bialgebra.
\end{ex}

Next, we will give a characterization of factorizable anti-Leibniz bialgebras by
the Rota-Baxter operator on anti-Leibniz algebras with a nondegenerate
skew-symmetric invariant bilinear form. Let $(A, \cdot)$ be an anti-Leibniz algebra
and $\lambda\in\Bbbk$. A linear map $R: A\rightarrow A$ is called a
{\it Rota-Baxter operator of weight $\lambda$} on $(A, \cdot)$ if
$$
R(a_{1})R(a_{2})=R\Big(R(a_{1})a_{2}+a_{1}R(a_{2})+\lambda a_{1}a_{2}\Big),
$$
for any $a_{1}, a_{2}\in A$. Let $(A, \cdot, R)$ be an anti-Leibniz algebra with a
Rota-Baxter operator $R$ of weight $\lambda$. Then there is a new multiplication
$\cdot_{R}$ on $A$ defined by
$$
a_{1}\cdot_{R}a_{2}=R(a_{1})a_{2}+a_{1}R(a_{2})+\lambda a_{1}a_{2},
$$
for any $a_{1}, a_{2}\in A$. This anti-Leibniz algebra is denoted by $A_{R}$.
It is easy to see that $R$ is a homomorphism of anti-Leibniz algebras from
$(A, \cdot_{R})$ to $(A, \cdot)$.

\begin{defi}\label{defi:quaRB}
A linear map $R: A\rightarrow A$ is called a {\rm Rota-Baxter operator of weight $\lambda$
on an anti-Leibniz algebra $(A, \cdot)$ with a nondegenerate skew-symmetric invariant
bilinear form $\mathfrak{B}(-,-)$} if $R$ is a Rota-Baxter operator of weight $\lambda$
on $(A, \cdot)$ and for any $a_{1}, a_{2}\in A$,
\begin{align*}
\mathfrak{B}(R(a_{1}), a_{2})+\mathfrak{B}(a_{1}, R(a_{2}))
+\lambda\mathfrak{B}(a_{1}, a_{2})=0.
\end{align*}
The quadruple $(A, \cdot, \mathfrak{B}, R)$ is called a {\rm skew-quadratic Rota-Baxter
anti-Leibniz algebra of weight $\lambda$}.
\end{defi}

By direct calculations, we have

\begin{pro}\label{pro:skew-dual}
Let $(A, \cdot)$ be an anti-Leibniz algebra, $\mathfrak{B}(-,-)$ be a
nondegenerate skew-symmetric invariant bilinear form on $(A, \cdot)$ and
$R: A\rightarrow A$ is a linear map. Then, $(A, \cdot, \mathfrak{B}, R)$ is
a skew-quadratic Rota-Baxter anti-Leibniz algebra of weight $\lambda$ if and only if
$(A, \cdot, -\mathfrak{B}, -(\lambda\id+R))$ is a skew-quadratic Rota-Baxter
anti-Leibniz algebra of weight $\lambda$.
\end{pro}

The following theorem shows that there is a one-to-one correspondence between factorizable
anti-Leibniz bialgebras and skew-quadratic Rota-Baxter anti-Leibniz algebras of nonzero weight.

\begin{thm}\label{thm:corre}
Let $(A, \cdot)$ be an anti-Leibniz algebra and $r\in A\otimes A$.
Denote $\mathcal{I}=r^{\sharp}-\tau(r)^{\sharp}$ and define a bilinear form
$\mathfrak{B}_{\mathcal{I}}$ by $\mathfrak{B}_{\mathcal{I}}(a_{1}, a_{2})=\langle
\mathcal{I}^{-1}(a_{1}),\; a_{2}\rangle$, for any $a_{1}, a_{2}\in A$.
If there is an anti-Leibniz bialgebra $(A, \cdot, \Delta)$ induced by $r$ and it is
factorizable, then $(A, \cdot, \mathfrak{B}_{\mathcal{I}}, R=\lambda\tau(r)^{\sharp}
\mathcal{I}^{-1})$ is a skew-quadratic Rota-Baxter anti-Leibniz algebra of weight $\lambda$.

Conversely, for any skew-quadratic Rota-Baxter anti-Leibniz algebra $(A, \cdot,
\mathfrak{B}, R)$ of weight $\lambda$, we have a linear isomorphism
$\mathcal{I}_{\mathfrak{B}}: A^{\ast}\rightarrow A$ by
$\langle\mathcal{I}^{-1}_{\mathfrak{B}}(a_{1}),\; a_{2}\rangle=\mathfrak{B}(a_{1}, a_{2})$,
for any $a_{1}, a_{2}\in A$. If $\lambda\neq0$, we define
$$
r^{\sharp}:=\mbox{$\frac{1}{\lambda}$}(R+\lambda\id)\mathcal{I}_{\mathfrak{B}}:\quad
A^{\ast}\longrightarrow A,
$$
and define $r\in A\otimes A$ by $\langle r^{\sharp}(\xi_{1}),\; \xi_{2}\rangle
=\langle r,\; \xi_{1}\otimes\xi_{2}\rangle$, for any $\xi_{1}, \xi_{2}\in A^{\ast}$.
Then $r$ is a solution of the {\rm aLYBE} in $(A, \cdot)$ and gives rise
to a factorizable anti-Leibniz bialgebra $(A, \cdot, \Delta_{r})$, where $\Delta_{r}$
is given by Eq. \eqref{cobo}.
\end{thm}

\begin{proof}
Since, $r^{\sharp},\; \tau(r)^{\sharp}: (A^{\ast}, \cdot_{r})\rightarrow(A, \cdot)$ are
homomorphisms of anti-Leibniz algebras, for all $a_{1}, a_{2}\in A$, we have
\begin{align}
\mathcal{I}\big(\mathcal{I}^{-1}(a_{1})\cdot_{r}\mathcal{I}^{-1}(a_{2})\big)
=&\; (r^{\sharp}-\tau(r)^{\sharp})(\mathcal{I}^{-1}(a_{1})\cdot_{r}
\mathcal{I}^{-1}(a_{2}))                                               \label{I}\\
=&\;(\mathcal{I}+\tau(r)^{\sharp})(\mathcal{I}^{-1}(a_{1}))
(\mathcal{I}+\tau(r)^{\sharp})(\mathcal{I}^{-1}(a_{2}))
-\tau(r)^{\sharp}(\mathcal{I}^{-1}(a_{1}))\tau(r)^{\sharp}
(\mathcal{I}^{-1}(a_{2}))  \nonumber\\
=&\; \tau(r)^{\sharp}(\mathcal{I}^{-1}(a_{1}))a_{2}
+a_{1}\tau(r)^{\sharp}(\mathcal{I}^{-1}(a_2))+a_{1}a_{2}.        \nonumber
\end{align}
Therefore, for any $a_{1}, a_{2}\in A$,
\begin{align*}
R\big(R(a_{1})a_{2}+a_{1}R(a_2)+\lambda a_{1}a_{2}\big)
=&\; \lambda^{2}\tau(r)^{\sharp}\mathcal{I}^{-1}\big(\tau(r)^{\sharp}
(\mathcal{I}^{-1}(a_{1}))a_{2}+a_{1}\tau(r)^{\sharp}(\mathcal{I}^{-1}(a_{2}))+a_{1}a_{2}\big)\\
=&\; \lambda^{2}\tau(r)^{\sharp}\big(\mathcal{I}^{-1}(a_{1})
\cdot_{r}\mathcal{I}^{-1}(a_{2})\big)\\
=&\; R(a_{1})R(a_{2}),
\end{align*}
which implies that $R$ is a Rota-Baxter operator of weight $\lambda$ on $(A, \cdot)$.

Next, we prove that $\mathfrak{B}_{\mathcal{I}}(-,-)$ is nondegenerate skew-symmetric
invariant. $\mathfrak{B}_{\mathcal{I}}(-,-)$ is nondegenerate since $\mathcal{I}$ is an
isomorphism. Moreover, since $(\mathcal{I}^{-1})^{\ast}=-\mathcal{I}^{-1}$, we get
$$
\mathfrak{B}_{\mathcal{I}}(a_{1}, a_{2})=\langle\mathcal{I}^{-1}(a_{1}),\; a_{2}\rangle
=-\langle a_{1},\; \mathcal{I}^{-1}(a_{2})\rangle=-\mathfrak{B}_{\mathcal{I}}(a_{2}, a_{1}),
$$
for any $a_{1}, a_{2}\in A$. Hence, $\mathfrak{B}_\mathcal{I}$ is skew-symmetric.
Moreover, since the $r-\tau(r)$ is invariant, by Lemma \ref{lem:inva}, we have
$\mathcal{I}^{-1}\fr_{A}(a)=(\fl_{A}^{\ast}-\fr_{A}^{\ast})(a)\mathcal{I}^{-1}$. Thus,
$$
\mathfrak{B}_{\mathcal{I}}(a_{1},\; a_{2}a_{3})
-\mathfrak{B}_{\mathcal{I}}(a_{1},\; a_{2}a_{3})
=\langle\mathcal{I}^{-1}(\fr_{A}(a_{2})(a_{1}))
-(\fl_{A}^{\ast}-\fr_{A}^{\ast})(a_{2})(\mathcal{I}^{-1}(a_{1}))),\ \ a_{3}\rangle=0,
$$
for any $a_{1}, a_{2}, a_{3}\in A$, which implies that $\mathfrak{B}_{\mathcal{I}}$ is
invariant. Moreover, since $(\tau(r)^{\sharp})^{\ast}=r^{\sharp}$, we get
\begin{align*}
&\; \mathfrak{B}_{\mathcal{I}}(a_{1}, R(a_{2}))+\mathfrak{B}_{\mathcal{I}}(R(a_{1}), a_{2})
+\lambda\mathfrak{B}_{\mathcal{I}}(a_{1}, a_{2})\\
=&\; \lambda\langle\mathcal{I}^{-1}(a_{1}),\; \tau(r)^{\sharp}(\mathcal{I}^{-1}(a_{2}))\rangle
+\lambda\langle\mathcal{I}^{-1}(\tau(r)^{\sharp}(\mathcal{I}^{-1}(a_{1}))),\; a_{2}\rangle
+\lambda\langle\mathcal{I}^{-1}(a_{1}),\; a_{2}\rangle\\
=&\;-\lambda\langle\mathcal{I}^{-1}(a_{1}),\; a_{2}\rangle
+\lambda\langle\mathcal{I}^{-1}(a_{1}),\; a_{2}\rangle\\
=&\;0,
\end{align*}
for $a_{1}, a_{2}\in A$. Thus, $(A, \cdot, \mathfrak{B}_{\mathcal{I}}, R)$ is a
skew-quadratic Rota-Baxter anti-Leibniz algebras of weight $\lambda$.

Conversely, suppose that $(A, \cdot, \mathfrak{B}, R)$ is a skew-quadratic Rota-Baxter
anti-Leibniz algebra of weight $\lambda$. First, since $\mathfrak{B}(-,-)$ is
skew-symmetric and $\mathfrak{B}(a_{1}, R(a_{2}))+\mathfrak{B}(R(a_{1}), a_{2})
+\lambda\mathfrak{B}(a_{1}, a_{2})=0$, we get $\mathcal{I}_{\mathfrak{B}}^{\ast}=
-\mathcal{I}_{\mathfrak{B}}$ and
$$
\langle\mathcal{I}_{\mathfrak{B}}^{-1}(a_{1}),\; R(a_{2})\rangle+\langle
\mathcal{I}_{\mathfrak{B}}^{-1}(R(a_{1})),\; a_{2}\rangle
+\lambda\langle\mathcal{I}_{\mathfrak{B}}^{-1}(a_{1}),\; a_{2}\rangle=0,
$$
for any $a_{1}, a_{2}\in A$, which implies that $R^{\ast}\mathcal{I}_{\mathfrak{B}}^{-1}
+\mathcal{I}_{\mathfrak{B}}^{-1}R+\lambda\mathcal{I}_{\mathfrak{B}}^{-1}=0$, and so that,
$\mathcal{I}_{\mathfrak{B}}R^{\ast}+R\mathcal{I}_{\mathfrak{B}}
+\lambda\mathcal{I}_{\mathfrak{B}}=0$. Therefore, we get
$\tau(r)^{\sharp}=(r^{\sharp})^{\ast}=-\frac{1}{\lambda}\big(\mathcal{I}_{\mathfrak{B}}R^{\ast}
+\lambda\mathcal{I}_{\mathfrak{B}}\big)=\frac{1}{\lambda}R\mathcal{I}_{\mathfrak{B}}$,
and $\mathcal{I}_{\mathfrak{B}}=r^{\sharp}-\tau(r)^{\sharp}$. Define a multiplication
$\cdot_{r}$ on $A^{\ast}$ by
$$
\xi_{1}\cdot_{r}\xi_{2}=\fl_{A}^{\ast}(r^{\sharp}(\xi_{1}))(\xi_{2})
+\fl_{A}^{\ast}(\tau(r)^{\sharp}(\xi_{2}))(\xi_{1})
-\fr_{A}^{\ast}(\tau(r)^{\sharp}(\xi_{2}))(\xi_{1}),
$$
for any $\xi_{1}, \xi_{2}\in A^{\ast}$. Second, we will show that $\frac{1}{\lambda}
\mathcal{I}_{\mathfrak{B}}: (A, \cdot_{r})\rightarrow(A, \cdot_{R})$ is an isomorphism
of anti-Leibniz algebras. Indeed, for any $a_{1}, a_{2}, a_{3}\in A$, since
$\mathfrak{B}(a_{1}a_{2},\; a_{3})=\mathfrak{B}(a_{1},\; a_{2}a_{3}-a_{3}a_{2})$, we get
$\langle\mathcal{I}^{-1}_{\mathfrak{B}}(\fr_{A}(a_{2})(a_{1})),\; a_{3}\rangle
=\langle(\fl_{A}^{\ast}-\fr_{A}^{\ast})(a_{2})(\mathcal{I}^{-1}_{\mathfrak{B}}(a_{1})),\;
a_{3}\rangle$, and so that, $\fr_{A}(a)\mathcal{I}_{\mathfrak{B}}=
\mathcal{I}_{\mathfrak{B}}\big((\fl_{A}^{\ast}-\fr_{A}^{\ast})(a_{1})\big)$.
That is, $r-\tau(r)$ is invariant. Thus, for any $\xi_{1}, \xi_{2}\in A^{\ast}$, we have
\begin{align*}
\mathcal{I}_{\mathfrak{B}}(\xi_{1}\cdot_{r}\xi_{2})
&=\mathcal{I}_{\mathfrak{B}}\Big(\fl_{A}^{\ast}(r^{\sharp}(\xi_{1}))(\xi_{2})
+\fl_{A}^{\ast}(\tau(r)^{\sharp}(\xi_{2}))(\xi_{1})
-\fr_{A}^{\ast}(\tau(r)^{\sharp}(\xi_{2}))(\xi_{1})\Big)\\
&=\fr_{A}(r^{\sharp}(\xi_{1}))(\mathcal{I}_{\mathfrak{B}}(\xi_{2}))
+\fr_{A}(\tau(r)^{\sharp}(\xi_{2}))(\mathcal{I}_{\mathfrak{B}}(\xi_{1}))\\
&=r^{\sharp}(\xi_{1})r^{\sharp}(\xi_{2})-r^{\sharp}(\xi_{1})\tau(r)^{\sharp}(\xi_{2})
+r^{\sharp}(\xi_{1})\tau(r)^{\sharp}(\xi_{2})-\tau(r)^{\sharp}(\xi_{1})\tau(r)^{\sharp}(\xi_{2})\\
&=r^{\sharp}(\xi_{1})r^{\sharp}(\xi_{2})-\tau(r)^{\sharp}(\xi_{1})\tau(r)^{\sharp}(\xi_{2}),
\end{align*}
and
\begin{align*}
&\;\mathcal{I}_{\mathfrak{B}}(\xi_{1})\cdot_{R}\mathcal{I}_{\mathfrak{B}}(\xi_{2})\\
=&\;R(\mathcal{I}_{\mathfrak{B}}(\xi_{1}))\mathcal{I}_{\mathfrak{B}}(\xi_{2})
+\mathcal{I}_{\mathfrak{B}}(\xi_{1})R(\mathcal{I}_{\mathfrak{B}}(\xi_{2}))
+\lambda\mathcal{I}_{\mathfrak{B}}(\xi_{1})\mathcal{I}_{\mathfrak{B}}(\xi_{2})\\
=&\;\lambda \tau(r)^{\sharp}(\xi_{1})\big(r^{\sharp}(\xi_{2})-\tau(r)^{\sharp}(\xi_{2})\big)
+\lambda\big(r^{\sharp}(\xi_{1})-\tau(r)^{\sharp}(\xi_{1})\big)\tau(r)^{\sharp}(\xi_{2})
+\lambda\big(r^{\sharp}(\xi_{1})-\tau(r)^{\sharp}(\xi_{1})\big)
\big(r^{\sharp}(\xi_{2})-\tau(r)^{\sharp}(\xi_{2})\big)\\
=&\;\lambda r^{\sharp}(\xi_{1})r^{\sharp}(\xi_{2})-\lambda
\tau(r)^{\sharp}(\xi_{1})\tau(r)^{\sharp}(\xi_{2}),
\end{align*}
Thus, $\frac{1}{\lambda}\mathcal{I}_\mathfrak{B}$ is an anti-Leibniz algebra
isomorphism from $(A^{\ast}, \cdot_{r})$ to $(A, \cdot_{R})$.

Finally, since $R+\lambda\id,\; R: (A, \cdot_{P})\rightarrow(A, \cdot)$ are
homomorphisms of anti-Leibniz algebras, we deduce that
$r^{\sharp}=\frac{1}{\lambda}(R+\lambda\id)\mathcal{I}_{\mathfrak{B}}$ and
$\tau(r)^{\sharp}=\frac{1}{\lambda}R\mathcal{I}_{\mathfrak{B}}$ are homomorphisms of
anti-Leibniz algebras from $(A^{\ast}, \cdot_{r})$ to $(A, \cdot)$.
Hence, by Proposition \ref{pro:r-homo}, we get that $r$ is a solution of the {\rm aLYBE}
in $(A, \cdot)$, and so that $(A, \cdot, \Delta)$ is a factorizable anti-Leibniz
bialgebra, since $\mathcal{I}_{\mathfrak{B}}=r^{\sharp}-\tau(r)^{\sharp}$ is an isomorphism.
\end{proof}

Thus, if $\lambda\neq0$, we have the following commutative diagram for the
factorizable anti-Leibniz bialgebras and skew-quadratic Rota-Baxter anti-Leibniz algebra:
{\small
$$
\xymatrix@C=3cm@R=1cm{
\txt{$(A, \cdot, \Delta_{r})$\\ {\tiny factorizable anti-Leibniz bialgebra}}
\ar@{<->}[r]^-{{\rm Cor.\; } \ref{cor:extbia}}
\ar@<-1ex>@{<->}[d]_-{{\rm Thm.\; } \ref{thm:corre}}
& \txt{$(A, \cdot, \Delta_{\tau(r)})$\\ {\tiny factorizable anti-Leibniz bialgebra}}
\ar@<-1ex>@{<->}[d]^-{{\rm Thm.\; } \ref{thm:corre}}\\
\txt{$(A, \cdot, \mathcal{B}, R)$\\ {\tiny skew-quadratic Rota-Baxter anti-Leibniz algebra}}
\ar@{<->}[r]^-{{\rm Prop.\; } \ref{pro:skew-dual}}
& \txt{$(A, \cdot, -\mathcal{B}, -(\lambda\id+R))$\\
{\tiny skew-quadratic Rota-Baxter anti-Leibniz algebra}}}
$$}

Moreover, following Theorem \ref{thm:corre}, we have two corollaries.

\begin{cor}\label{cor:corre1}
Let $(A, \cdot)$ be an anti-Leibniz algebra and $r\in A\otimes A$. If there is an
anti-Leibniz bialgebra $(A, \cdot, \Delta)$ induced by $r$ and it is factorizable,
then $-\lambda\id-P$ is also a Rota-Baxter operator of weight $\lambda$ on quadratic
anti-Leibniz algebra $(A, \cdot, \mathfrak{B}_{\mathcal{I}})$, where
$\mathcal{I}=r^{\sharp}-\tau(r)^{\sharp}$, $P=\lambda \tau(r)^{\sharp}\mathcal{I}^{-1}:
A\rightarrow A$ and the bilinear form $\mathfrak{B}_{\mathcal{I}}$ is defined
by $\mathfrak{B}_{\mathcal{I}}(a_{1}, a_{2})=\langle\mathcal{I}^{-1}(a_{1}),\;
a_{2}\rangle$, for any $a_{1}, a_{2}\in A$.
\end{cor}

\begin{proof}
It can be obtained by direct calculation.
\end{proof}

\begin{cor}\label{cor:iso}
Let $(A, \cdot)$ be an anti-Leibniz algebra and $r\in A\otimes A$.
Define $\mathcal{I}=r^{\sharp}-\tau(r)^{\sharp}$ and $P=\lambda \tau(r)^{\sharp}\mathcal{I}^{-1}:
A\rightarrow A$ where $0\neq\lambda\in\Bbbk$. If there is an anti-Leibniz
bialgebra $(A, \cdot, \Delta)$ induced by $r$ and it is factorizable,
then $(A, \cdot_{P}, \Delta_{\mathcal{I}})$ is an anti-Leibniz bialgebra, where
$$
\Delta_{\mathcal{I}}^{\ast}(\xi_{1}\otimes\xi_{2})=
-\mbox{$\frac{1}{\lambda}$}\mathcal{I}^{-1}(\mathcal{I}(\xi_{1})\mathcal{I}(\xi_{2})),
$$
for any $\xi_{1}, \xi_{2}\in A^{\ast}$. Moreover, $\frac{1}{\lambda}\mathcal{I}: A^{\ast}
\rightarrow A$ gives an isomorphism of anti-Leibniz bialgebras from $(A^{\ast}, \cdot_{r},
\Delta_{A^{\ast}})$ to $(A, \cdot_{P}, \Delta_{\mathcal{I}})$, where $\Delta_{A^{\ast}}$
is the dual of product in $(A, \cdot)$, i.e., $\Delta_{A^{\ast}}^{\ast}
(a_{1}\otimes a_{2})=a_{1}a_{2}$, for any $a_{1}, a_{2}\in A$.
\end{cor}

\begin{proof}
First, we show that $\frac{1}{\lambda}\mathcal{I}$ is an anti-Leibniz algebra
isomorphism from $(A^{\ast}, \cdot_{r})$ to $(A, \cdot_{P})$. In fact, for any
$\xi_{1}, \xi_{2}\in A^{\ast}$, taking $a_{1}=\mathcal{I}(\xi_{1})\in A$ and
$a_{2}=\mathcal{I}(\xi_{2})\in A$, by Eq. \eqref{I}, we have
$$
\mbox{$\frac{1}{\lambda}$}\mathcal{I}(\xi_{1}\cdot_{r}\xi_{2})
=\mbox{$\frac{1}{\lambda^{2}}$}\big(P(\mathcal{I}(\xi_{1}))\mathcal{I}(\xi_{2})
+\mathcal{I}(\xi_{1})P(\mathcal{I}(\xi_{2}))+\lambda\mathcal{I}(\xi_{1})\mathcal{I}(\xi_{2})\big)
=\mbox{$\frac{1}{\lambda}$}\mathcal{I}(\xi_{1})\cdot_{P}
\mbox{$\frac{1}{\lambda}$}\mathcal{I}(\xi_{2}).
$$
So, $\frac{1}{\lambda}\mathcal{I}$ is an anti-Leibniz algebra isomorphism.

Second, denote the dual of $\Delta_{\mathcal{I}}$ by $\cdot_{\mathcal{I}}$. Since
$(\frac{1}{\lambda}\mathcal{I})^{\ast}=-\frac{1}{\lambda}\mathcal{I}$, we get
$$
(\mbox{$\frac{1}{\lambda}$}\mathcal{I})^{\ast}(\xi_{1}\cdot_{\mathcal{I}}\xi_{2})
=\mbox{$\frac{1}{\lambda}$}\mathcal{I}(\xi_{1})\mbox{$\frac{1}{\lambda}$}\mathcal{I}(\xi_{2})
=(\mbox{$\frac{1}{\lambda}$}\mathcal{I})^{\ast}(\xi_{1})(\mbox{$\frac{1}{\lambda}$}
\mathcal{I})^{\ast}(\xi_{2}),
$$
which means $(\frac{1}{\lambda}\mathcal{I})^{\ast}:(A^{\ast}, \cdot_{\mathcal{I}})
\rightarrow(A, \cdot)$ is an isomorphism of anti-Leibniz algebras.
Hence, $\frac{1}{\lambda}\mathcal{I}: (A^{\ast}, \Delta_{A^{\ast}})\rightarrow
(A, \Delta_{\mathcal{I}})$ is an isomorphism of anti-Leibniz coalgebras.
Therefore, $\frac{1}{\lambda}\mathcal{I}$ is an isomorphism of anti-Leibniz
bialgebras and $(A, \cdot_{P}, \Delta_{\mathcal{I}})$ is an anti-Leibniz bialgebra.
\end{proof}

\subsection{Anti-Leibniz bialgebras via Leibniz bialgebras}\label{subsec:tensor}
Here, we constrict the anti-Leibniz bialgebras form Leibniz bialgebras and anti-commutative
anti-associative algebras. By using this method, we can construct a large number of
examples for anti-Leibniz algebra and anti-Leibniz bialgebra.
Recall that a {\it(left) Leibniz algebra} is a pair $(L, [-,-])$, where $L$ is a vector space
and $[-,-]: L\otimes L\rightarrow L$ is a multiplication such that
\begin{align*}
[x_{1}, [x_{2}, x_{3}]]=[[x_{1}, x_{2}], x_{3}]+[x_{2}, [x_{1}, x_{3}]],
\end{align*}
for any $x_{1}, x_{2}, x_{3}\in L$. An {\it anti-commutative anti-associative algebra}
$(B, \cdot)$ is a vector space $B$ with a bilinear operations $\cdot: B\otimes B
\rightarrow B$, $(b_{1}, b_{2})\mapsto b_{1}b_{2}:=b_{1}\cdot b_{2}$, such that
\begin{align*}
b_{1}b_{2}=-b_{2}b_{1},\qquad \mbox{ and }\qquad b_{1}(b_{2}b_{3})=-(b_{1}b_{2})b_{3},
\end{align*}
for any $b_{1}, b_{2}, b_{3}\in B$. The anti-commutative anti-associative algebra is also
called the dual mock-Lie algebra, since the operad of anti-commutative anti-associative
algebra is the Koszul dual to the mock-Lie operad.

\begin{lem}\label{lem:Leib-ant}
Let $(L, [-,-])$ be a Leibniz algebra, $(B, \cdot)$ be an anti-commutative anti-associative
algebra. Define a binary operation on $A:=L\otimes B$ by
$$
(x_{1}\otimes b_{1})\circ(x_{2}\otimes b_{2})=[x_{1}, x_{2}]\otimes (b_{1}b_{2}),
$$
for any $x_{1}, x_{2}\in L$ and $b_{1}, b_{2}\in B$. Then $(A, \circ)$ is an
anti-Leibniz algebra, which is called the anti-Leibniz algebra induced by $(L, [-,-])$
and $(B, \cdot)$.
\end{lem}

\begin{proof}
For any $x_{1}, x_{2}, x_{3}\in L$ and $b_{1}, b_{2}, b_{3}\in B$, note that
\begin{align*}
&\;(x_{1}\otimes b_{1})\circ\big((x_{2}\otimes b_{2})\circ(x_{3}\otimes b_{3})\big)
+\big((x_{1}\otimes b_{1})\circ(x_{2}\otimes b_{2})\big)\circ(x_{3}\otimes b_{3})\\[-1mm]
&\qquad+(x_{2}\otimes b_{2})\circ\big((x_{1}\otimes b_{1})\circ(x_{3}\otimes b_{3})\big)\\
=&\; [x_{1}, [x_{2}, x_{3}]]\otimes(b_{1}(b_{2}b_{3}))
+[[x_{1}, x_{2}], x_{3}]\otimes((b_{1}b_{2})b_{3})
+[x_{2}, [x_{1}, x_{3}]]\otimes(b_{2}(b_{1}b_{3}))\\
=&\; \Big([x_{1}, [x_{2}, x_{3}]]-[[x_{1}, x_{2}], x_{3}]-[x_{2}, [x_{1}, x_{3}]]\Big)
\otimes(b_{1}(b_{2}b_{3}))\\
=&\; 0,
\end{align*}
we get $(A, \circ)$ is an anti-Leibniz algebra.
\end{proof}

For the classification of anti-Leibniz algebras, we still know very little.
But for the classification of Leibniz algebras, some results have been obtained \cite{AOR}.
Recall that the non trivial anti-commutative anti-associative algebra of dimensional two
is only isomorphic to the algebra $(B=\Bbbk\{e_{1}, e_{2}\}, \cdot)$, where the nonzero
multiplication is given by $e_{1}e_{1}=e_{2}$.
Two dimensional Leibniz algebras have been classified by Cuvier \cite{Cuv}.
More precisely, the non trivial Leibniz algebraic structure on $L=\Bbbk\{x_{1}, x_{2}\}$
is only isomorphic to the algebra with nonzero multiplication as follows:
$$
L_{1}: [x_{1}, x_{1}]=x_{2};\quad\qquad
L_{2}: [x_{1}, x_{2}]=-[x_{2}, x_{1}]=x_{2};\quad\qquad
L_{1}: [x_{1}, x_{2}]=x_{1}=[x_{2}, x_{2}].
$$
Thus, by Lemma \ref{lem:Leib-ant}, we get some $4$-dimensional anti-Leibniz algebra
as follows: $A=\Bbbk\{a_{1}, a_{2}, a_{3}, a_{4}\}$ with nonzero multiplication
$$
A_{1}: a_{1}\circ a_{1}=a_{2};\quad\qquad
A_{2}: a_{1}\circ a_{2}=-a_{2}\circ a_{1}=a_{3};\quad\qquad
A_{1}: a_{1}\circ a_{2}=a_{3}=a_{2}\circ a_{2}.
$$

Let $B$ be a vector space and $\Delta: B\rightarrow B\otimes B$ be a linear map.
Then $(B, \Delta)$ is called an {\it anti-cocommutative anti-coassociative coalgebra} if
$$
\tau\Delta=-\Delta,\qquad \mbox{ and }\qquad(\id\otimes\Delta)\Delta=-(\Delta\otimes\id)\Delta.
$$
Clearly, the notion of anti-cocommutative anti-coassociative coalgebras is the dualization
of the notion of anti-commutative anti-associative algebras, that is, $(B, \Delta)$ is an
anti-cocommutative anti-coassociative coalgebra if and only if $(B^{\ast}, \Delta^{\ast})$
is an anti-commutative anti-associative algebra. Recall that a {\it Leibniz coalgebra}
is a pair $(L, \delta)$, where $L$ is a vector space and $\delta:
L\rightarrow L\otimes L$ is a comultiplication such that
$$
(\delta\otimes\id)\delta(x)+(\tau\otimes\id)(\id\otimes\delta)\delta(x)
-(\id\otimes\delta)\delta(x)=0,
$$
for any $x\in L$. Now, we give the dual version of Lemma \ref{lem:Leib-ant}.

\begin{lem}\label{lem:Leibco-ant}
Let $(L, \delta)$ be a Leibniz coalgebra, $(B, \Delta)$ be an anti-cocommutative
anti-coassociative coalgebra. Define $A:=L\otimes B$ and a linear map $\tilde{\Delta}:
A\rightarrow A\otimes A$ by
$$
\tilde{\Delta}(x\otimes b)=\delta(x)\bullet\Delta(b)=\sum_{(x)}\sum_{(b)}
(x_{(1)}\otimes b_{(1)})\otimes(x_{(2)}\otimes b_{(2)}),
$$
for any $x\in L$ and $b\in B$, where $\Delta(b)=\sum_{(b)}b_{(1)}\otimes b_{(2)}$
and $\delta(x)=\sum_{(x)}x_{(1)}\otimes x_{(2)}$ in the Sweedler notation.
Then $(A, \tilde{\Delta})$ is an anti-Leibniz algebra.
\end{lem}

\begin{proof}
For any $x^{i}_{1}, x^{i}_{2}, x^{i}_{3}\in L$ and $b^{j}_{1}, b^{j}_{2}, b^{j}_{3}\in B$,
we denote
$$
\Big(\sum_{i}x^{i}_{1}\otimes x^{i}_{2}\otimes x^{i}_{3}\Big)\bullet
\Big(\sum_{j}b^{j}_{1}\otimes b^{j}_{2}\otimes b^{j}_{3}\Big)
=\sum_{i,j}(x^{i}_{1}\otimes b^{j}_{1})\otimes(x^{i}_{2}\otimes b^{j}_{2})
\otimes(x^{i}_{3}\otimes b^{j}_{3}).
$$
Then, for any $x\in L$ and $b\in B$, we have
\begin{align*}
&\;(\tilde{\Delta}\otimes\id)\tilde{\Delta}(x\otimes b)
+(\id\otimes\tilde{\Delta})\tilde{\Delta}(x\otimes b)
+(\tau\otimes\id)(\id\otimes\tilde{\Delta})\tilde{\Delta}(x\otimes b)\\
=&\;(\delta\otimes\id)\delta(x)\bullet(\Delta\otimes\id)\Delta(b)
+(\id\otimes\delta)\delta(x)\bullet(\id\otimes\Delta)\Delta(b)
+(\tau\otimes\id)(\id\otimes\delta)\delta(x)\bullet(\tau\otimes\id)(\id\otimes\Delta)\Delta(b)\\
=&\;\Big((\delta\otimes\id)\delta(x)+(\tau\otimes\id)(\id\otimes\delta)\delta(x)
-(\id\otimes\delta)\delta(x)\Big)\bullet(\Delta\otimes\id)\Delta(b)\\
=&\;0.
\end{align*}
Thus, $(A, \tilde{\Delta})$ is an anti-Leibniz coalgebra.
\end{proof}

Let $(B, \cdot)$ be an anti-commutative anti-associative algebra. A bilinear form
$\varpi(-, -)$ on $(B, \cdot)$ is called {\it nondegenerate} if
$\varpi(b_{1}, b_{2})=0$ for any $b_{2}\in B$, then $b_{1}=0$;
is called {\it symmetric} if $\varpi(b_{1}, b_{2})=\varpi(b_{2}, b_{1})$;
is called {\it invariant} if $\varpi(b_{1}b_{2},\, b_{3})=\varpi(b_{1},\, b_{2}b_{3})$,
for any $b_{1}, b_{2}, b_{3}\in B$.
A {\it quadratic anti-commutative anti-associative algebra}, denoted by
$(B, \varpi)$, is an anti-commutative anti-associative algebra $(B, \cdot)$
together with a symmetric invariant nondegenerate bilinear form.

\begin{lem}\label{lem:quad-dual}
Let $(B, \varpi)$ be a quadratic anti-commutative anti-associative algebra.
Suppose $\Delta_{\varpi}: B\rightarrow B\otimes B$ be the dual of the multiplication
of $(B, \varpi)$ under the bilinear form $\varpi(-,-)$, i,e.,
\begin{align*}
\varpi(\Delta_{\varpi}(b_{1}),\; b_{2}\otimes b_{3})=\varpi(b_{1},\; b_{2}b_{3}),
\end{align*}
for any $b_{1}, b_{2}, b_{3}\in B$, where $\varpi(\Delta_{\varpi}(b_{1}),\; b_{2}
\otimes b_{3})=\sum_{(b_{1})}\varpi((b_{1})_{(1)}, b_{2})\varpi((b_{1})_{(2)}, b_{3})$
if $\Delta_{\varpi}(b_{1})=\sum_{(b_{1})}(b_{1})_{(1)}$ $\otimes(b_{1})_{(2)}$.
Then $(B, \Delta_{\varpi})$ is an anti-cocommutative anti-coassociative coalgebra.
\end{lem}

\begin{proof}
For any $b\in B$ and $b_{1}\otimes b_{2}\otimes b_{3}\in B\otimes B\otimes B$,
by the definition of $\Delta_{\varpi}$, we get $\varpi(\tau\Delta_{\varpi}(b),\;
b_{1}\otimes b_{2})=\varpi(b,\; b_{2}b_{1})=-\varpi(b,\; b_{1}b_{2})=
\varpi(-\Delta_{\varpi}(b),\; b_{1}\otimes b_{2})$ and
\begin{align*}
\varpi\big((\Delta_{\varpi}\otimes\id)\Delta_{\varpi}(b),\ \
b_{1}\otimes b_{2}\otimes b_{3}\big)
&=\varpi\big(b,\; (b_{1}b_{2})b_{3})\big)\\
&=-\varpi\big(b,\; b_{1}(b_{2}b_{3})\big)\\
&=\varpi\big(-(\id\otimes\Delta_{\varpi})\Delta_{\varpi}(b),\ \
b_{1}\otimes b_{2}\otimes b_{3}\big).
\end{align*}
That is, $(\Delta_{\varpi}\otimes\id)\Delta_{\varpi}
=-(\id\otimes\Delta_{\varpi})\Delta_{\varpi}$ and $\Delta_{\varpi}=-\tau\Delta_{\varpi}$.
Thus $(B, \Delta_{\varpi})$ is an anti-cocommutative anti-coassociative coalgebra.
\end{proof}

\begin{ex}\label{ex:quad-coa}
Consider the $2$-dimensional anti-commutative anti-associative algebra $(B=\Bbbk\{e_{1},
e_{2}\},\; \cdot)$, where the nonzero multiplication is given by $e_{1}e_{1}=e_{2}$.
Define a bilinear form $\varpi(-,-)$ on $(B, \cdot)$ by $\varpi(e_{1}, e_{2})=
\varpi(e_{2}, e_{1})=1$ and $\varpi(e_{1}, e_{1})=\varpi(e_{2}, e_{2})=0$. Then
$(B, \varpi)$ is a quadratic anti-commutative anti-associative algebra. Thus, we get
an anti-cocommutative anti-coassociative coalgebra $(B, \Delta_{\varpi})$, where
$\Delta_{\varpi}(e_{1})=e_{2}\otimes e_{2}$ and $\Delta_{\varpi}(e_{2})=0$.
\end{ex}

Recall that a {\it Leibniz bialgebra} is a triple $(L, [-,-], \delta)$, where
$(L, [-,-])$ is a Leibniz algebra, $(L, \delta)$ is a Leibniz coalgebra and the
following equations hold:
\begin{align*}
&\qquad\qquad\qquad\qquad\tau(\fr_{L}(y)\otimes\id)\delta(x)=(\fr_{L}(x)\otimes\id)\delta(y),\\
&\delta([x, y])=\big(\id\otimes\,\fr_{L}(y)-(\fl_{L}+\fr_{L})(y)\otimes\id\big)
(\id\otimes\id+\tau)\delta(x)+(\id\otimes\,\fl_{L}(x)+\fl_{L}(x)\otimes\id)\delta(y),
\end{align*}
for all $x, y\in L$.

\begin{thm}\label{thm:Lei-antLei}
Let $(L, [-,-], \delta)$ be a Leibniz bialgebra and $(B, \varpi)$ be a quadratic
anti-commutative anti-associative algebra. Suppose $(A=L\otimes B, \circ)$ is the anti-Leibniz
algebra induced by $(L, [-,-])$ by $(B, \cdot)$, $\Delta_{\varpi}: B\rightarrow B\otimes B$
is the dual of the multiplication of $(B, \varpi)$ under the bilinear form $\varpi(-,-)$,
and $\tilde{\Delta}: A\rightarrow A\otimes A$ is the linear map defined by
$\tilde{\Delta}(x\otimes b)=\delta(x)\bullet\Delta_{\varpi}(b)$ for $x\in L$, $b\in B$.
Then $(A, \circ, \tilde{\Delta})$ is an anti-Leibniz bialgebra, which is called the
anti-Leibniz bialgebra induced by $(L, [-,-], \delta)$ and $(B, \varpi)$.
\end{thm}

\begin{proof}
Since $\varpi(-, -)$ is an invariant bilinear form, for any $b, b', c, d\in B$, we have
\begin{align*}
\varpi\Big(\sum_{(b)}b_{(1)}b'\otimes b_{(2)},\; c\otimes d\Big)
&=\varpi\Big(\sum_{(b)}b_{(1)}\otimes b_{(1)},\;
b'c\otimes d\Big)\\[-2mm]
&=\varpi(b,\; (b'c)d)\\
&=-\varpi(\Delta_{\varpi}(bb'),\; c\otimes d).
\end{align*}
That is to say, $\Delta_{\varpi}(bb')=-\sum_{(b)}b_{(1)}b'\otimes b_{(2)}$. Moreover, since
$(B, \cdot)$ is anti-commutative and $\Delta_{\varpi}$ is anti-cocommutative, we get
$-\Delta_{\varpi}(b'b)=\Delta_{\varpi}(bb')=\sum_{(b')}b'_{(1)}b\otimes b'_{(2)}
=\sum_{(b)}b_{(2)}\otimes b_{(1)}b'=\sum_{(b)}b'b_{(1)}\otimes b_{(2)}
=-\sum_{(b)}b_{(1)}b'\otimes b_{(2)}=\sum_{(b)}b_{(2)}\otimes b_{(1)}b'
=-\sum_{(b)}b'b_{(2)}\otimes b_{(1)}=\sum_{(b)}b_{(2)}b'\otimes b_{(1)}
=-\sum_{(b')}b'_{(1)}\otimes bb'_{(2)}=-\sum_{(b')}bb'_{(1)}\otimes b'_{(2)}$.
Thus, for any $x\otimes b$, $x'\otimes b'\in A$,
\begin{align*}
&\;\big(\fr_{A}(x\otimes b)\otimes\id\big)\tilde{\Delta}(x'\otimes b')
-\hat{\tau}\big(\big(\fr_{A}(x'\otimes b')\otimes\id\big)\tilde{\Delta}
(x\otimes b)\big)\\
=&\;\sum_{(x')}\sum_{(b')}([x'_{(1)}, x]\otimes x'_{(2)})\bullet(b'_{(1)}b\otimes b'_{(2)})
-\sum_{(x)}\sum_{(b)}(x_{(2)}\otimes[x_{(1)}, x'])\bullet(b_{(2)}\otimes b_{(1)}b')\\
=&\;\Big((\fr_{L}(x)\otimes\id)\delta(x')
-\tau\big((\fr_{L}(x')\otimes\id)\delta(x)\big)\Big)\bullet\Delta_{\varpi}(bb')\\
=&\;0,
\end{align*}
and
\begin{align*}
&\;\tilde{\Delta}((x\otimes b)\circ(a'\otimes c'))+\Big(\big(\id\otimes
\fr_{\la}(x'\otimes b')+(\fl_{A}-\fr_{A})(x'\otimes b')\otimes\id\big)
(\id\otimes\id-\tau)\Big)\tilde{\Delta}(x\otimes b)\\[-2mm]
&\qquad+\Big(\id\otimes\fl_{A}(x\otimes b)
+\fl_{A}(x\otimes b)\otimes\id\Big)\tilde{\Delta}(x'\otimes b')\\
=&\;\delta([x, x'])\bullet\Delta_{\varpi}(bb')
+\sum_{(x)}\sum_{(b)}(x_{(1)}\otimes[x_{(2)}, x'])\bullet(b_{(1)}\otimes b_{(2)}b')
+\sum_{(x)}\sum_{(b)}([x', x_{(1)}]\otimes x_{(2)})\bullet(b'b_{(1)}\otimes b_{(2)})\\[-2mm]
&\qquad-\sum_{(x)}\sum_{(b)}([x_{(1)}, x']\otimes x_{(2)})\bullet(b_{(1)}b'\otimes b_{(2)})
-\sum_{(x)}\sum_{(b)}(x_{(2)}\otimes[x_{(1)}, x'])\bullet(b_{(2)}\otimes b_{(1)}b')\\[-2mm]
&\qquad-\sum_{(x)}\sum_{(b)}([x', x_{(2)}]\otimes x_{(1)})\bullet(b'b_{(2)}\otimes b_{(1)})
+\sum_{(x)}\sum_{(b)}([x_{(2)}, x']\otimes x_{(1)})\bullet(b_{(2)}b'\otimes b_{(1)})\\[-2mm]
&\qquad+\sum_{(x')}\sum_{(b')}(x'_{(1)}\otimes[x, x'_{(2)}])\bullet(b'_{(1)}\otimes bb'_{(2)})
+\sum_{(x')}\sum_{(b')}([x, x'_{(1)}]\otimes x'_{(2)})\bullet(bb'_{(1)}\otimes b'_{(2)})\\
=&\;\Big(\Delta([x, x'])-\big(\big(\id\otimes\,\fr_{L}(x')+\fr_{L}(x')\otimes\id
+\fl_{L}(x')\otimes\id\big)(\id\otimes\id+\tau)\big)\delta(x)\Big)
\bullet\Delta_{\varpi}(bb')\\[-2mm]
&\qquad-\Big(\big(\id\otimes\,\fl_{L}(x)+\fl_{L}(x)\otimes\id\big)
\delta(x')\Big)\bullet\Delta_{\varpi}(bb')\\
=&\;0,
\end{align*}
since $(L, [-,-], \delta)$ is a Leibniz bialgebra.
By the definition, we get that $(A, \circ, \tilde{\Delta})$ is a completed
anti-Leibniz bialgebra.
\end{proof}

\begin{ex}\label{ex:ten-bial}
In \cite{LMW}, the authors have given some examples of triangular Leibniz bialgebras.
Here we using Theorem {\rm \ref{thm:Lei-antLei}} to give some examples of anti-Leibniz
bialgebras. We fix $(B, \cdot)$ as the $2$-dimensional anti-commutative anti-associative
algebra given in Example {\rm \ref{ex:quad-coa}}, i.e., $B=\Bbbk\{e_{1}, e_{2}\}$ and
$e_{1}e_{1}=e_{2}$ By the bilinear form $\varpi(-,-)$ on $(B, \cdot)$ defined in Example
{\rm \ref{ex:quad-coa}}, we get an anti-cocommutative anti-coassociative coalgebra
$(B, \Delta_{\varpi})$, where $\Delta_{\varpi}(e_{1})=e_{2}\otimes e_{2}$.

$(I)$ Let $(L, [-,-], \delta)$ be a $2$-dimensional Leibniz bialgebra with basis
$\{x_{1}, x_{2}\}$, and $(A=L\otimes B, \circ, \tilde{\Delta})$ be the induced
$4$-dimensional anti-Leibniz bialgebra with basis $\{a_{1}, a_{2}, a_{3}, a_{4}\}$.
\begin{enumerate}\itemsep=0pt
\item[$(i)$] If $[x_{1}, x_{1}]=x_{2}$, $\delta(x_{1})=k(x_{2}\otimes x_{2})$, then
     $a_{1}\circ a_{1}=a_{2}$ and $\tilde{\Delta}(a_{1})=k(a_{2}\otimes a_{2})$, where
     $k\in\Bbbk$. That is, the Leibniz bialgebra $(L, [-,-], \delta)$ is also an
     anti-Leibniz bialgebra.
\item[$(ii)$] If $[x_{1}, x_{2}]=x_{1}=[x_{2}, x_{1}]$, $\delta(x_{1})=\delta(x_{2})=
     k(x_{1}\otimes x_{2}-x_{2}\otimes x_{1})$, then $a_{1}\circ a_{2}=a_{3}=a_{2}\circ
     a_{2}$ and $\tilde{\Delta}(a_{1})=\tilde{\Delta}(a_{1})=k(a_{2}\otimes a_{4}
     -a_{4}\otimes a_{2})$, where $k\in\Bbbk$.
\end{enumerate}

$(II)$ Let $(L, [-,-], \delta)$ be a $3$-dimensional Leibniz bialgebra with basis
$\{x_{1}, x_{2}, x_{3}\}$, and $(A=L\otimes B, \circ, \tilde{\Delta})$ be the induced
$6$-dimensional anti-Leibniz bialgebra with basis $\{a_{1}, a_{2}, a_{3},
a_{4}, a_{5}, a_{6}\}$.
\begin{enumerate}\itemsep=0pt
\item[$(i)$] If $[x_{1}, x_{3}]=x_{1}+x_{2}$, $[x_{3}, x_{3}]=x_{1}$,
     $\delta(x_{3})=k(x_{1}\otimes x_{1}+x_{2}\otimes x_{1})+l(x_{1}\otimes x_{2}
     +x_{2}\otimes x_{2})$, then $a_{1}\circ a_{5}=a_{2}+a_{4}$, $a_{5}\circ a_{5}=a_{2}$
     and $\tilde{\Delta}(a_{5})=k(a_{2}\otimes a_{2}+a_{4}\otimes a_{2})+l(a_{1}\otimes a_{4}
     +a_{4}\otimes a_{4})$, where $k, l\in\Bbbk$.
\item[$(ii)$] If $[x_{2}, x_{3}]=x_{2}=-[x_{3}, x_{2}]$, $[x_{3}, x_{3}]=x_{1}$,
     $\delta(x_{2})=k(x_{1}\otimes x_{2})$, $\delta(x_{3})=k(x_{1}\otimes x_{1})
     -l(x_{1}\otimes x_{2})$, then $a_{3}\circ a_{5}=a_{4}=-a_{5}\circ a_{3}$,
     $a_{5}\circ a_{5}=a_{3}$, $\tilde{\Delta}(a_{3})=k(a_{2}\otimes a_{4})$ and
     $\tilde{\Delta}(a_{5})=k(a_{2}\otimes a_{2})-l(a_{2}\otimes a_{4})$, where
     $k, l\in\Bbbk$.
\item[$(iii)$] If $[x_{2}, x_{2}]=x_{1}=[x_{3}, x_{3}]$, $\delta(x_{2})=k(x_{1}\otimes
     x_{1})$, then $a_{3}\circ a_{3}=a_{2}=a_{5}\circ a_{5}$ and $\tilde{\Delta}(a_{3})=
     k(a_{2}\otimes a_{2})$, where $k\in\Bbbk$.
\end{enumerate}
\end{ex}

\section{Infinite-dimensional anti-Leibniz bialgebras}\label{sec:infin}
In this section, we constrict infinite-dimensional anti-Leibniz bialgebras
form finite-dimensional anti-Leibniz bialgebras by the completed tensor product.
In this section, a commutative algebra means a commutative associative algebra, and
a cocommutative coalgebra means a cocommutative coassociative coalgebra for simply.

\begin{defi}\label{def:zgrad-alg}
A {\rm $\bz$-graded commutative algebra} (resp. a {\rm $\bz$-graded
anti-Leibniz algebra}) is a commutative algebra $(C, \cdot)$
(resp. an anti-Leibniz algebra $(A, \cdot)$) with a linear decomposition
$C=\oplus_{i\in\bz}C_{i}$ (resp. $A=\oplus_{i\in\bz}A_{i}$) such that each $C_{i}$
(resp. $A_{i}$) is finite-dimensional and $C_{i}\cdot C_{j}\subseteq C_{i+j}$
(resp. $A_{i}\cdot A_{j}\subseteq A_{i+j}$) for all $i, j\in\bz$.
\end{defi}

Clearly, the Laurent polynomial algebra $\Bbbk[\kt, \kt^{-1}]$ is a $\bz$-graded
commutative algebra. We now consider a tensor product of anti-Leibniz
algebras and $\bz$-graded commutative algebras.

\begin{lem}\label{lem:aff-anL}
Let $(A, \cdot)$ be a finite-dimensional anti-Leibniz algebra, $(C=\oplus_{i\in\bz}C_{i},
\cdot)$ be a $\bz$-graded commutative algebra. Define a binary operation on
$\la:=A\otimes C$ by
$$
(a_{1}\otimes c_{1})\diamond(a_{2}\otimes c_{2})=(a_{1}a_{2})\otimes(c_{1}c_{2}),
$$
for any $a_{1}, a_{2}\in A$ and $c_{1}, c_{2}\in C$. Then $(\la, \diamond)$ is a
$\bz$-graded anti-Leibniz algebra, which is called {\rm affine anti-Leibniz
algebra} from $(A, \cdot)$ by $(C=\oplus_{i\in\bz}C_{i}, \cdot)$.
Moreover, if $(C=\oplus_{i\in\bz}C_{i}, \cdot)=\Bbbk[\kt, \kt^{-1}]$ is the Laurent
polynomial algebra, then $(\la, \diamond)$ is a $\bz$-graded anti-Leibniz
algebra if and only if $(A, \cdot)$ is an anti-Leibniz algebra.
\end{lem}

\begin{proof}
It is straightforward.
\end{proof}

To carry out the affine anti-Leibniz coalgebra, we need to extend the codomain
of the coproduct $\Delta$ to allow infinite sums.
Let $U=\oplus_{i\in\bz}U_{i}$ and $V=\oplus_{j\in\bz}V_{j}$ be $\bz$-graded vector spaces.
We call the {\it completed tensor product} of $U$ and $V$ to be the vector space
$$
U\,\hat{\otimes}\,V:=\prod_{i,j\in\bz}U_{i}\otimes V_{j}.
$$
If $U$ and $V$ are finite-dimensional, then $U\,\hat{\otimes}\,V$ is just the usual
tensor product $U\otimes V$. In general, an element in $U\,\hat{\otimes}\,V$ is an
infinite formal sum $\sum_{i,j\in\bz}X_{ij} $ with $X_{ij}\in U_{i}\otimes V_{j}$.
So $X_{ij}=\sum_{\alpha} u_{i, \alpha}\otimes v_{j, \alpha}$ for pure tensors
$u_{i, \alpha}\otimes v_{j, \alpha}\in U_{i}\otimes V_{j}$ with $\alpha$ in a finite
index set. Thus a general term of $U\,\hat{\otimes}\,V$ is a possibly infinite sum
$\sum_{i,j,\alpha}u_{i\alpha}\otimes v_{j\alpha}$, where $i, j\in\bz$ and $\alpha$
is in a finite index set (which might depend on $i, j$). With these notations, for
linear maps $f: U\rightarrow U'$ and $g: V\rightarrow V'$, define
$$
f\,\hat{\otimes}\,g: U\,\hat{\otimes}\,V\rightarrow U'\,\hat{\otimes}\,V',
\qquad \sum_{i,j,\alpha}u_{i,\alpha}\otimes v_{j, \alpha}\mapsto
\sum_{i,j,\alpha} f(u_{i, \alpha})\otimes g(v_{j, \alpha}).
$$
Also the twist map $\tau$ has its completion
$$
\hat{\tau}: V\,\hat{\otimes}\,V\rightarrow V\,\hat{\otimes}\,V, \qquad
\sum_{i,j,\alpha}u_{i, \alpha}\otimes v_{j, \alpha}\mapsto
\sum_{i,j,\alpha}v_{j, \alpha}\otimes u_{i, \alpha}.
$$
Finally, we define a (completed) coproduct to be a linear map
$$
\hat{\Delta}: V\rightarrow V\,\hat{\otimes}\,V, \qquad
\hat{\Delta}(v):=\sum_{i, j, \alpha}v_{1, i, \alpha}\otimes v_{2, j, \alpha}.
$$
Then we have the well-defined map
$$
(\hat{\Delta}\,\hat{\otimes}\,\id)\hat{\Delta}(v)=(\hat{\Delta}\,\hat{\otimes}\,\id)
\Big(\sum_{i,j,\alpha}v_{1, i, \alpha}\otimes v_{2, j, \alpha}\Big)
:=\sum_{i,j,\alpha}\hat{\Delta}(v_{1, i, \alpha})\otimes v_{2, j, \alpha}
\in V\,\hat{\otimes}\,V\,\hat{\otimes}\,V.
$$

\begin{defi}\label{def:anL-coa}
$(i)$ A {\rm completed cocommutative coalgebra} is a pair
$(C, \bar{\Delta})$ where $C=\oplus_{i\in\bz}C_{i}$ is a $\bz$-graded vector space and
$\bar{\Delta}: C\rightarrow C\,\hat{\otimes}\, C$ is a linear map satisfying
\begin{align*}
\hat{\tau}\bar{\Delta}=\bar{\Delta},\qquad\qquad (\bar{\Delta}\,\hat{\otimes}\,\id)
\bar{\Delta}=(\id\,\hat{\otimes}\,\bar{\Delta})\bar{\Delta}.
\end{align*}
$(ii)$ A {\rm completed anti-Leibniz coalgebra} is a pair $(A, \hat{\Delta})$ where
$A=\oplus_{i\in\bz}A_{i}$ is a $\bz$-graded vector space and $\hat{\Delta}: A\rightarrow
A\,\hat{\otimes}\, A$ is a linear map satisfying
$$
(\hat{\Delta}\otimes\id)\hat{\Delta}+(\id\otimes\hat{\Delta})\hat{\Delta}
+(\hat{\tau}\otimes\id)(\id\otimes\hat{\Delta})\hat{\Delta}=0.
$$
\end{defi}

\begin{ex}\label{ex:lau-coa}
Consider the Laurent polynomial algebra $\Bbbk[\kt, \kt^{-1}]$, there is a naturally
completed cocommutative coalgebra structure on $\Bbbk[\kt, \kt^{-1}]$, which
is given by $\bar{\Delta}(\kt^{k})=\sum_{i\in\bz}\kt^{i}\otimes\kt^{k-i}$ for all $k\in\bz$.
This completed cocommutative coalgebra is called
{\rm Laurent polynomial coalgebra}.
\end{ex}

For completed anti-Leibniz coalgebra, we give the dual version
of Lemma \ref{lem:aff-anL}.

\begin{lem}\label{lem:canL-co}
Let $(A, \Delta)$ be a finite-dimensional anti-Leibniz coalgebra,
$(C, \bar{\Delta})$ be a completed cocommutative coalgebra and
$\la:=A\otimes C$. Define a linear map $\hat{\Delta}: \la\rightarrow\la\,
\hat{\otimes}\,\la$ by
\begin{align}
\hat{\Delta}(a\otimes c)=\Delta(a)\bullet\bar{\Delta}(c)=\sum_{(a)}\sum_{i,j,\alpha}
(a_{(1)}\otimes c_{1,i,\alpha})\otimes(a_{(2)}\otimes c_{2,j,\alpha}),        \label{coalg}
\end{align}
for any $a\in A$ and $c\in C$, where $\Delta(a)=\sum_{(a)}a_{(1)}\otimes a_{(2)}$
in the Sweedler notation and $\bar{\Delta}(c)=\sum _{i,j,\alpha}c_{1,i,\alpha}\otimes
c_{2, j, \alpha}$. Then $(\la, \hat{\Delta})$ is a completed anti-Leibniz coalgebra.

Furthermore, if $A$ is finite-dimensional, $(C, \bar{\Delta})$ is the Laurent polynomial
coalgebra $(\Bbbk[\kt, \kt^{-1}], \bar{\Delta})$, then $(\la, \hat{\Delta})$ is a completed
anti-Leibniz coalgebra if and only if $(A, \Delta)$ is an anti-Leibniz coalgebra.
\end{lem}

\begin{proof}
First, for any $\sum_{l}a'_{l}\otimes a''_{l}\otimes a'''_{l}\in A\otimes A\otimes A$ and
$\sum_{i,j,k,\alpha}c'_{i,\alpha}\otimes c''_{j,\alpha}\otimes c'''_{k,\alpha}\in
C\,\hat{\otimes}\,C\,\hat{\otimes}\,C$, we denote
$$
\Big(\sum_{l}a'_{l}\otimes a''_{l}\otimes a'''_{l}\Big)\bullet
\Big(\sum_{i,j,k,\alpha}c'_{i,\alpha}\otimes c''_{j,\alpha}\otimes c'''_{k,\alpha}\Big)
=\sum_{l}\sum_{i,j,k,\alpha}(a'_{l}\otimes c'_{i,\alpha})\otimes
(a''_{l}\otimes c''_{j,\alpha})\otimes(a'''_{l}\otimes c'''_{k,\alpha}).
$$
Using the above notations, for any $a\otimes c\in\la$, we have
\begin{align*}
&\;(\hat{\Delta}\,\hat{\otimes}\,\id)\hat{\Delta}(a\otimes c)
+(\id\,\hat{\otimes}\,\hat{\Delta})\hat{\Delta}(a\otimes c)
+(\hat{\tau}\,\hat{\otimes}\,\id)(\id\,\hat{\otimes}\,\hat{\Delta})\hat{\Delta}(a\otimes c)\\
=&\;(\Delta\otimes\id)\Delta(a)\bullet(\bar{\Delta}\,\hat{\otimes}\,\id)\bar{\Delta}(c)
+(\id\otimes\Delta)\Delta(a)\bullet(\id\,\hat{\otimes}\,\bar{\Delta})\bar{\Delta}(c)
+(\tau\otimes\id)(\id\otimes\Delta)\Delta(a)\bullet
(\hat{\tau}\,\hat{\otimes}\,\id)(\id\,\hat{\otimes}\,\bar{\Delta})\bar{\Delta}(c)\\
=&\;\Big((\Delta\otimes\id)\Delta(a)+(\id\otimes\Delta)\Delta(a)
+(\tau\otimes\id)(\id\otimes\Delta)\Delta(a)\Big)\bullet
(\bar{\Delta}\,\hat{\otimes}\,\id)\bar{\Delta}(c)\\
=&\;0.
\end{align*}
That is, $(\la, \hat{\Delta})$ is a completed anti-Leibniz coalgebra.

Second, if $(C, \bar{\Delta})=(\Bbbk[\kt, \kt^{-1}], \bar{\Delta})$ and
$(\la, \hat{\Delta})$ is a completed anti-Leibniz coalgebra constructed
as above, we abbreviated $a\otimes\kt^{k}$ as $a\kt^{k}$ for $a\in A$, then
\begin{align*}
(\hat{\Delta}\,\hat{\otimes}\,\id)\hat{\Delta}(a\kt^{k})
&=\sum_{(a)}\sum_{i,j}(a_{11}\kt^{j})\otimes
(a_{12}\kt^{i-j})\otimes(a_{2}\kt^{k-i}),\\[-1mm]
(\id\,\hat{\otimes}\,\hat{\Delta})\hat{\Delta}(a\kt^{k})
&=\sum_{(a)}\sum_{i,j}(a_{1}\kt^{i})\otimes
(a_{21}\kt^{j})\otimes(a_{22}\kt^{k-i-j}),\\[-1mm]
(\hat{\tau}\,\hat{\otimes}\,\id)(\id\,\hat{\otimes}\,\hat{\Delta})\hat{\Delta}(a\kt^{k})
&=\sum_{(a)}\sum_{i,j}(a_{21}\kt^{j})\otimes(a_{1}\kt^{i})
\otimes(a_{22}\kt^{k-i-j}).
\end{align*}
Let $i=j=0$. Comparing the coefficients of $1\otimes1\otimes\kt^{k}$, we obtain
$\sum_{(a)}a_{11}\otimes a_{12}\otimes a_{2}+a_{1}\otimes a_{21}\otimes a_{22}
+a_{21}\otimes a_{1}\otimes a_{22}$, i.e., $(\Delta\otimes\id)\Delta+(\id\otimes\Delta)\Delta
+(\tau\otimes\id)(\id\otimes\Delta)\Delta$. Thus $(A, \Delta)$ is an
anti-Leibniz coalgebra. The proof is finished.
\end{proof}

Finally, we extend the affine anti-Leibniz algebras and the affine
anti-Leibniz coalgebras to the affine anti-Leibniz bialgebra.

\begin{defi}\label{def:quad}
Let $\varpi(-, -)$ be a bilinear form on a $\bz$-graded commutative
algebra $(C=\oplus_{i\in\bz}C_{i}, \cdot)$. Then
\begin{itemize}
\item[$(i)$] $\varpi(-, -)$ is called {\rm invariant}, if $\varpi(c_{1}c_{2},\, c_{3})
     =\varpi(c_{1},\, c_{2}c_{3})$, for any $c_{1}, c_{2}, c_{3}\in C$;
\item[$(ii)$] $\varpi(-, -)$ is called {\rm graded}, if there exists
     $m\in\bz$ such that $\varpi(C_{i}, C_{j})=0$ when $i+j+m\neq0$.
\end{itemize}
A {\rm quadratic $\bz$-graded commutative algebra}, denoted by
$(C=\oplus_{i\in\bz}C_{i},\; \varpi)$, is a $\bz$-graded commutative
algebra together with a symmetric invariant nondegenerate graded bilinear form.
\end{defi}

Clearly, in Definition \ref{def:quad}, if $C=C_{0}$, then a quadratic $\bz$-graded
commutative algebra is just a quadratic commutative algebra.

\begin{ex}\label{ex:lau-quad}
Consider the Laurent polynomial algebra $\Bbbk[\kt, \kt^{-1}]$, there is a natural
bilinear form $\varpi(-,-)$ on $\Bbbk[\kt, \kt^{-1}]$ which is given by
$\varpi(\kt^{i}, \kt^{j})=\delta_{i, j}$, for $i, j\in\bz$, where $\delta_{i, j}$
is the Kronecker delta. It is easy to see that this bilinear form is a symmetric
invariant nondegenerate graded bilinear form. Thus, $(\Bbbk[\kt, \kt^{-1}], \varpi)$
is a quadratic $\bz$-graded commutative algebra, and it is called a
{\rm quadratic $\bz$-graded Laurent polynomial algebra}.
\end{ex}

For a quadratic $\bz$-graded commutative algebra $(C=\oplus_{i\in\bz}C_{i},\;
\varpi)$, the nondegenerate symmetric bilinear form $\varpi(-,-)$ induces bilinear forms
$$
(\underbrace{C\,\hat{\otimes}\,C\,\hat{\otimes}\,\cdots\,\hat{\otimes}\,C}_{n\text{-fold}})
\otimes(\underbrace{C\otimes C\otimes\cdots\otimes C}_{n\text{-fold}})\longrightarrow\Bbbk,
$$
for all $n\geq2$, which is denoted by $\bar{\varpi}(-,-)$ and defined by
$$
\bar{\varpi}\Big(\sum_{i_{1},\cdots,i_{n},\alpha} c'_{1, i_{1}, \alpha}
\otimes\cdots\otimes c'_{n, i_{n}, \alpha},\quad c_{1}\otimes\cdots\otimes c_{n}\Big)
=\sum_{i_{1},\cdots,i_{n},\alpha}\prod_{j=1}^{n}
\varpi(c'_{j, i_{j}, \alpha},\; c_{j}).
$$
One can check that $\bar{\varpi}(-,-)$ is {\it left nondegenerate}, i.e., if
$$
\bar{\varpi}\Big(\sum_{i_{1},\cdots,i_{n},\alpha} c'_{1, i_{1}, \alpha}
\otimes\cdots\otimes c'_{n, i_{n}, \alpha},\quad c_{1}\otimes\cdots\otimes c_{n}\Big)
=\bar{\varpi}\Big(\sum_{i_{1},\cdots,i_{n},\alpha} c''_{1, i_{1}, \alpha}
\otimes\cdots\otimes c''_{n, i_{n}, \alpha},\quad c_{1}\otimes\cdots\otimes c_{n}\Big)
$$
for any homogeneous elements $c_{1}, c_{2},\cdots, c_{n}\in C$, then
$$
\sum_{i_{1},\cdots,i_{n},\alpha} c'_{1, i_{1}, \alpha}
\otimes\cdots\otimes c'_{n, i_{n}, \alpha}
=\sum_{i_{1},\cdots,i_{n},\alpha} c''_{1, i_{1}, \alpha}
\otimes\cdots\otimes c''_{n, i_{n}, \alpha}.
$$

\begin{lem}\label{lem:qdual}
Let $(C=\oplus_{i\in\bz}C_{i},\; \varpi)$ be a quadratic $\bz$-graded commutative
associative algebra. Suppose $\bar{\Delta}: C\rightarrow C\,\hat{\otimes}\,C$
be the dual of the multiplication of $C$ under the left nondegenerate bilinear form
$\bar{\varpi}(-,-)$, i,e.,
\begin{align}
\bar{\varpi}(\bar{\Delta}(c_{1}),\; c_{2}\otimes c_{3})
=\varpi(c_{1},\; c_{2}c_{3}),                          \label{ccomu}
\end{align}
for any $c_{1}, c_{2}, c_{3}\in C$. Then $(C, \bar{\Delta})$ is a completed
cocommutative coalgebra.
\end{lem}

\begin{proof}
For any $c\in C$ and $c_{1}\otimes c_{2}\otimes c_{3}\in C_{i}\otimes C_{j}\otimes C_{k}$,
by the definition of $\bar{\Delta}$, we get
\begin{align*}
\bar{\varpi}\big((\bar{\Delta}\,\hat{\otimes}\,\id)(\bar{\Delta}(c)),\ \
c_{1}\otimes c_{2}\otimes c_{3}\big)
&=\varpi\big(c,\; (c_{1}c_{2})c_{3})\big)=\varpi\big(c,\; c_{1}(c_{2}c_{3})\big)\\
&=\bar{\varpi}\big((\id\,\hat{\otimes}\,\bar{\Delta})(\bar{\Delta}(c)),\ \
c_{1}\otimes c_{2}\otimes c_{3}\big).
\end{align*}
That is, $(\bar{\Delta}\,\hat{\otimes}\,\id)\bar{\Delta}
=(\id\,\hat{\otimes}\,\bar{\Delta})\bar{\Delta}$. Similarly, one can check that
$\bar{\Delta}=\hat{\tau}\bar{\Delta}$. Thus $(C, \bar{\Delta})$ is a completed
cocommutative coalgebra.
\end{proof}

\begin{ex}\label{ex:lau-quad-co}
Consider the quadratic $\bz$-graded commutative Laurent polynomial algebra
$(\Bbbk[\kt, \kt^{-1}],\; \varpi)$ which is given in Example {\rm\ref{ex:lau-quad}}, by
Lemma {\rm\ref{lem:qdual}}, we get a completed cocommutative coalgebra
$(\Bbbk[\kt, \kt^{-1}],\; \bar{\Delta})$, where $\bar{\Delta}$ is given by:
$\bar{\Delta}(\kt^{k})=\sum_{i\in\bz}\kt^{i}\otimes\kt^{k-i}$ for all $k\in\bz$.
This completed cocommutative coalgebra $(\Bbbk[\kt, \kt^{-1}],\; \bar{\Delta})$
is just the Laurent polynomial coalgebra given in Example {\rm\ref{ex:lau-coa}}.
\end{ex}

We now give the notion and results on completed anti-Leibniz bialgebras.

\begin{defi}\label{def:com-dipoi}
A {\rm completed anti-Leibniz bialgebra} is a triple $(A=\oplus_{i\in\bz}A_{i},
\cdot, \hat{\Delta})$ such that $(A=\oplus_{i\in\bz}A_{i}, \cdot)$ is a $\bz$-graded
anti-Leibniz algebra, $(A=\oplus_{i\in\bz}A_{i}, \hat{\Delta})$ is a completed
anti-Leibniz coalgebra, and the following compatibility condition holds:
\begin{align*}
&\qquad\quad\qquad\quad\qquad\quad\qquad\quad
\big(\fr_{A}(a_{1})\,\hat{\otimes}\,\id\big)\hat{\Delta}(a_{2})
-\hat{\tau}\Big(\big(\fr_{A}(a_{2})\,\hat{\otimes}\,\id\big)\hat{\Delta}(a_{1})\Big)=0. \\
&\hat{\Delta}(a_{1}a_{2})+\Big(\big(\id\,\hat{\otimes}\,\fr_{A}(a_{2})
+(\fl_{A}-\fr_{A})(a_{2})\,\hat{\otimes}\,\id\big)
(\id\,\hat{\otimes}\,\id-\hat{\tau})\Big)\hat{\Delta}(a_{1})
+\big(\id\,\hat{\otimes}\,\fl_{A}(a_{1})+\fl_{A}(a_{1})\,\hat{\otimes}\,\id\big)
\hat{\Delta}(a_{2})=0,
\end{align*}
for any $a_{1}, a_{2}\in A$.
\end{defi}

In the definition above, if $A=A_{0}$, the completed anti-Leibniz bialgebra
is just the anti-Leibniz bialgebra given in Definition \ref{def:bialg}.

\begin{thm}\label{thm:aff-antLei}
Let $(A, \cdot, \Delta)$ be a finite-dimensional anti-Leibniz bialgebra and
$(C=\oplus_{i\in\bz}C_{i}$,\; $\varpi)$ be a quadratic $\bz$-graded commutative
associative algebra. Suppose that $(\la, \diamond)$ is the $\bz$-graded
anti-Leibniz algebra from $(A, \cdot)$ by $(C=\oplus_{i\in\bz}C_{i}, \cdot)$ in
Lemma {\rm \ref{lem:aff-anL}}, $\bar{\Delta}: C\rightarrow C\,\hat{\otimes}\,C$
is the linear map defined by Eq. \eqref{ccomu} and $\hat{\Delta}: \la\rightarrow
\la\,\hat{\otimes}\,\la$ is the linear map defined by Eq. \eqref{coalg}.
Then $(\la, \diamond, \hat{\Delta})$ is a completed anti-Leibniz bialgebra,
which is called an {\rm affine anti-Leibniz bialgebra} from $(A, \cdot, \Delta)$
by $(C=\oplus_{i\in\bz}C_{i},\; \varpi)$.

Furthermore, if $A$ is finite-dimensional and $(C=\oplus_{i\in\bz}C_{i},\; \varpi)$
is the quadratic $\bz$-graded Laurent polynomial algebra $\big(\Bbbk[\kt, \kt^{-1}],\;
\varpi\big)$, then $(\la, \diamond, \hat{\Delta})$ is a completed anti-Leibniz
bialgebra constructed as above if and only if $(A, \cdot, \Delta)$ is an
anti-Leibniz bialgebra.
\end{thm}

\begin{proof}
First, since the bilinear form $\varpi(-, -)$ on $C$ is invariant and $(C, \cdot)$ is
an associative algebra, for any $c, c', x, y\in C$,
we have
\begin{align*}
\bar{\varpi}\Big(\sum_{i,j,\alpha}c'_{1,i,\alpha}c\otimes c'_{2,j,\alpha},\; x\otimes y\Big)
&=\bar{\varpi}\Big(\sum_{i,j,\alpha}c'_{1,i,\alpha}\otimes c'_{2,j,\alpha},\;
cx\otimes y\Big)\\[-2mm]
&=\varpi(c',\, cxy)=\varpi(c'c,\; xy)=\bar{\varpi}(\bar{\Delta}(c'c),\; x\otimes y).
\end{align*}
That is, $\bar{\Delta}(c'c)=\sum_{i,j,\alpha}c'_{1,i,\alpha}c\otimes c'_{2,j,\alpha}$.
Note that $(C, \cdot)$ is commutative and $\bar{\Delta}$ is cocommutative, we have
$\bar{\Delta}(c'c)=\bar{\Delta}(cc')=\sum_{i,j,\alpha}c'_{1,i,\alpha}\otimes cc'_{2,j,\alpha}
=\sum_{i,j,\alpha}c_{1,i,\alpha}c'\otimes c_{2,j,\alpha}
=\sum_{i,j,\alpha}cc'_{1,i,\alpha}\otimes c'_{2,j,\alpha}
=\sum_{i,j,\alpha}c'_{1,i,\alpha}\otimes c'_{2,j,\alpha}c
=\sum_{i,j,\alpha}c_{2,j,\alpha}c'\otimes c_{1,i,\alpha}
=\sum_{i,j,\alpha}c_{2,j,\alpha}\otimes c'c_{1,i,\alpha}$.
Thus, for any $a\otimes c$, $a'\otimes c'\in\la$,
\begin{align*}
&\;\big(\fr_{\la}(a\otimes c)\,\hat{\otimes}\,\id\big)\hat{\Delta}(a'\otimes c')
-\hat{\tau}\big(\big(\fr_{\la}(a'\otimes c')\,\hat{\otimes}\,\id\big)\hat{\Delta}
(a\otimes c)\big)\\
=&\;\sum_{(a')}\sum_{i,j,\alpha}(a'_{(1)}a\otimes a'_{(2)})\bullet(c'_{1,i,\alpha}c
\otimes c'_{2,j,\alpha})-\sum_{(a)}\sum_{i,j,\alpha}(a_{(2)}\otimes a_{(1)}a')\bullet
(c_{2,j,\alpha}\otimes c_{1,i,\alpha}c')\\
=&\;\Big((\fr_{A}(a)\otimes\id)\Delta(a')
-\tau\big((\fr_{A}(a')\otimes\id)\Delta(a)\big)\Big)\bullet\bar{\Delta}(cc')\\
=&\;0,
\end{align*}
and
\begin{align*}
&\;\hat{\Delta}((a\otimes c)\diamond(a'\otimes c'))+\Big(\big(\id\,\hat{\otimes}\,
\fr_{\la}(a'\otimes c')+(\fl_{\la}-\fr_{\la})(a'\otimes c')\,\hat{\otimes}\,\id\big)
(\id\,\hat{\otimes}\,\id-\hat{\tau})\Big)\hat{\Delta}(a\otimes c)\\[-2mm]
&\qquad+\Big(\id\,\hat{\otimes}\,\fl_{\la}(a\otimes c)
+\fl_{\la}(a\otimes c)\,\hat{\otimes}\,\id\Big)\hat{\Delta}(a'\otimes c')\\
=&\;\Delta(aa')\bullet\bar{\Delta}(cc')
+\sum_{(a)}\sum_{i,j,\alpha}(a_{(1)}\otimes a_{(2)}a')\bullet(c_{1,i,\alpha}
\otimes c_{2,j,\alpha}c')
+\sum_{(a)}\sum_{i,j,\alpha}(a'a_{(1)}\otimes a_{(2)})\bullet(c'c_{1,i,\alpha}
\otimes c_{2,j,\alpha})\\[-2mm]
&\qquad-\sum_{(a)}\sum_{i,j,\alpha}(a_{(1)}a'\otimes a_{(2)})\bullet(c_{1,i,\alpha}c'
\otimes c_{2,j,\alpha})
-\sum_{(a)}\sum_{i,j,\alpha}(a_{(2)}\otimes a_{(1)}a')\bullet(c_{2,i,\alpha}
\otimes c_{1,j,\alpha}c')\\[-2mm]
&\qquad-\sum_{(a)}\sum_{i,j,\alpha}(a'a_{(2)}\otimes a_{(1)})\bullet(c'c_{2,i,\alpha}
\otimes c_{1,j,\alpha})
+\sum_{(a)}\sum_{i,j,\alpha}(a_{(2)}a'\otimes a_{(1)})\bullet(c_{2,i,\alpha}c'
\otimes c_{1,j,\alpha})\\[-2mm]
&\qquad+\sum_{(a')}\sum_{i,j,\alpha}(a'_{(1)}\otimes aa'_{(2)})\bullet(c'_{1,i,\alpha}
\otimes cc'_{2,j,\alpha})
+\sum_{(a')}\sum_{i,j,\alpha}(aa'_{(1)}\otimes a'_{(2)})\bullet(cc'_{1,i,\alpha}
\otimes c'_{2,j,\alpha})\\
=&\;\Big(\Delta(aa')+\big(\big(\id\otimes\,\fr_{A}(a')-\fr_{A}(a')\otimes\id
+\fl_{A}(a')\otimes\id\big)(\id\otimes\id-\tau)\big)\Delta(a)\Big)
\bullet\bar{\Delta}(cc')\\[-2mm]
&\qquad+\Big(\big(\id\otimes\,\fl_{A}(a)+\fl_{A}(a)\otimes\id\big)
\Delta(a')\Big)\bullet\bar{\Delta}(cc')\\
=&\;0,
\end{align*}
since $(A, \cdot, \Delta)$ is an anti-Leibniz bialgebra.
By the definition, we get that $(\la, \diamond, \hat{\Delta})$ is a completed
anti-Leibniz bialgebra.

Second, if $(\la, \diamond, \hat{\Delta})$ is a completed anti-Leibniz
bialgebra, where $(C=\oplus_{i\in\bz}C_{i},\; \varpi)$ is the quadratic
$\bz$-graded Laurent polynomial algebra $(\Bbbk[\kt, \kt^{-1}]$,\; $\varpi)$,
denote $a\otimes\kt^{k}$ by $a\kt^{k}$ for $a\in A$, then for any $a, a'\in A$ and
$k, l\in\bz$, we have
\begin{align*}
0=&\;\big(\fr_{\la}(a\kt^{k})\,\hat{\otimes}\,\id\big)\hat{\Delta}(a'\kt^{l})
-\hat{\tau}\big(\big(\fr_{\la}(a'\kt^{l})\,\hat{\otimes}\,\id\big)\hat{\Delta}(a\kt^{k})\big)\\
=&\;\Big((\fr_{A}(a)\otimes\id)\Delta(a')-\tau\big((\fr_{A}(a')\otimes\id)\Delta
(a)\big)\Big)\bullet\bar{\Delta}(\kt^{k+l}).
\end{align*}
Thus, we get Eq. \eqref{bialg2} holds. Similarly, we can get Eq. \eqref{bialg1} also holds.
That is, $(A, \cdot, \Delta)$ is an anti-Leibniz bialgebra.
\end{proof}

\begin{ex}\label{ex:comp-bia}
$(I)$ Let $(\Lambda^{2}_{1}, \Delta)$ be the $2$-dimension anti-Leibniz bialgebra given
in Example {\rm \ref{ex:bialg}}. Consider the $\bz$-graded vector space
$\Lambda^{2}_{1}\otimes\Bbbk[\kt, \kt^{-1}]$, we denote the homogeneous element in
$\Lambda^{2}_{1}\otimes\Bbbk[\kt, \kt^{-1}]$ by $a\kt^{k}$, where $a\in\Lambda^{2}$.
Then the multiplication $\diamond$ on $\Lambda^{2}_{1}\otimes\Bbbk[\kt, \kt^{-1}]$
is given by $e_{1}\kt^{i}\diamond e_{1}\kt^{j}=e_{2}\kt^{i+j}$, and the completed
comultiplication $\hat{\Delta}$ on $\Lambda^{2}_{1}\otimes\Bbbk[\kt, \kt^{-1}]$ is
given by $\hat{\Delta}(e_{1}\kt^{k})=\sum_{i\in\bz}e_{2}\kt^{i}\otimes e_{2}\kt^{k-i}$,
for any $i, j, k\in\bz$. This anti-Leibniz bialgebra is also a Leibniz bialgebra,
is also an anti-associative bialgebra (see \cite{HC}).

$(II)$ Let $(A=\Bbbk\{e_{1}, e_{2}, e_{3}\}, \cdot, \Delta)$ be an anti-Leibniz bialgebra
with nonzero multiplication and comultiplication: $e_{2}e_{2}=e_{1}=e_{3}e_{3}$ and
$\Delta(e_{2})=e_{1}\otimes e_{1}$. Then we get an infinite-dimensional anti-Leibniz
bialgebraic structure on $A\otimes\Bbbk[\kt, \kt^{-1}]$ under the multiplication
$e_{2}\kt^{i}\diamond e_{2}\kt^{j}=e_{1}\kt^{i+j}=e_{3}\kt^{i}\ast e_{3}\kt^{j}$ and
completed comultiplication $\hat{\Delta}(e_{2}\kt^{k})=\sum_{i\in\bz}e_{1}\kt^{i}
\otimes e_{1}\kt^{k-i}$, for any $i, j, k\in\bz$.
\end{ex}

\begin{rmk}\label{rmk:aff}
In this section, we using the $\bz$-graded commutative associative algebra to construct
infinite-dimensional anti-Leibniz bialgebra. In fact, we can using the $\bz$-graded perm
algebra to construct infinite-dimensional anti-Leibniz bialgebra. Recently, Lin, Zhou
and Bai construct a $\bz$-graded perm algebra in \cite{LZB} as follows. Let
$P=\{f_{1}\partial_{1}+f_{2}\partial_{2}\mid f_{1}, f_{2}\in\Bbbk[x_{1}^{\pm},
x_{2}^{\pm}]\}$ and define a binary operation $\cdot: P\otimes P\rightarrow P$ by
$$
(x_{1}^{i_{1}}x_{2}^{i_{2}}\partial_{s})\cdot(x_{1}^{j_{1}}x_{2}^{j_{2}}\partial_{t})
=\delta_{s,1}x_{1}^{i_{1}+j_{1}+1}x_{2}^{i_{2}+j_{2}}\partial_{t}
+\delta_{s,2}x_{1}^{i_{1}+j_{1}}x_{2}^{i_{2}+j_{2}+1}\partial_{t},
$$
for any $i_{1}, i_{2}, j_{1}, j_{2}\in\bz$ and $s, t\in\{1, 2\}$.
Then $(P, \cdot)$ is a $\bz$-graded perm algebra with the linear decomposition
$P=\oplus_{i\in\bz}P_{i}$, where
$$
P_{i}=\Big\{\sum_{k=1}^{2}f_{k}\partial_{k}\mid f_{k}\text{ is a homogeneous polynomial with }
\deg(f_{k})=i-1,\, k=1, 2\Big\},
$$
for all $i\in\bz$. If we replace the Laurent polynomial algebra with this $\bz$-graded
perm algebra, the conclusions of this chapter can also be obtained.
\end{rmk}

\bigskip
\noindent
{\bf Acknowledgements. } This work was financially supported by National
Natural Science Foundation of China (No. 11801141).

%

 \end{document}